\documentclass{article}
\usepackage{import}
\usepackage{AGUtils}

\newcounter{subtheorem}
\counterwithin{theoremgrp}{subtheorem}

\newcommand*{\vcenteredhbox}[1]{\begingroup
\setbox0=\hbox{#1}\parbox{\wd0}{\box0}\endgroup}

\numberwithin{equation}{section}

\title{Quadrature Domains and the Faber Transform}

\author{Andrew J. Graven and Nikolai G. Makarov}

\date{\vspace{-4ex}}

\begin{document}

\maketitle

%\noindent\rule{\textwidth}{.5pt}
\begin{abstract}
We present a framework for reconstructing any simply connected, bounded or unbounded, quadrature domain $\Omega$ from its quadrature function $h$. Using the Faber transform, we derive formulae directly relating $h$ to the Riemann map for $\Omega$. Through this approach, we obtain a complete classification of one point quadrature domains with complex charge. We proceed to develop a theory of weighted quadrature domains with respect to weights of the form $\rho_a(w)=|w|^{2(a-1)}$ when $a > 0$ (``power-weighted'' quadrature domains) and the limiting case of when $a=0$ (``log-weighted'' quadrature domains). Furthermore, we obtain Faber transform formulae for reconstructing weighted quadrature domains from their respective quadrature functions. Several examples are presented throughout to illustrate this approach both in the simply connected setting and in the presence of rotational symmetry. 
\end{abstract}

\section{Introduction}\label{sec:Introduction}
The purpose of this article is to construct and analyze certain classes of quadrature domains (QDs) using the Faber transform. In particular, we will provide explicit representations for the Riemann map associated to simply connected quadrature domains and weighted quadrature domains. We will also cover several applications of this approach to questions of the existence and uniqueness for QDs and weighted QDs.
%While classical QDs have been extensively studied, their weighted generalizations have received little attention in the literature (cf. \cite{DragnevLeggSaff,Skinner_2015,GustafssonRiemannSurfaces}).

Quadrature domains are an important class of planar domains, characterized by their admission of exact quadrature identities. They arise naturally across a wide array of topics in analysis and mathematical physics, including moment problems, approximation theory, random matrix theory, potential theory, Hele-Shaw flow, the Schwarz reflection, and conformal dynamics \cite{Richardson_1972,Putinar2002,Mineev_Weinstein_2008,GustafssonQDsInvProb_1990,DropOrderFour,davis_1974,LEE_2018}. Several classes of quadrature domains have been known for over a century, \cite{CNeumann} however the work of systematizing these didn't begin en force until the early 1970s, with the work of Richardson, Davis, and Aharonov \& Shapiro \cite{Richardson_1972,davis_1974,AharonovShapiro}.

We shall begin with what is undoubtedly the most well-known example of a quadrature domain: the disk. Recall the mean value property for functions analytic in the disk asserts that
\begin{equation}\label{eqn:ClassicMVP}
    \int_{\D_{r}(w_0)}fdA=r^2f(w_0),
\end{equation}
where $dA=\frac{dxdy}{\pi}$. Far deeper however is the converse statement: disks are the {\it only} finitely connected bounded domains which satisfy the mean value property \cite{AharonovShapiro}. Addressing this question of uniqueness for quadrature domains more generally will be central to our discussion.\\

The paper is organized as follows: The first two sections are concerned with the theory of classical quadrature domains. In particular, in Section \ref{sec:Introduction} we review the basic theory of quadrature domains, introduce the Faber transform, and cover several Faber transform formulae which relate QDs to their Riemann map. In Section \ref{sec:OnePtQDClass}, we apply the results of the prior section to obtain a complete classification of one point quadrature domains, including unbounded domains.

The remaining two sections of the paper are concerned with weighted quadrature domains and extending our results for classical QDs to the weighted setting. In Section \ref{sec:PQDs}, we define and establish various properties of ``power-weighted'' quadrature domains: domains admitting a quadrature identity in which the integral over the domain is taken with respect to the weight $\rho_a(w)=|w|^{2(a-1)}$ for some $a>0$. A key observation, necessary for much of the analysis in this section is that a simply connected domain is a power-weighted QD iff its outer part is the $a$th root of a rational function (Theorem \ref{thm:SCPQDCharacterization}). From this, we conclude that if $\varphi$ is the Riemann map associated to a power-weighted QD, then it admits a representation of the form
$$\varphi(z)=\varphi_{\rm in}(z)r^{\#}(z)^{\frac{1}{a}},$$
where both $\varphi_{\rm in}$ and $r$ are rational functions. We proceed by deriving formulae which relate the Riemann map of a power-weighted QD to its quadrature function, and we use these formulae to characterize several families of power-weighted QDs.

Finally, in Section \ref{sec:LQDs}, we define and establish various properties of ``log-weighted'' quadrature domains: domains admitting a quadrature identity in which the integral over the domain is taken with respect to the weight $\rho_0(w)=|w|^{-2}$. Similar to the situation for power-weighted QDs, we observe that if $\varphi$ is the Riemann map associated to a log-weighted QD, then it admits a representation of the form
$$\varphi(z)=\varphi_{\rm in}(z)e^{r^{\#}(z)},$$
where both $\varphi_{\rm in}$ and $r$ are rational functions (Theorem \ref{theorem:logweightedAQDRiemannMaps}). After fixing some conventions and general notation, we will start with a discussion of classical quadrature domains.\\

{\it General notation:} A domain $\Omega$ is an open connected subset of the Riemann sphere $\Ch$. $\text{Cl}(\Omega)$ denotes the closure, $\Omega^c$ the complement in $\Ch$, $\Omega\IntComp:=\text{Cl}(\Omega)^c$, and $f\dEquals g$ denotes equality of $f$ and $g$ on $\partial\Omega$ (absent risk of ambiguity). Finally, $\D_r(a)$ denotes the disk of radius $r$ centered at $a$, $\D_r:=\D_r(0)$, and $\overline{(\cdot)}$ denotes complex conjugation. 

$\A(\Omega)$ denotes the space of functions analytic in $\Omega\subseteq\Ch$, $\A_0(\Omega)$ those which are $0$ at $\infty$, $\M(\Omega)=$ the space of functions meromorphic in $\Omega$, $\Rat(\Omega)$ the space of rational functions with poles only in $\Omega$, and $\Rat_0(\Omega)$ those which are $0$ at $\infty$. We also impose the requirement that $f$ extends continuously to $\partial\Omega$ in each of these function spaces. Contour integrals are normalized such that the standard factor of $2\pi i$ is suppressed. E.g when $\Gamma$ is piecewise $C^1$,
$$\oint_{\Gamma}f(w)dw:=\dfrac{1}{2\pi i}\int_{0}^{1}f\circ\gamma (t)\gamma'(t)dt.$$

\subsection{Bounded Quadrature Domains}\label{subsec:ClassicBQDs}
We refer to a bounded domain $\Omega\subset\C$ as a classical bounded quadrature domain with respect to the analytic test class $\mathcal{F}$ if there exists a finite collection of $c_k\in\C$, $a_k\in\Omega$, and $m_k\in\N_{0}$ for which
\begin{equation}\label{eqn:ClassicPtQID}
\int_{\Omega}fdA=\sum_{k}c_kf^{(m_k)}(a_k)
\end{equation}
for all $f\in\mathcal{F}$. This is equivalent, via the residue theorem, to the existence of $h\in\Rat(\Omega)$ for which
\begin{equation}\label{eqn:ClassicContourQID}
\int_{\Omega}fdA=\oint_{\partial\Omega}f(w)h(w)dw
\end{equation}
for all $f\in\mathcal{F}$. In particular, taking $h(w)=\sum_{k}\frac{c_km_k!}{(w-a_k)^{m_k+1}}$ reproduces Equation \ref{eqn:ClassicPtQID}. Such an $h$ is referred to as a {\it quadrature function} for $\Omega$. $h$ may or may not be unique when it exists - this depends crucially on the choice of test class. For example, $\mathcal{F}=\A(\Omega)$ then $\Omega$ has a unique quadrature function $h\in\Rat_0(\Omega)$, provided that one exists. This motivates the following definition.
\begin{definition}[Bounded quadrature domain]\label{def:BQD}
A bounded domain $\Omega\subseteq\C$, equal to the interior of its closure, is called a {\it bounded quadrature domain} (BQD) if there exists an $h\in\Rat_0(\Omega)$ such that Equation \ref{eqn:ClassicContourQID} holds for all $f\in\A(\Omega)$. This is denoted by $\Omega\in\QD(h)$.
\end{definition}

The simplest non-trivial example of a bounded QD is the cardioid $\Omega=\left\{z+\frac{z^2}{2}:z\in\D\right\}$, which admits a quadrature identity involving both $f$ and its first derivative:
\begin{equation}\label{eqn:CardioidQI}
    \int_{\Omega}fdA=\dfrac{3}{2}f(0)+\dfrac{1}{2}f'(0).
\end{equation}

The formula just below Equation \ref{eqn:ClassicContourQID} combined with the quadrature identity (\ref{eqn:CardioidQI}) suggests that $\Omega\in\QD\left(h\right)$ for $h(w)=\frac{3}{2}w^{-1}+\frac{1}{2}w^{-2}$. This is readily verified via the residue theorem. See \ref{subsubsec:FaberCardioid} for a succinct derivation of this QI using the associated Riemann map.\\
\begin{figure}[ht]
  \centering
\includegraphics[height=0.19\linewidth,width=1\linewidth]{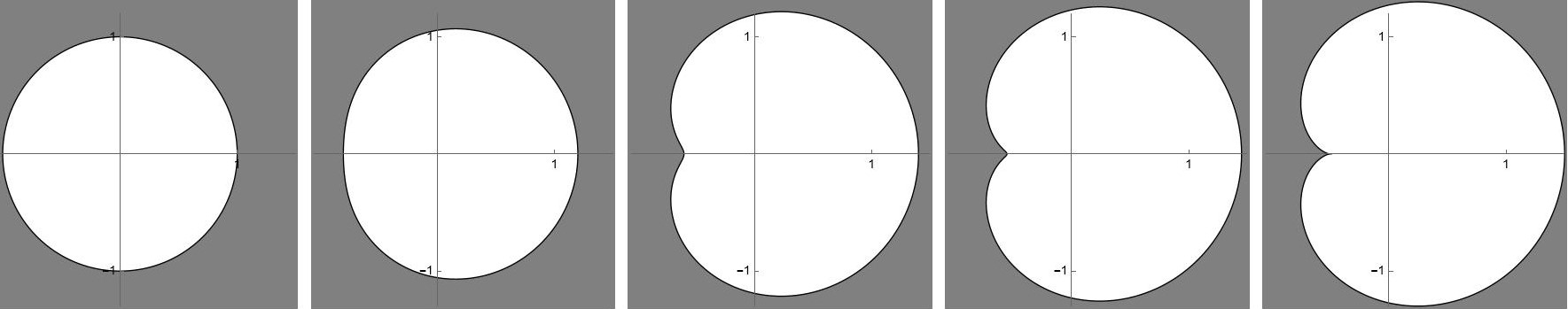}
    \caption{$\Omega\in\QD\left(w^{-1}+\alpha_1w^{-2}\right)$ (the complement of the shaded region) for $\alpha_1\in\{0,.2,.4,.45,.5\}$.}\vspace{.5em}\label{fig:BinomialCardioidQDFamily}
\end{figure}
\FloatBarrier

\subsection{Unbounded Quadrature Domains}\label{subsec:ClassicUQDs}
Parallel to the theory of bounded quadrature domains, is that of unbounded quadrature domains. It turns out that Equation \ref{eqn:ClassicContourQID}, where the area integral is understood in the sense of its {\it Cauchy principal value}:
$$\int_{\Omega}fdA:=\lim_{r\to\infty}\int_{\Omega\cap\D_r}fdA,$$
is a much more natural starting point in this case. A QI in the sense of Equation \ref{eqn:ClassicPtQID} still exists, but includes a linear combination of the Laurent coefficients of $f$ at $\infty$.

More care is required when it comes to the choice of test class $\mathcal{F}$ for unbounded quadrature domains. On the one hand, $\A(\Omega)$ cannot be used because it isn't integrable; on the other hand, the integrable test class $L_a^1(\Omega)=L^1(\Omega)\cap\A(\Omega)$ is too small to ensure uniqueness of the quadrature function: If $f\in L^1(\Omega)$, then $f\in O(w^{-3})$ about $\infty$. Hence if $h\in\Rat(\Omega)$ is a quadrature function for $\Omega$, then so is $h+\alpha$ for each $\alpha\in\C$ because
$$\oint_{\partial\Omega}f(w)(h(w)+\alpha)dw=\oint_{\partial\Omega}f(w)h(w)dw+\alpha\Res{\infty}(f)=\oint_{\partial\Omega}f(w)h(w)dw.$$

We can recover uniqueness in this sense, however, if we consider the intermediate test class $\mathcal{F}=\A_0(\Omega)$ ($L^1(\Omega)\cap\A(\Omega)\subset\A_0(\Omega)\subset \A(\Omega)$). In particular, if $\Omega$ is an unbounded domain for which there exists an $h\in\Rat(\Omega)$ satisfying Equation \ref{eqn:ClassicContourQID} for all $f\in\A_0(\Omega)$, then $h$ is unique (apply the quadrature identity to the Cauchy kernel $\xi\mapsto(\xi-w)^{-1}$). This motivates the following definition.

\begin{definition}[Unbounded quadrature domain]\label{def:UQD}
An unbounded domain $\Omega\subseteq\Ch$, equal to the interior of its closure, is called an {\it unbounded quadrature domain} (UQD) if there exists an $h\in\Rat(\Omega)$ such that Equation \ref{eqn:ClassicContourQID} holds for all $f\in\A_0(\Omega)$. This is denoted by $\Omega\in\QD(h)$.
\end{definition}

\begin{example}[The ellipse]\label{ex:ellipse}
If $\Omega\in\QD\left(\alpha w\right)$ is unbounded and simply connected and $|\alpha|<1$, then $\varphi(z)=c(z+\overline{\alpha}z^{-1})$ is a Riemann map for $\Omega$ (where $c$ is the conformal radius of $\Omega$). In particular, $\Omega$ is the complement of an ellipse centered at the origin of eccentricity $\frac{2\sqrt{|\alpha|}}{|\alpha|+1}$.\\
\begin{figure}[ht]
  \centering
    \includegraphics[scale=.55]{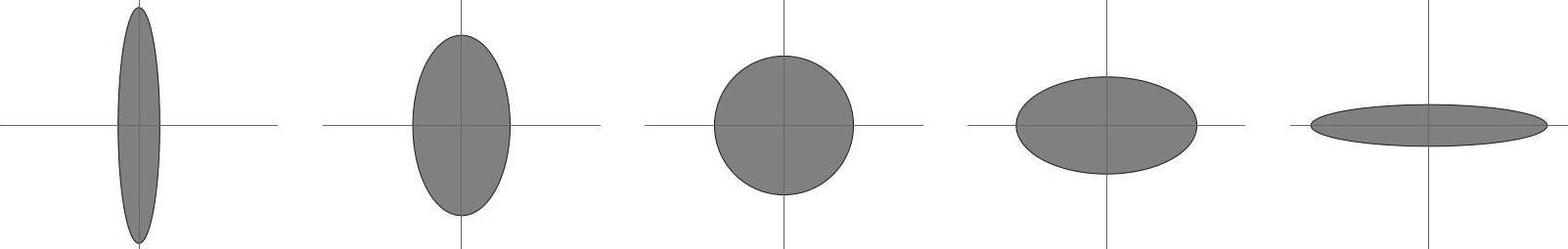}
    \caption{$\Omega\in\QD(\alpha w)$ (the complement of the shaded region) for $\alpha \in\{-.7,-.3,0,.3,.7\}$.}\vspace{1.2em}\label{fig:UQDEllipseEx1}
\end{figure}
\FloatBarrier
In particular, if $f\in\A_0(\Omega)$, then
$$\int_{\Omega}fdA=\alpha\oint_{\partial\Omega}f(w)wdw=-\alpha f_2,$$
where $f_2$ is the second Laurent coefficient of $f$ at $\infty$. The above example also raises the question of the existence of quadrature domains with a given quadrature function. One can show that no such domain exists when $|\alpha|\geq1$ for instance.
\end{example}

\subsection{Quadrature Domains}\label{subsec:ClassicQDs}

We refer to UQDs and BQDs collectively as quadrature domains (QDs), characterized by the existence of a quadrature function. In particular,
\begin{definition}[Quadrature domain]\label{def:QD}
A domain $\Omega\subseteq\Ch$ is called a {\it quadrature domain} if there exists a rational function $h$ such that $\Omega\in\QD(h)$. This is denoted by $\Omega\in\QD$. We furthermore denote by $d_\Omega:=\deg(h)$ the {\it order} of $\Omega$.
\end{definition}

As previously discussed, the uniqueness of the quadrature function associated to a given quadrature domain is well established. However, the inverse problem of existence and uniqueness for quadrature domains associated to a given quadrature function isn't so straightforward. Non-uniqueness has been exhibited for multiply connected regions: consider the disk and the complement of an annulus with the same inner radius (Varchenko \& Etingof (\cite{DropOrderFour}, \S 2.3) provide another example). On the other hand, it is conjectured that simply connected QDs are indeed uniquely associated to their quadrature function (modulo conformal radius). Aharonov \& Shapiro (1976) established that disks are the unique one point QDs among all bounded finitely connected domains \cite{AharonovShapiro}. They established in the same paper a similar result for the cardioid. See \cite{Ameur_2021}, Theorem 1.1, for a discussion of additional known uniqueness results.

\subsubsection{The Schwarz Function}

A useful tool for approaching these questions arises from a characteristic property of quadrature domains: they admit a {\it Schwarz function} $S$. In particular, a domain $\Omega$ is a QD if and only if there exists a function $S\in\M(\Omega)$ such that $S(w)\dEquals\overline{w}$ (recall $\dEquals$ denotes equality on $\partial\Omega$). $S$ is unique when it exists and admits a unique representation is in terms of the quadrature function and the {\it Cauchy transform} of $\Omega\IntComp$, $C^{\Omega\IntComp}$. The Schwarz function associated to a quadrature domain $\Omega\in\QD(h)$ can be expressed as \cite{davis_1974}
\begin{equation}\label{eqn:QuadCoincidence}
S(w)=C^{\Omega\IntComp}(w)+h(w)\dEquals\overline{w}.
\end{equation}
$C^{U}$ is defined for measurable $U\subset\Ch$ with compact boundary and is given by
\begin{equation}
C^{U}(w)=\int_{U}\dfrac{dA(\xi)}{w-\xi}.
\end{equation}

A handful of equivalent characterizations of quadrature domains, tied together by the Schwarz function equation \ref{eqn:QuadCoincidence}, are given in Lemma \ref{lemma:QDChars}.
\begin{lemma}\label{lemma:QDChars}
Let $\Omega\subset\Ch$ be a domain equal to the interior of its closure. Then the following are equivalent
\begin{enumerate}
    \item there exists $h\in\Rat(\Omega)$ for which $\Omega\in\QD(h)$,
    \item there exists $h\in\Rat(\Omega)$ for which $\left.C^{\Omega}\right\vert_{\Omega\IntComp}=h$,
    \item there exists $S\in\M(\Omega)$ for which $S(w)\dEquals\overline{w}$.
\end{enumerate}
\end{lemma}
\noindent The proof of this equivalence is standard \cite{avci1977quadrature,Lee_2015}.\\

The statement $\left.C^{\Omega}\right\vert_{\Omega\IntComp}=h$ has a natural physical interpretation. $\overline{C^{\Omega}}$ corresponds to the logarithmic electrostatic field due to a uniform charge distribution on $\Omega$, whereas $\overline{h}$ corresponds to the electrostatic field due to a collection of point charges (and multipoles) at the poles of $h$. Thus the above equality tells us that the electrostatic field in $\Omega\IntComp$ due to a uniform charge distribution on $\Omega$ is equal to that of a finite collection of point charges (and multipoles) in $\Omega$. Figure \ref{fig:QDPotentialInterp} illustrates this phenomenon in the case of the cardioid.
\begin{figure}[ht]
  \centering
    \includegraphics[width=.9\linewidth]{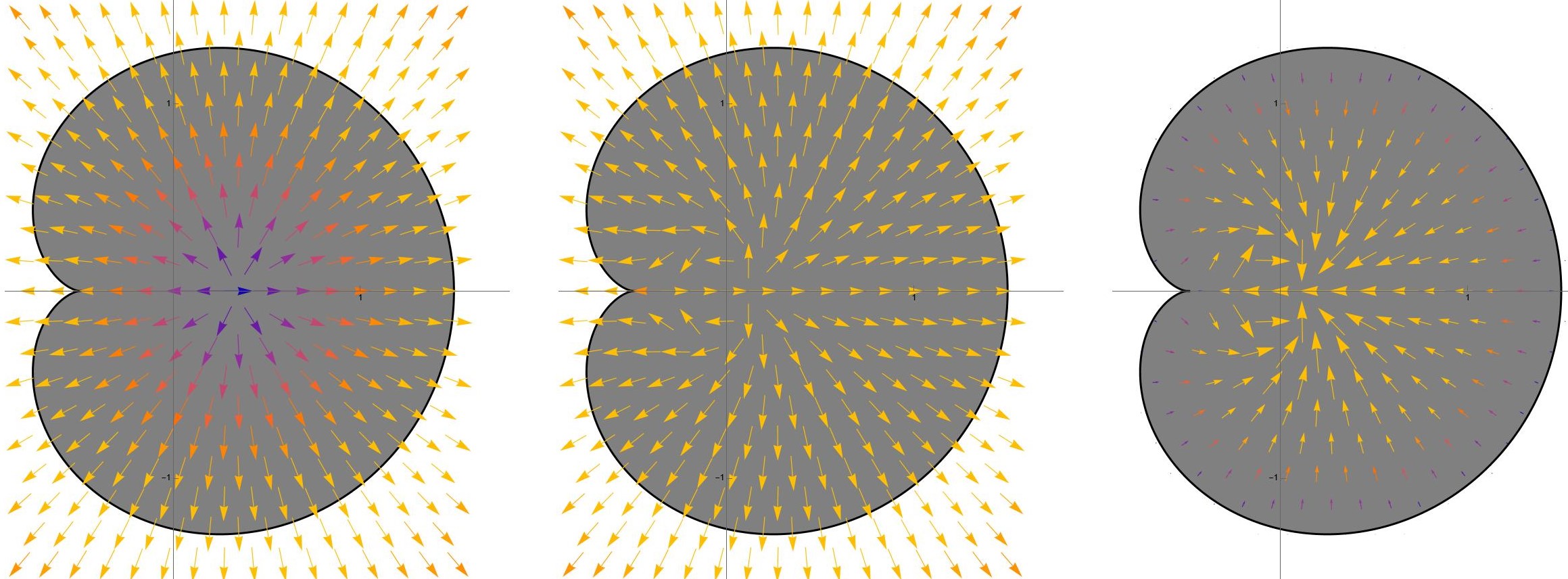}
    \caption{Illustration of the potential-theoretic interpretation of quadrature domains for the cardioid, $\Omega\in\QD\left(\frac{1}{w}+\frac{1}{2w^2}\right)$. $\overline{C^{\Omega}}$ (left), $\overline{h}$ (center), $\overline{C^{\Omega}}-\overline{h}$ (right). }\vspace{1.2em}\label{fig:QDPotentialInterp}
\end{figure}
\FloatBarrier

\subsubsection{Boundary Regularity of QDs}\label{subsubsec:QDBoundaryRegularity}
It turns out that the regularity of the boundary of QDs is essentially determined by the associated Schwarz function. This is a direct consequence of the celebrated Sakai regularity theorem (\cite{SakaiRegularity} Theorem 5.2 and \cite{Lee_2015} \S3.2). Fix a domain $\Omega$ and $w_0\in\partial\Omega$. $w_0$ is called a {\it regular} point if there exists $\epsilon>0$ for which $\Omega\cap\D_{\epsilon}(w_0)$ is a Jordan domain and $\partial\Omega\cap\D_{\epsilon}(w_0)$ is a simple real analytic arc. A point which is not regular is called a {\it singular} point.

\begin{theorem}\label{theorem:SakaiRegularity}
    If $S$ is a {\it local Schwarz function} at a singular point $w_0$, then $w_0$ is either a (conformal) {\it cusp}, a {\it double point}, or a {\it degenerate point}.
\end{theorem}

A function $S:\Cl(\Omega)\cap\D_{\epsilon}(w_0)\rightarrow\C$ is {\it local Schwarz function} at $w_0\in\partial\Omega$ if $S\in\A(\Omega\cap\D_{\epsilon}(w_0))$ and $S(w)=\overline{w}$ on $\partial\Omega\cap\D_{\epsilon}(w_0)$. A {\it degenerate point} is a point $w_0$ for which $\Omega^c\cap\D_{\epsilon}(w_0)$ is a proper subset of a real analytic curve. See \cite{Lee_2015}, \S3.2 for the definitions of cusps and double points.

As we require quadrature domains to equal to the interior of their closure, we may neglect the degenerate case. Combining this with the fact that a Schwarz function restricts to a local Schwarz function about every boundary point, we obtain the following regularity theorem for QDs,
\begin{theorem}\label{thm:QDBoundaryRegularity}
    If $\Omega$ is a quadrature domain then $\partial\Omega$ has finitely many singular points, each of which is either a cusp or a double point.
\end{theorem}

\subsection{The Inverse and Direct Problems for Simply Connected QDs}\label{subsec:SCQDs}
A central problem in the theory of quadrature domains is that of reconstructing a QD given its quadrature function. This is referred to as the {\it inverse problem} for quadrature domains. Complementary to this is the {\it direct problem}, which is concerned with determining the quadrature function associated to a given quadrature domain.

For simply connected domains, these problems may be reduced to relating the quadrature function of the domain and its Riemann map. A classic result in the theory of quadrature domains provides a characterization of simply connected QDs in terms of the associated Riemann map.

\begin{theorem}[\cite{AharonovShapiro}]\label{thm:QDIffRationalRiemann}
    A simply connected domain is a quadrature domain iff it has a rational Riemann map.
\end{theorem}

Hence, for simply connected QDs, the inverse and direct problems reduce to the determination of the relationship between the finitely many poles and coefficients of the quadrature function and the finitely many poles and coefficients of the Riemann map. Formulae relating the Riemann map of a given simply connected QD with its quadrature identity have appeared in various forms going as far back as Davis \cite{davis_1974} (1974) in the case of bounded classical QDs. More recently, Ameur, Helmer, and Tellander (2021) apply a form of this identity (\cite{Ameur_2021} Equation 2), referred to as the ``master formula'', to the uniqueness problems for bounded classical QDs. This ``master formula'' represents the Riemann map in terms of the quadrature measure $\mu=\overline{\partial}h$:
\begin{equation}
\varphi(z)=\overline{\varphi_{\ast}\mu}\left(\dfrac{z}{\overline{\varphi'(\lambda)}(1-z\overline{\lambda})}\right),
\end{equation}
where $\varphi_{\ast}\mu$ is the pushforward measure.

Chang \cite{Chang_2013} considers an essentially equivalent formula in terms of the {\it Faber transform}, an isomorphism of analytic functions on the disk with those on another bounded simply connected domain. We provide generalizations of this approach to unbounded domains (\ref{thm:ClassicUQDRationalRiemannIffQD}) and some classes of weighted domains (\S\ref{sec:PQDs}, \S\ref{sec:LQDs}). After fixing some conventions and notation regarding the Riemann map, we will proceed with a brief proof of Theorem \ref{thm:QDIffRationalRiemann}, followed by an exposition of the method for classical BQDs and UQDs.\\

Let $\Omega$ be a bounded simply connected domain. The Riemann mapping theorem asserts that $\Omega$ is conformally equivalent to the disk, and that the associated conformal map is unique up to a rotation and recentering. In particular, for each $w_0\in\Omega$ there exists a unique biholomorphism $\varphi:\D\rightarrow\Omega$ for which $\varphi(0)=w_0$ and $\varphi'(0)>0$, so that
$$\varphi(z)=w_0+\varphi'(0)z+O(z^2)$$
We refer to $\varphi$ as {\it the} Riemann map associated to $\Omega$ and $w_0$. On the other hand, when $\Omega$ is unbounded, there exists a unique biholomorphism $\varphi:\D\IntComp\rightarrow\Omega$ for which $\varphi(\infty)=\infty$ and $\varphi'(\infty)>0$, so that
$$\varphi(z)=cz+f_0+\left<z^{-1}\right>.$$
In this case, too, we refer to $\varphi$ as {\it the} Riemann map associated to $\Omega$.

\subsection{Solving the Inverse and Direct Problems via the Faber Transform}
A central goal of this paper is to demonstrate a method of reconstructing the Riemann map associated to a simply connected QD from its quadrature function and vice-versa (i.e. to solve the inverse and direct problems for simply connected QDs). Theorems \ref{thm:ClassicBQDRationalRiemannIffQD} and \ref{thm:ClassicUQDRationalRiemannIffQD} provide characterizations of this method for bounded and unbounded domains respectively. We denote by $\Phi_{\varphi}$ the {\it Faber transform} associated to the Riemann map $\varphi$.

\begin{theorem}\label{thm:ClassicBQDRationalRiemannIffQD}
Let $\Omega$ be a bounded and simply connected domain with quadrature function $h\in\Rat_0(\Omega)$ and Riemann map $\varphi:\D\rightarrow\Omega$. In this case,
\begin{equation}\label{eqn:BoundedFaberTransformHFormula}
h=\Phi_{\varphi}\left(\AnalyticIn{\varphi^{\#}}{\D\IntComp}\right)
\end{equation}
and
\begin{equation}\label{eqn:BoundedFaberTransformPhiFormula}
\varphi=\varphi(0)+\Phi_{\varphi}^{-1}(h)^{\#}.
\end{equation}
\end{theorem}
In particular, Equations \ref{eqn:BoundedFaberTransformHFormula} and \ref{eqn:BoundedFaberTransformPhiFormula} respectively solve the direct and inverse problems for BQDs. Note that $\AnalyticInNoBracket{\varphi^{\#}}{\D\IntComp}=\varphi^{\#}-\overline{\varphi(0)}$, so $h=\Phi_{\varphi}\left(\varphi^{\#}-\overline{\varphi(0)}\right)$ is an equivalent form of Equation \ref{eqn:BoundedFaberTransformHFormula}. Theorem \ref{thm:ClassicUQDRationalRiemannIffQD} covers the analogous result for UQDs.
\begin{theorem}\label{thm:ClassicUQDRationalRiemannIffQD}
Let $\Omega$ be an unbounded and simply connected domain with quadrature function $h\in\Rat(\Omega)$ and Riemann map $\varphi:\D\IntComp\rightarrow\Omega$. In this case,
\begin{equation}\label{eqn:UnboundedFaberTransformHFormula}
h=\Phi_{\varphi}\left(\AnalyticIn{\varphi^{\#}}{\D}\right)
\end{equation}
and
\begin{equation}\label{eqn:UnboundedFaberTransformPhiFormula}
\varphi(z)=cz+\Phi_{\varphi}^{-1}(h)^{\#}(z).
\end{equation}
\end{theorem}

To make sense of these formulae, a formal definition of the Faber transform is needed. This, along with a discussion of the various properties of the Faber transform is the topic of the next section.

\subsubsection{The Faber Transform}\label{subsubsec:FTMethod}
Recall that if $\Omega$ is a bounded simply connected domain and $w_0\in\Omega$, then by ``the Riemann map associated to $\Omega$ and $w_0$'' we refer to the unique biholomorphism $\varphi:\D\rightarrow\Omega$ for which $\varphi'(0)>0$ and $\varphi(0)=w_0$. We furthermore denote by $\psi$ the inverse of $\varphi$.

If $\Omega$ is a domain with a rectifiable boundary on which a function $f$ integrable, then we denote by $\AnalyticIn{f}{\Omega}$ the {\it Cauchy projection} of $f$ onto $\Omega$,
\begin{equation}\label{eqn:analyticprojection}
\AnalyticIn{f}{\Omega}(w):=\oint_{\partial\Omega}\dfrac{f(\xi)}{\xi-w}d\xi.
\end{equation}
It follows from the residue theorem that if $f\in\M(\Omega\IntComp)$, then $\AnalyticIn{f}{\Omega}$ is a rational function with the same poles as $f$. The Cauchy projection, $\AnalyticInNoBracket{f}{\Omega}$, is also referred to as the ``analytic part'' of $f$ in $\Omega$.

\begin{comment}
\begin{lemma}\label{lemma:AnalyticDecompositionLemma}
If $\Omega$ is a bounded domain with piecewise $C^1$ boundary, then $C^0(\partial\Omega)=\A(\Omega)\oplus\A_0(\Omega\IntComp)$ and for each $f\in C^0(\partial\Omega)$,
\begin{enumerate}
    \item $f$ decomposes uniquely in $\A(\Omega)\oplus\A_0(\Omega\IntComp)$ as $f\dEquals\AnalyticInNoBracket{f}{\Omega}+\AnalyticInNoBracket{f}{\Omega\IntComp}$,
    \item $\AnalyticIn{\AnalyticInNoBracket{f}{\Omega}}{\Omega}=\AnalyticInNoBracket{f}{\Omega},\tab\AnalyticIn{\AnalyticInNoBracket{f}{\Omega\IntComp}}{\Omega\IntComp}=\AnalyticInNoBracket{f}{\Omega\IntComp},\tab\AnalyticIn{\AnalyticInNoBracket{f}{\Omega}}{\Omega\IntComp}=\AnalyticIn{\AnalyticInNoBracket{f}{\Omega\IntComp}}{\Omega}=0$,
    \item If $f\in\M(\Omega)$ then $\AnalyticInNoBracket{f}{\Omega\IntComp}\in\Rat(\Omega)$.
\end{enumerate}
\end{lemma}
The result naturally extends to unbounded $\Omega$ via the identity $(\Omega\IntComp)\IntComp=\Omega$. See Theorem 2.3 of \cite{lenells2014matrixriemannhilbertproblemsjumps} for details.\footnote{\#3 follows from Mittag-Leffler's theorem, which implies there exists $r\in\Rat(\Omega)$ for which $f-r\in\A(\Omega)$.}
\end{comment}

\begin{definition}[Interior Faber transform]
The {\it interior Faber transform} associated to $\Omega$ and $\varphi$ (as above) is an isomorphism $\Phi_{\varphi}:\A_0(\D\IntComp)\rightarrow\A_0(\Omega\IntComp)$ given by
\begin{equation}\label{eqn:InteriorTransformFormula}
\Phi_{\varphi}(f)(w):=\AnalyticIn{f\circ\psi}{\Omega\IntComp}(w)=\oint_{\partial\Omega\IntComp}\dfrac{f\circ\psi(\xi)}{\xi-w}d\xi.
\end{equation}
\end{definition}
If $f\in C^0(\partial\D)$ then the Sokhotski-Plemelj theorem implies $f=\AnalyticIn{f}{\D}+\AnalyticIn{f}{\D\IntComp}$, so
\begin{equation}\label{eqn:InteriorFaberTransformProjectionExtension}
    \AnalyticIn{f\circ\psi}{\Omega\IntComp}=\AnalyticIn{\AnalyticIn{f}{\D}\circ\psi+\AnalyticIn{f}{\D\IntComp}\circ\psi}{\Omega\IntComp}=\AnalyticIn{\AnalyticIn{f}{\D\IntComp}\circ\psi}{\Omega\IntComp}=\Phi_\varphi\left(\AnalyticInNoBracket{f}{\D\IntComp}\right).
\end{equation}
The analogous result holds for the inverse transform, $\AnalyticIn{f\circ\varphi}{\D\IntComp}=\Phi_{\varphi}^{-1}\left(\AnalyticInNoBracket{f}{\Omega\IntComp}\right)$, under the additional assumption that $\partial\Omega$ is piecewise $C^1$.\\

If the domain under consideration is unbounded, then it has an associated {\it exterior Faber transform}. Recall that if $\Omega$ is an unbounded simply connected domain, then by ``the Riemann map associated to $\Omega$'' we refer to the unique biholomorphism $\varphi:\D\IntComp\rightarrow\Omega$ for which $\varphi(\infty)=\infty$ and $\varphi'(\infty)=c>0$, where $c>0$ is the conformal radius of $\Omega$. 
\begin{definition}[Exterior Faber transform]
The exterior Faber transform associated to $\Omega$ and $\varphi$ (as above) is an isomorphism $\Phi_{\varphi}:\A(\D)\rightarrow\A(\Omega\IntComp)$ defined precisely as in Equation \ref{eqn:InteriorTransformFormula}. The inverse transform is given by $\Phi_{\varphi}^{-1}(f)=\AnalyticIn{f\circ\varphi}{\D}$.
\end{definition}
Moreover note that because $\partial\D$ is $C^1$, the Sokhotski-Plemelj theorem implies that for each $f\in C^0(\partial\D)$,
\begin{equation}\label{eqn:ExteriorFaberTransformProjectionExtension}
    \AnalyticIn{f\circ\psi}{\Omega\IntComp}=\AnalyticIn{\AnalyticIn{f}{\D}\circ\psi+\AnalyticIn{f}{\D\IntComp}\circ\psi}{\Omega\IntComp}=\AnalyticIn{\AnalyticIn{f}{\D}\circ\psi}{\Omega\IntComp}=\Phi_\varphi\left(\AnalyticInNoBracket{f}{\D}\right).
\end{equation}
The analogous result holds for the inverse transform, $\AnalyticIn{f\circ\varphi}{\D}=\Phi_{\varphi}^{-1}\left(\AnalyticInNoBracket{f}{\Omega\IntComp}\right)$, under the additional assumption that $\partial\Omega$ is piecewise $C^1$.

Note that, in contrast to the interior case, $\C[z]\subseteq\A(\D)$ so one can consider the Faber transform of polynomials. The {\it Faber polynomials} associated to $\Omega$ (or $\varphi$) are the polynomial basis $\{F_n\}_{n=0}^{\infty}$ given by $F_n=\Phi_{\varphi}(z^n)=\AnalyticIn{\psi^n}{\Omega\IntComp}$. The inverse Faber polynomials, $\{W_n\}_{n=0}^{\infty}$ are defined analogously. $F_n$ and $W_n$ are given by the polynomial parts of $\psi^n$ and $\varphi^n$ respectively. In fact the exterior transform is an automorphism of $\C[z]$ and an isomorphism $\Rat(\D\IntComp)\leftrightarrow\Rat(\Omega)$.

We are now prepared to prove Theorem \ref{thm:ClassicUQDRationalRiemannIffQD}. We omit the proof of Theorem \ref{thm:ClassicBQDRationalRiemannIffQD} because it is largely analogous (the result can also be obtained from Theorem \ref{thm:GenBPQDLefflerFTFormula}).

\begin{proof}[Proof of Theorem \ref{thm:ClassicUQDRationalRiemannIffQD}]

If $\Omega\in\QD(h)$ with Riemann map $\varphi:\D\IntComp\rightarrow\Omega$, then $\varphi(z)=cz+f_0+\left<z^{-1}\right>$, which implies $\varphi^{\#}(z)-cz^{-1}\in\A(\D)$. Hence,
$$\varphi^{\#}(z)-cz^{-1}=\AnalyticIn{\varphi^{\#}(z)-cz^{-1}}{\D}=\AnalyticIn{\varphi^{\#}(z)}{\D}-c \AnalyticIn{z^{-1}}{\D}=\AnalyticIn{\overline{\varphi(z)}}{\D}.$$
as, $\AnalyticIn{z^{-1}}{\D}=0$. Equation \ref{eqn:QuadCoincidence} tells us that the Schwarz function $S(w)\dEquals\overline{w}$ is given by $S=C^{\Omega\IntComp}+h$, so
\begin{align*}
    \varphi^{\#}(z)-cz^{-1}&=\AnalyticIn{\overline{\varphi(z)}}{\D}=\AnalyticIn{S\circ\varphi(z)}{\D}=\AnalyticIn{(C^{\Omega\IntComp}+h)\circ\varphi(z)}{\D}=\AnalyticIn{C^{\Omega\IntComp}\circ\varphi(z)}{\D}+\AnalyticIn{h\circ\varphi(z)}{\D}.
\end{align*}
However $C^{\Omega\IntComp}\circ\varphi\in\A_0(\D\IntComp)$, so we obtain
\begin{equation}\label{eqn:UQDPreFTFormula}
    \varphi(z)=cz+\AnalyticIn{h\circ\varphi(z)}{\D}^{\#}
\end{equation}

by Equation \ref{eqn:UQDPreFTFormula}, $\varphi(z)=cz+\AnalyticIn{h\circ\varphi(z)}{\D}^{\#}$. As $h\in\A(\Omega\IntComp)$, we see that $\AnalyticIn{h\circ\varphi(z)}{\D}=\Phi_{\varphi}^{-1}\left(h\right)$ by the definition of the exterior Faber transform. Hence,
$$\varphi(z)=cz+\Phi_{\varphi}^{-1}\left(h\right)^{\#}(z).$$
Rearranging yields $\varphi^{\#}(z)-cz^{-1}=\Phi_{\varphi}^{-1}\left(h\right)(z)$. Taking the Faber transform of both sides, we obtain
$$h=\Phi_{\varphi}\left(\varphi^{\#}(z)-cz^{-1}\right)=\Phi_{\varphi}\left(\AnalyticIn{\varphi^{\#}}{\D}\right).$$
\end{proof}

Explicit formulae for the Faber transforms of rational functions can be obtained by calculus of residues. Let $w_0\in\Omega$ and $z_0\in\D$ ($\D\IntComp$ for the exterior transform). Then
\begin{equation}\label{eqn:FTFormulae}
\begin{alignedat}{1}
\Phi_{\varphi}\left(\dfrac{1}{(z-z_0)^{n}}\right)(w)&=\frac{1}{(n-1)!}\left.\left(\frac{\varphi'(\xi)}{w-\varphi(\xi)}\right)^{(n-1)}\right\vert_{\xi=z_0}\\
\Phi_{\varphi}^{-1}\left(\dfrac{1}{(w-w_0)^{n}}\right)(z)&=\frac{1}{(n-1)!}\left.\left(\frac{\psi'(\xi)}{z-\psi(\xi)}\right)^{(n-1)}\right\vert_{\xi=w_0}
\end{alignedat}
\end{equation}
The specific formulae for $n=1$ and $n=2$ are provided below.

\begin{equation}\label{eqn:FaberPolyFormulae}
\begin{alignedat}{2}
\Phi_{\varphi}\left(\dfrac{1}{z-z_0}\right)(w)&=\dfrac{\varphi'(z_0)}{w-\varphi(z_0)},\;\;\;&\Phi_{\varphi}\left(\dfrac{1}{(z-z_0)^{2}}\right)(w)&=\dfrac{\varphi''(z_0)}{w-\varphi(z_0)}+\dfrac{\varphi'(z_0)^2}{(w-\varphi(z_0))^2}\\
\Phi_{\varphi}^{-1}\left(\dfrac{1}{w-w_0}\right)(z)&=\dfrac{\psi'(w_0)}{z-\psi(w_0)},\;\;\;&\Phi_{\varphi}^{-1}\left(\dfrac{1}{(w-w_0)^{2}}\right)(z)&=\dfrac{\psi''(w_0)}{z-\psi(w_0)}+\dfrac{\psi'(w_0)^2}{(z-\psi(w_0))^2}\\
& & &\\
F_1(w)&=\dfrac{w}{c}-\dfrac{f_0}{c}, &F_2(w)&=\dfrac{w^2}{c^2}-\dfrac{2f_0}{c^2}w+\dfrac{f_0^2-2cf_1}{c^2}\\
W_1(z)&=cz+f_0, &W_2(z)&=c^2z^2+2cf_0z+f_0^2+2cf_1,
\end{alignedat}
\end{equation}
where the $f_j$ are the coefficients in the Laurent expansion of $\varphi$ at $\infty$, $\varphi(z)=cz+f_0+\left<z^{-1}\right>$.

\noindent The basic theory of the Faber transform having been established, we are now prepared to discuss our first major example of this approach: the cardioid.

\subsubsection{The Cardioid}\label{subsubsec:FaberCardioid}
Using the Faber transform, we may directly verify the cardioid example depicted in Figure \ref{fig:BinomialCardioidQDFamily}. Suppose $\Omega=\varphi(\D)$ for some $\varphi(z)=az+bz^2$ with $a>0$ and $b\in\C$ such that $\varphi$ is univalent in $\D$ (this is precisely when $2|b|\leq a$). Then, by Theorem \ref{thm:ClassicBQDRationalRiemannIffQD}, $\Omega$ is a quadrature domain with quadrature function
\begin{align*}
h(w)&=\Phi_{\varphi}\left(az^{-1}+\overline{b}z^{-2}\right)(w)=a\dfrac{\varphi'(0)}{w-\varphi(0)}+\overline{b}\dfrac{\varphi''(0)}{w-\varphi(0)}+\overline{b}\dfrac{\varphi'(0)^2}{(w-\varphi(0))^2}=\dfrac{a^2+2|b|^2}{w}+\dfrac{\overline{b}a^2}{w^2}.
\end{align*}
On the other hand if we assume knowledge only of $h(w)=\alpha_0w^{-1}+\alpha_1w^{-2}$ ($\alpha_0>0$) and that $\Omega\in\QD(h)$ exists and is simply connected, then applying Equations \ref{eqn:BoundedFaberTransformPhiFormula} and \ref{eqn:FaberPolyFormulae} (WLOG setting $\varphi(0)=0$),
\begin{align*}
\varphi(z)&=\varphi(0)+\Phi_{\varphi}^{-1}\left(\alpha_0w^{-1}+\alpha_1w^{-2}\right)^{\#}(z)=\left(\dfrac{\alpha_0\psi'(0)}{z-\psi(0)}+\dfrac{\alpha_1\psi''(0)}{z-\psi(0)}+\dfrac{\alpha_1\psi'(0)^2}{(z-\psi(0))^2}\right)^{\#}=az+bz^{2},
\end{align*}
for some $a,b\in\C$. In particular, if $\Omega$ is a simply connected domain, then there exists an $h$ of the above form for which $\Omega\in\QD(h)$ if and only if $\Omega$ admits a quadratic Riemann map (c.f. \cite{AharonovShapiro}). Moreover, the coefficients satisfy the system of algebraic equations: $\alpha_0=a^2+2|b|^2$, $\alpha_1=\overline{b}a^2$. Thus this approach provides a way of converting the inverse problem for quadrature domains into a problem of solving finite systems of algebraic equations.

We now take a brief detour into logarithmic potential theory. We will see that solutions to the associated ``obstacle problem'' are, in some sense, dual to quadrature domains. Applying this duality, along with the prior results, we obtain a classification of simply connected one point QDs - both bounded and unbounded.

\subsection{Potential Theory and Hele-Shaw Flow}\label{sec:PotentialTheoryHeleShaw}
Let $\mu$ be a non-negative measure on $\C$ with compact support, $K_\mu=\text{supp}(\mu)$. The logarithmic potential and logarithmic energy associated to $\mu$ are respectively given by
\begin{equation*}
\begin{aligned}[c]
   U^\mu(w)=\int_{\C}\ln\dfrac{1}{|w-\xi|}d\mu(\xi),
\end{aligned}
\qquad\qquad
\begin{aligned}[c]
    I(\mu)=\int_{\C}U^\mu d\mu.
\end{aligned}
\end{equation*}

Now suppose that we have an {\it admissible} external field $Q$. By admissible, we mean $Q:\C\rightarrow(-\infty,\infty]$ is lower-semicontinuous and satisfies the growth condition
\begin{equation*}
    \lim_{|w|\to\infty}(Q(w)-\ln|w|)=\infty.
\end{equation*}
In this case, the $Q-$potential and $Q-$energy (also referred to as the weighted logarithmic capacity of $\mu$) are respectively given by
\begin{equation*}
\begin{aligned}[c]
   U_{Q}^{\mu}(w)=\int_{\C}\left(\ln\dfrac{1}{|w-\xi|}+Q(\xi)+Q(w)\right)d\mu(\xi),
\end{aligned}
\qquad\qquad
\begin{aligned}[c]
    I_Q(\mu)=\int_{\C}U_{Q}^{\mu}d\mu.
\end{aligned}
\end{equation*}
If $Q$ is admissible, then there exists a unique $Q-$energy minimizing measure $\mu$, $\mu(1)=1$, called the {\it equilibrium measure} associated to $Q$. This is stated rigorously in the following theorem due to Frostman \cite{OFrostman}.
\begin{lemma}[Frostman]\label{lemma:Frostman}
If an external field $Q$ is admissible then there exists a measure $\muhat$ such that
\begin{equation*}
    I_{Q}(\muhat)=\gamma:=\underset{\mu:\mu(1)=1}{\inf}I_Q(\mu).
\end{equation*}
Moreover, $\widehat{\mu}$ is the unique compactly supported measure for which $U_{Q}^{\muhat}=\gamma$ on $K_\muhat$ and $U_{Q}^{\muhat}\geq\gamma$ q.e. (up to a set of capacity zero) in $\C$.\footnote{Intuitively, the equilibrium charge distribution must induce a constant potential on its support because otherwise a current would be induced, redistributing the charge.} We refer to the support of the equilibrium measure $K_{\muhat}$ as a ``droplet''.
\end{lemma}
Note that even if $Q$ fails to satisfy the desired growth condition at $\infty$, we can still consider the obstacle problem {\it locally}. In particular, if $X$ is a compact subset of $\C$, then we can localize to $X$ by instead considering the obstacle problem associated to the potential $Q_X(w)=Q(w)\1_{X}+\infty\1_{X^c}$. We refer to the support of the equilibrium measure in this case as a ``\hypertarget{text:localdropletdef}{local droplet}''.

Now suppose that $Q$ satisfies the stronger growth condition,
\begin{equation*}
    \lim_{|w|\to\infty}\left(Q(w)-t\ln|w|\right)=\infty,\;\forall t\in\R.
\end{equation*}

By Frostman's theorem, there exists a unique equilibrium measure $\muhat$ associated to $Q$. The Gauss variational (or ``obstacle'') problem is to find
\begin{equation*}
    \nu_t(w)=\sup\left\{\nu(w):\nu\leq Q\text{ q.e., }\nu\text{ is subharmonic in }\C\text{, and }\sim t\ln|w|\right\}.
\end{equation*}
A direct consequence of this definition is that $\nu_t(w)=U_{Q}^{\mu_t}(w)+c(t)$, where $\mu_t=t\muhat$, for $\muhat$ the equilibrium measure associated to $t^{-1}Q$, and $c(t)$ is a constant. Furthermore if $Q\in W^{2,p},\;p>1$ in an open nbhd of $K_t:=\text{supp}(\mu_t)$, then
\begin{equation*}
    d\mu_t=\1_{K_t}\dfrac{\Delta Q}{2}dA.
\end{equation*}
In particular, with the above assumptions, solving the obstacle problem is equivalent to finding the support of the equilibrium measure associated to the $Q-$potential. This motivates the following addendum to Frostman's theorem.
\begin{lemma}[Frostman, part 2]\label{lemma:FrostmanPt2}
Let $Q$ be an admissible external field, $K$ a compact subset of $\C$, and $d\mu=\1_{K}\frac{\Delta Q}{2}dA$. If $U^{\mu}+Q$ is constant q.e. on $K$ then $K$ is a local droplet of $Q$. \cite{SaffTotik,Lee_2015} (lemma 3.2)
\end{lemma} 

Lemma \ref{lemma:FrostmanPt2} can be applied to obtain a useful relation for solving the obstacle problem, the {\it coincidence equation} (Equation \ref{eqn:ceq}). To this end, we define the coincidence set $K_t^\ast=\left\{w\in\C:\nu_t(w)=-Q(w)\right\}$. $K_t^\ast$ is compact, nonempty, and $K_t\subseteq K_t^\ast$. In addition, $\nu_t$ is harmonic in $\C\setminus K_t^\ast$. As a consequence, we have that
\begin{equation*}
    -Q(w)=\nu_t(w)=U^{\mu_t}(w)+c(t)
\end{equation*}
on $K_t^\ast$. Differentiating with respect to $w$, we find that
\begin{equation}\label{eqn:ceq}
    \frac{\partial Q}{\partial w}\dEquals C^{K_t}_Q(w),
\end{equation}
where $\dEquals$ denotes equality on $\partial K_t$, and
\begin{equation*}
    C^{K_t}_Q(w):=\dfrac{1}{4}\int_{K_t}\dfrac{\Delta Q}{w-\xi}dA(\xi)
\end{equation*}
is $\frac{1}{2}$ the weighted Cauchy transform associated to $\mu_t$. Note that $C^{K_t}_Q$ is analytic in $K_t^c$ and $C^{K_t}_Q(w)=O(w^{-1})$ about $\infty$.\\

The obstacle problem is quite difficult to solve explicitly in general. However, there are a handful of cases in the literature for which this has been accomplished in the algebraic case.\cite{DragnevLeggSaff} Lemma \ref{lemma:ClassicQDDroplet} asserts a certain kind of duality between local droplets and quadrature domains: to solve the obstacle problem, it is sufficient to determine the quadrature domains associated to the local droplet. In \ref{subsec:SCQDs} we described a method of recovering simply connected QDs from their quadrature function via the Faber transform. Thus this method may also be used to tackle the obstacle problem. We will also demonstrate how this duality can be applied to the uniqueness problem for quadrature domains.

\subsubsection{Duality of Local Droplets and Quadrature Domains}

For {\it Hele-Shaw} potentials ($\Delta Q$ is constant), we may correspond the complements of certain droplets with quadrature domains. In particular,
\begin{lemma}\label{lemma:ClassicQDDroplet}
Let $K$ be a local droplet corresponding to a Hele-Shaw potential, $Q(w)=|w|^2-2\Re(H(w))$, where $h=H'$ is a rational function. Then $K^c$ is a disjoint union of QDs, with $h=$ the sum of their quadrature functions.

Conversely, if the complement of some compact set $K$ is a disjoint union of QDs with $h=$ the sum of their quadrature functions, then $K$ is a local droplet of a potential of the form
\begin{equation*}
Q(w)=|w|^2-2\Re\left(\sum_{l}\1_{K_l}H_l(w)\right),\;\;\;H_l(w)=c_l+\int_{w_l}^{w}h(\xi)d\xi,
\end{equation*}
where the $\{K_l\}_l$ are the connected components of $K$, $w_l$ is some point in $K_l$, and $c_l$ is a real constant. \cite{Lee_2015}
\end{lemma}

If we consider a family of these droplets parameterized with respect to $t$, we arrive at the concept of a {\it \hypertarget{text:HSChain}{backward Hele-Shaw chain}}, or simply {\it Hele-Shaw chain}. Let $K_{t_0}$ be a local droplet of charge $t_0>0$ of an admissible potential, $Q$. The backward backward Hele-Shaw chain associated to $K_{t_0}$ is the family of compact sets,
$$K_t:=S_t[Q_{K_{t_0}}],\;\;0<t\leq t_0,$$
where $Q_{K_{t_0}}$ is the localization of $Q$ to $K_{t_0}$ and $S_t[Q_{K_{t_0}}]$ is the support of the equilibrium measure of mass $t$ associated to $Q_{K_{t_0}}$. Several properties of the chain $\{K_t\}_t$ are enumerated below.
\begin{lemma}\label{lemma:HSChains}
If $K_{t_0}$ is a local droplet of charge $t_0$ associated to an admissible potential $Q$. Then the backward Hele-Shaw chain $\{K_t\}_{0<t\leq t_0}$ associated to $K_{t_0}$ exists and satisfies the following
\begin{enumerate}
\item Each $K_t$ is a local algebraic droplet of $Q$.
\item The chain $\{K_t\}_t$ is monotone increasing.
\item The chain is left-continuous, $K_{t_0}=\text{Cl}\left(\bigcup_{0<t<t_0}K_t\right)$.
\item If $K_{t_0}$ has no double points, then $\exists\epsilon>0$ such that the outer boundary of $K_t$ has no singular points for all $t\in(t_0-\epsilon,t_0)$.
\end{enumerate}
\end{lemma}
See \cite{Lee_2015,Hedenmalm_Makarov_2004,Hedenmalm_Makarov_2011} for a more detailed discussion.

If $\{K_t\}_{0<t\leq t_0}$ is a backward Hele-Shaw chain associated to a potential $Q$, then we can extend the chain to $t=0$ by defining $K_0=\emptyset$ and $K_0^\ast$ the set where the global minimum of $Q$ is attained. That is,
\begin{lemma}\label{lemma:HSChainZeroCoincidence}
If $\{K_t\}_{0<t\leq t_0}$ is a backward Hele-Shaw chain associated to some potential $Q$, then we can extend the chain to $t=0$ by defining $K_0=\emptyset$ and $K_0^\ast$ to be the set where the global minimum of $Q$ is attained. Moreover, $K_0^\ast\subset K_{t}$ for each $0<t<t_0$. (see remark 4.20 and corr. 4.23 in \cite{Hedenmalm_Makarov_2011})
\end{lemma}
This also implies that if $K$ is a droplet of an admissible potential $Q$, then $K$ contains the set where the global minimum of $Q$ is attained.

\section{Classification of One Point QDs}\label{sec:OnePtQDClass}
We are now prepared to consider the problem of classifying simply connected one point quadrature domains. That is, simply connected $\Omega\in\QD\left(\frac{\alpha}{w-w_0}\right)$ for some $\alpha\in\C$ and $w_0\in\C$. Equation \ref{eqn:ClassicMVP} is the canonical example of this. That is, we're interested in classifying domains $\Omega$ admitting a quadrature identity of the form
$$\int_{\Omega}fdA=\alpha f(w_0)$$
While the case in which $\alpha\geq0$ has been extensively studied (see, e.g. \cite{DragnevLeggSaff}), the $\alpha\in\C$ case has received little attention. It is well-known that if $\alpha=0$ then $\Omega$ is an exterior disk centered at the origin (e.g. \cite{Lee_2015}). It's worth noting that, while both BQDs and UQDs generically exist when $\alpha>0$, this is not true when $\alpha\not\geq0$, in which case there are no BQDs (this follows straightforwardly from applying the quadrature identity to $1\in\A(\Omega)$). The central result of this section is the following classification theorem:

\begin{theorem}\label{thm:OnePtQDClassification}
Take $\alpha\in\C\setminus\R_{\geq0}$ and $w_0\neq0$. Then
$$\QD\left(\frac{\alpha}{w-w_0}\right)\neq\emptyset\;\;\;\;\;\iff\;\;\;\;\;|w_0|^2+2\Re\left(\alpha\right)>2|\alpha|.$$
In this case, there exists $t_\ast\in(0,\infty)$ such that
\begin{enumerate}
    \item $\QD\left(\frac{\alpha}{w-w_0}\right)=\left\{\Omega_t\right\}_{0<t\leq t_\ast}$, a continuous monotone (\ref{lemma:HSChains}) family of simply connected domains (where $t=\text{A}(\Omega_t)$);
    \item $\partial\Omega_t$ is smooth for all $0<t<t_\ast$ and the family terminates with the formation of a $(3,2)$ cusp at $t=t_\ast$.
\end{enumerate}
\end{theorem}
In the course of the proof of \ref{thm:OnePtQDClassification}, we also obtain a characterization of the Riemann map associated to each $\Omega_t$,

\begin{corollary}
If $\Omega\in\QD\left(\frac{\alpha}{w-w_0}\right)$ is unbounded then it is simply connected and its Riemann map is given by Equation \ref{eqn:classicQDClassification}.\footnote{If $w_0=2$ then $c$, $\alpha$, and $z_0$ are related by Equation \ref{eqn:OnePtQDConservedQuantity}. See Fig. \ref{fig:OnePtUQDFams} for several sample families of solutions.}
\end{corollary}

We begin by simplifying the situation via a change of variables: Note that if $a,b\in\C$ and $a\neq0$, then
\begin{equation}\label{eqn:classicQDCOV}
\begin{split}
\Omega\in\QD(h)\tab&\iff\tab a\Omega+b\in\QD\left(\overline{a}h\left(\frac{w-b}{a}\right)\right),\\
\Omega\in\QD(h)\tab&\iff\tab a\Omega+b\in\QD\left(\overline{a}h\left(\frac{w-b}{a}\right)+\overline{b}\right),
\end{split}
\end{equation}
when $\Omega$ is bounded and unbounded respectively. To see this note that if we set $g(w)=aw+b$, then

\begin{align*}
\int_{g(\Omega)}fdA=\int_{\Omega}f\circ g|g'|^2dA=|a|^{2}\oint_{\partial\Omega}f\circ g(w)h(w)dw=\overline{a}\oint_{\partial g(\Omega)}f\left(w\right)h\circ g^{-1}(w)dw.
\end{align*}
The unbounded case is analogous, except one must take into account the Cauchy principal value.\\

Applying this result (Equation \ref{eqn:classicQDCOV}) with $a=\frac{|w_0|}{w_0}$ and $b=0$, yields $\Omega\in\QD\left(\frac{\alpha}{w-w_0}\right)$ if and only if $\frac{|w_0|}{w_0}\Omega\in\QD\left(\frac{\alpha}{w-|w_0|}\right)$. So we can WLOG assume that $w_0>0$ (unless $w_0=0$, in which case $\alpha$ must be positive and $\Omega=\D_{\sqrt{\alpha}}(0)$). We begin with a discussion of the case in which $\alpha$ is real. 

\subsection{\texorpdfstring{One Point QDs for Real $\mathbf{\alpha}$}{One Point QDs for Real α}}
While the family of one point QDs with positive quadrature constant $\alpha$ is well understood, much less is known concerning the case in which $\alpha<0$. We review the positive case first.

\subsubsection{\texorpdfstring{$\mathbf{\alpha>0}$}{α>0}}

Here it is instructive to consider the problem from the perspective of Hele-Shaw flow. In particular, by Lemma \ref{lemma:ClassicQDDroplet} if $\Omega\in\QD\left(\frac{\alpha}{w-w_0}\right)$ is unbounded and simply connected, then its complement $K$ is a local droplet of the potential
$$Q(w)=|w|^2-2\Re(\alpha\ln(w-w_0)).$$

Thus one might expect a $1-$parameter family of QDs $\{\Omega_t\}_{t}$, parameterized by the measure of the droplets in the corresponding Hele-Shaw chain. This turns out to be the case: if $\alpha>0$ then there is a unique $1-$parameter family of simply connected domains $\{\Omega_t\}_{0<t\leq t_{\ast}}=\QD\left(\frac{\alpha}{w-w_0}\right)\bigcap\left\{\Omega\subset\Ch:\Omega\text{ is unbounded}\right\}$, where $t_{\ast}=w_0(w_0+2\sqrt{\alpha})$. These may be characterized in terms of their Riemann maps, $\{\varphi_t:\D\IntComp\rightarrow\Omega_t\}_t$, given by 
\begin{equation}\label{eqn:classicQDClassification}
\varphi_t(z)=cz\dfrac{z-z_0+\frac{w_0}{c}\frac{|z_0|^2-1}{|z_0|^2}}{z-\overline{z_0}^{-1}},
\end{equation}
where $c=c(t)$ is the conformal radius of $\Omega_t$, and $z_0=z_0(c)$ is the unique real number $\geq1$ for which
\begin{equation}\label{eqn:RealClassicOnePtz0Eqn}
0=c^2z_0^4-cw_0z_0^3-\alpha z_0^2-cw_0z_0+w_0^2,\;\text{ and }\;\varphi_t(1)<w_0.
\end{equation}
Such a $z_0$ can be shown to exist if and only if $0<c\leq w_0+\sqrt{\alpha}$ (corresponding to $0<t\leq w_0(w_0+2\sqrt{\alpha})$). $c$ and $t$ satisfy
\begin{equation}\label{eqn:RealClassicOnePtaImpt}
0=t^4+(2\alpha-c^2)t^3+(\alpha^2+2c^2(3w_0^2-\alpha))t^2+c^2w_0^2(4\alpha+w_0^2)t+8c^6w_0^2
\end{equation}
\begin{figure}[ht]
  \centering
    \includegraphics[scale=.315]{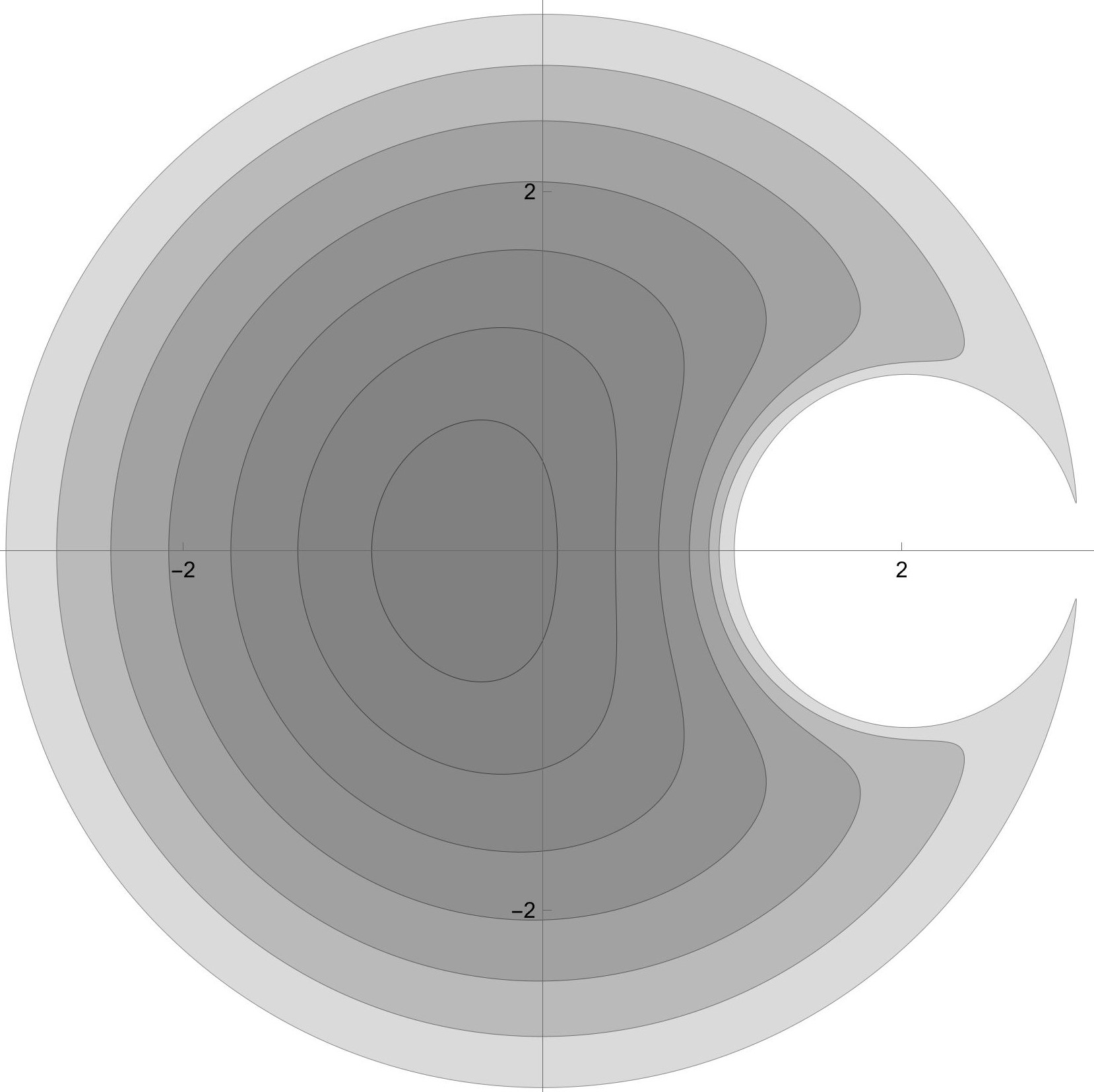}
    \includegraphics[scale=0.5768]{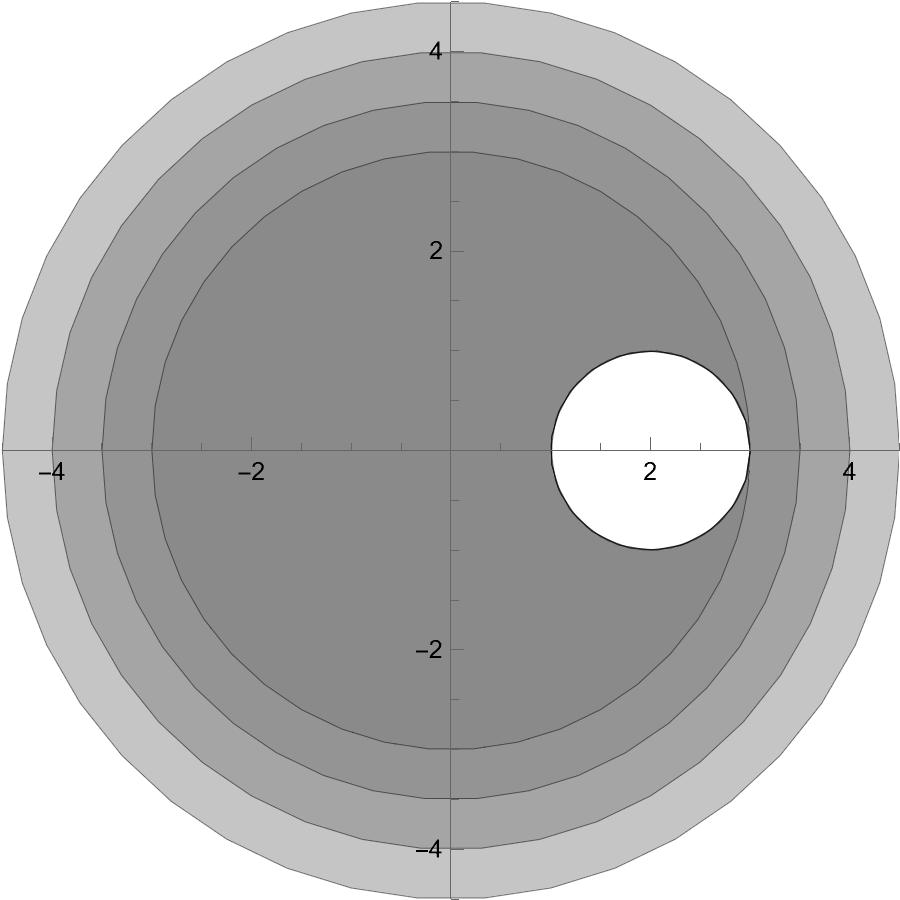}
    \caption{$\Omega_t\in\QD\left(\frac{1}{w-2}\right)$ (complements of shaded regions) for increasing $t$ up to the critical time (left) and after the critical time (right).}\label{fig:OnePtRealFam}\vspace{1.5em}
\end{figure}

On the other hand when $t>t_{\ast}$ the unique {\it bounded} family $\{\Omega_t\}_{t>t_{\ast}}$ is constant. It is well-known that $\Omega_t=\D_{\sqrt{\alpha}}(w_0)$. From the perspective Hele-Shaw flow, this can be understood as a two phase flow for which $K_t=\Omega_t^c$ is initially simply connected, before forming a double-point at some critical time $t_\ast$, after which $K_t$ becomes doubly connected of the form $\D_{r(t)}(0)\setminus\D_{\sqrt{\alpha}}(w_0)$ (see Fig. \ref{fig:OnePtRealFam}). Note that for $t\geq t_\ast$, $K_t^c$ actually contains two components, one being our one point BQD, and the other a null QD.

\subsubsection{\texorpdfstring{$\mathbf{\alpha<0}$}{α<0}}

The situation when $\alpha<0$ is in many ways the same as the $\alpha>0$ case. In particular, there exists $0<t_{\ast}<\infty$ and a simply connected $1-$parameter family $\{\Omega_t\}_{0<t\leq t_{\ast}}=\QD\left(\frac{\alpha}{w-w_0}\right)$ if and only if $-w_0^2<4\alpha<0$. They too are characterized by Equations \ref{eqn:classicQDClassification} and \ref{eqn:RealClassicOnePtz0Eqn}. The two cases are distinguished, however, by their behaviour at the critical time. In particular, while $\Omega_{t_{\ast}}$ forms a double point when $\alpha>0$, a $(3,2)-$cusp develops as $t\to t_\ast$ when $\alpha<0$, after which no such $\Omega_t$ exists. Moreover $c_{\ast}$ is the middle root of the cubic $8c^3-15w_0c^2+6(w_0^2+4\alpha)c+\frac{(w_0^2-\alpha)(w_0^2+4\alpha)}{w_0},$
which is given explicitly by
$$c_{\ast}=\dfrac{w_0}{8}\left(5-6\sqrt{1-\beta}\sin\left(\dfrac{1}{3}\sin^{-1}\left(\dfrac{\frac{3}{32}\beta^2-3\beta-1}{\left(1-\beta\right)^{\frac{3}{2}}}\right)\right)\right),\;\;\;\beta=\alpha\left(\dfrac{8}{3w_0}\right)^{2}.$$
$t_{\ast}$ can then be determined from Equation \ref{eqn:RealClassicOnePtaImpt} (see Fig \ref{fig:OnePtRealFamneg}).\\
\begin{figure}[ht]
  \centering
    \includegraphics[scale=.5]{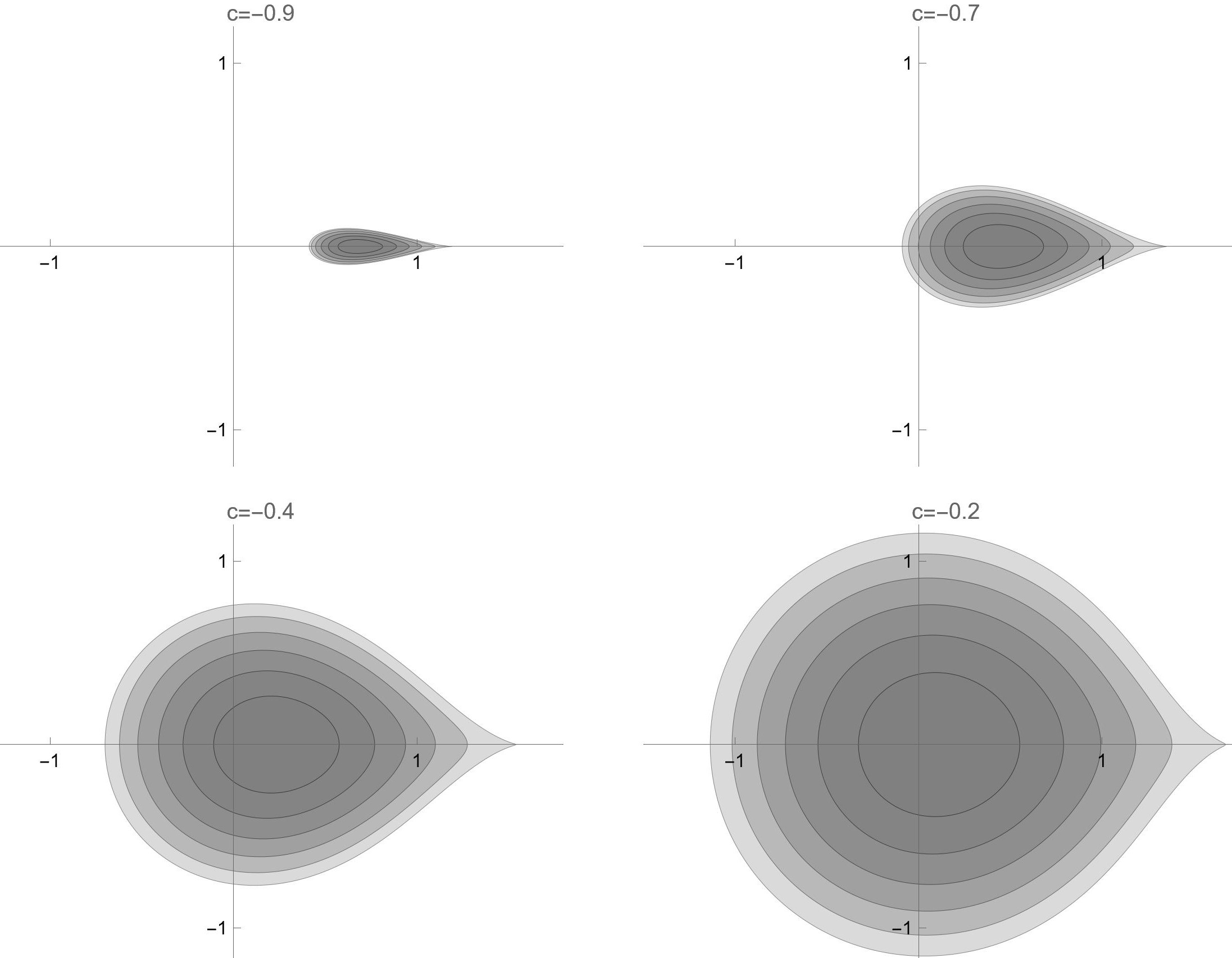}
    \caption{$\Omega_t\in\QD\left(\frac{\alpha}{w-2}\right)$ (complements of shaded) for $\alpha\in\{-.9,-.7,-.4,-.2\}$ for increasing $t$ up to the critical time.}\label{fig:OnePtRealFamneg}\vspace{1.5em}
\end{figure}

We now shift our consideration to the more general situation of $\alpha\in\C\setminus\R_{\geq0}$, in which case $\Omega$ is necessarily unbounded when it exists. Here we will see that such domains exist precisely when $\alpha$ lies within the parabolic region characterized by $|w_0|^2+2\Re(\alpha)>2|\alpha|$. In addition, we will see that $\Omega_t$ (as defined above) ceases to exist for $t$ sufficiently large. We're now prepared to prove the main theorem of the section.

\subsection{Proof of Theorem \ref{thm:OnePtQDClassification}}\label{proof:OnePtQDClassification}
\begin{comment}

For $\alpha\in i\R\setminus\{0\}$ symmetry about the real axis is lost, which complicates things. However, it turns out that the qualitative structure of the solution is quite similar to the $\alpha<0$ case. In particular, there exists $0<t_{\ast}<\infty$ and a simply connected $1-$parameter family\\$\{\Omega_t\}_{0<t\leq t_{\ast}}=\QD\left(\frac{\alpha}{w-w_0}\right)\cap\left\{\Omega\subset\Ch:\Omega\text{ is unbounded}\right\}$ if and only if $|\alpha|<\frac{w_0^2}{2}$. These are unique when they exist and are characterized by Equation \ref{eqn:classicQDClassification}. Moreover, similar to the $\alpha<0$ case, a $(3,2)-$conformal cusp develops as $t\to t_\ast$, after which no such $\Omega_t$ exists (see Fig. \ref{fig:OnePtRealFamImaginary}).\\
\end{comment}
\begin{figure}[ht]
  \centering
    \includegraphics[scale=.35]{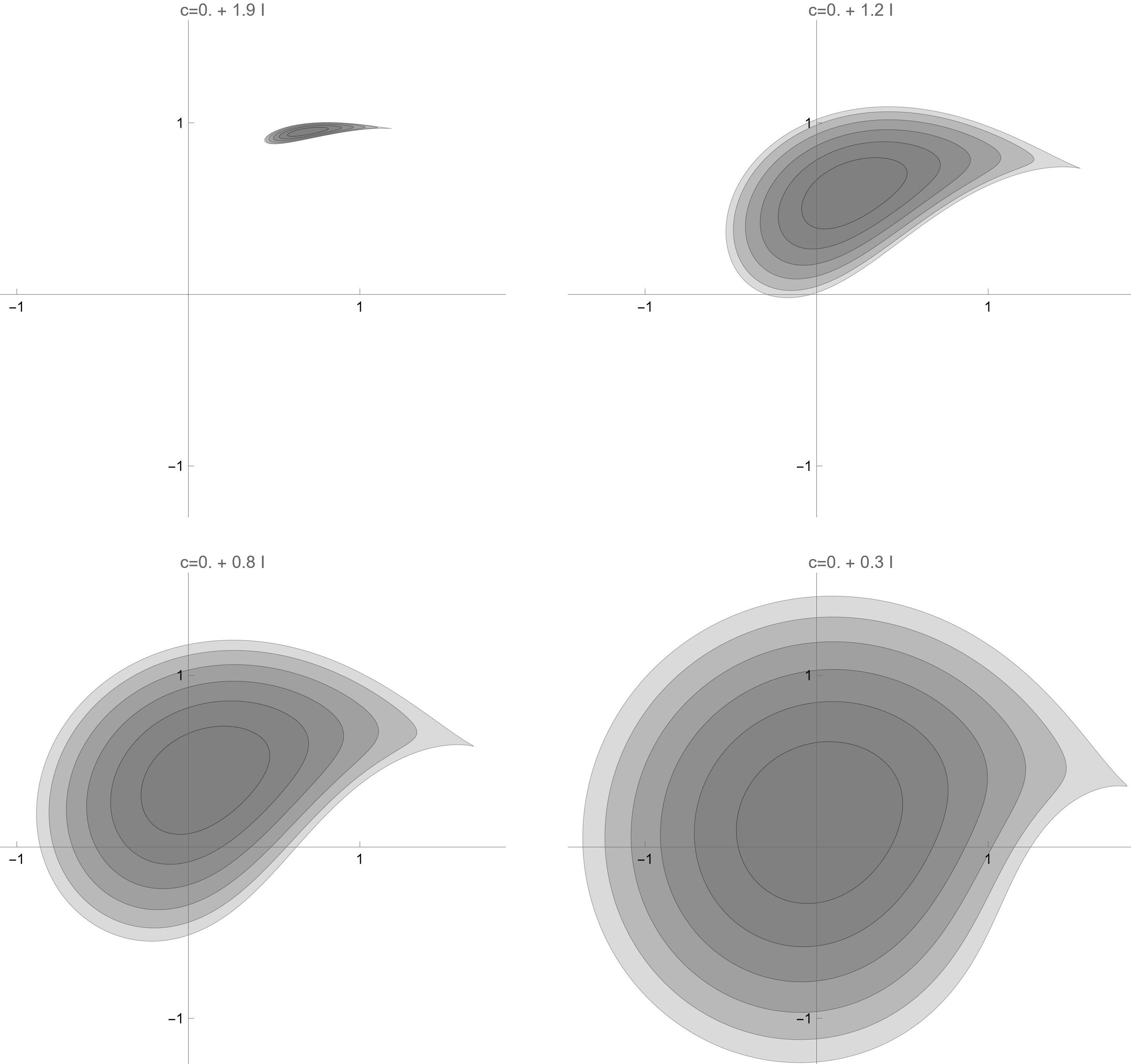}
    \caption{$\Omega_t^c$ for $\alpha\in\{1.9i,1.2i,.8i,.3i\}$ and $w_0=2$ (shaded) for increasing $t$ up to the critical time.}\label{fig:OnePtRealFamImaginary}\vspace{1.5em}
\end{figure}

First and foremost, we establish that $\Omega\in\QD\left(\frac{\alpha}{w-w_0}\right)$ is simply connected: This is an immediate consequence of the fact that $\Omega$ is a quadrature domain of order $d_\Omega=\deg\left(\frac{\alpha}{w-w_0}\right)=1$. (\cite{Lee_2015} \S 1.4) Next, by Equation \ref{eqn:classicQDCOV}, $\Omega\in\QD\left(\frac{\alpha}{w-w_0}\right)$ if and only if $\frac{2}{w_0}\Omega\in\QD\left(\frac{4\alpha|w_0|^{-2}}{w-2}\right)$. That is, it will be sufficient to consider $\Omega\in\QD\left(\frac{\alpha}{w-2}\right)$.

\subsubsection{Existence Criterion}\label{subsubsec:PnePtQDExistenceCriterion}
The general approach here is to show
\begin{enumerate}
    \item $\QD\left(\frac{\alpha}{w-2}\right)$ is non-empty $\iff$ a certain potential $Q$ has a local minimum,
    \item $Q$ has a local minimum $\iff$ $2+\Re(\alpha)>|\alpha|$.
\end{enumerate}

{\bf(1)} Suppose there exists an unbounded $\Omega\in\QD\left(\frac{\alpha}{w-2}\right)$. Then, by lemma \ref{lemma:ClassicQDDroplet}, $K:=\Omega^c$ is a local droplet of the potential
$$Q(w)=|w|^2-2\Re(\alpha\ln(w-w_0))+\sum_{j=1}^{n}c_j\1_{K^{(j)}},\;\;K=\bigsqcup_{j=1}^{n}K^{(j)},$$
where the $K^{(j)}$ are the connected components of $K$. Let $\{K_t\}_{t=0}^{\text{A}(K)}$ be the backwards Hele-Shaw chain associated to $K$, so that $K_0^{\ast}\neq\emptyset$ (\ref{lemma:HSChains}). However, $K_0^{\ast}$ is the locus of minima of $Q$ so this set, too, is non-empty. 

On the other hand, suppose $Q(w)=|w|^2-2\Re(\alpha\ln(w-2))$ has a unique local minimum, so it admits a local droplet $K_t$ for some $t>0$. Then $\Omega:=K_t^{c}\in\QD\left(\frac{\alpha}{w-2}\right)$ (by \ref{lemma:ClassicQDDroplet}).\\

{\bf(2)} We now show that $Q(w)=|w|^2-2\Re(\alpha\ln(w-2))$ admits a local minimum if and only if $2+\Re(\alpha)>|\alpha|$, and that this local minimum is unique when it exists. $Q$ may have a local minimum only if $\overline{w}=\frac{\alpha}{w-2}$. Taking this equation with its conjugate, we arrive at the following set of all candidate minima:
$$w=\frac{1}{2}\left(2+i\Im(\alpha)\pm\sqrt{4+4\Re(\alpha)-\Im^2(\alpha)}\right).$$
If $4+4\Re(\alpha)-\Im^2(\alpha)<0$, then
$$\overline{w}-\frac{\alpha}{w-2}=\pm i\sqrt{\Im^2(\alpha)-4\Re(\alpha)-4}\neq0,$$
so no solutions exist in this case. 
If $4+4\Re(\alpha)-\Im^2(\alpha)=0$, then our unique candidate solution is $w=1+i\dfrac{\Im(\alpha)}{2}$. However, expanding in series about this point demonstrates that it is not a minimum.\\
Finally, if $4+4\Re(\alpha)-\Im^2(\alpha)>0$, then its easy to show that $\frac{\partial Q}{\partial w}=0$, so we have candidate minima only in this case. Local analysis demonstrates that $w=\frac{1}{2}\left(2+i\Im(\alpha)+\sqrt{4+4\Re(\alpha)-\Im^2(\alpha)}\right)$ is not, in fact, a minimum. So, finally, we'll check that $w=\frac{1}{2}\left(2+i\Im(\alpha)-\sqrt{4+4\Re(\alpha)-\Im^2(\alpha)}\right)$ is a local minimum. For a root of $\frac{\partial Q}{\partial w}$ to be a local minimum, it's sufficient to check that $\Delta Q>4\left|\frac{\partial^2Q}{\partial w^2}\right|$ or, equivalently, $|\alpha|<|w-2|^2$. By assumption,
$$-1<\dfrac{\Re(\alpha)+1}{\sqrt{4+4\Re(\alpha)-\Im^2(\alpha)}},$$
so
$$|\alpha|^2=\Re^2(\alpha)+\Im^2(\alpha)<\left(\Re(\alpha)+2+\sqrt{4+4\Re(\alpha)-\Im^2(\alpha)}\right)^2,$$
thus, noting that $\Re(\alpha)>-1$,
$$|\alpha|<\Re(\alpha)+2+\sqrt{4+4\Re(\alpha)-\Im^2(\alpha)}=\frac{\Im^2(\alpha)+\left(2+\sqrt{4+4\Re(\alpha)-\Im^2(\alpha)}\right)^2}{4}=|w-2|^2$$
The fact that $4+4\Re(\alpha)<\Im^2(\alpha)$ $\iff$ $2+\Re(\alpha)>|\alpha|$ completes the proof of this second claim. We conclude that $\QD\left(\frac{\alpha}{w-2}\right)$ is non-empty $\iff$ $2+\Re(\alpha)>|\alpha|$.

\subsubsection{The Riemann Map}
Fix $\alpha\not\geq0$ so $\Omega$ is unbounded (and simply connected, by above), and let $\varphi$ be the Riemann map associated to $\Omega$. Then by Equation \ref{eqn:UnboundedFaberTransformPhiFormula},
\begin{align*}
\varphi(z)&=cz+\Phi_{\varphi}^{-1}\left(\dfrac{\alpha}{w-2}\right)^{\#}(z)=cz+\left(\dfrac{\alpha\psi'(2)}{z-\psi(2)}\right)^{\#}=cz+\dfrac{\overline{\alpha}}{\overline{\varphi'(z_0)}}\dfrac{1}{z^{-1}-\overline{z_0}}=cz\dfrac{z-z_1}{z-\overline{z_0}^{-1}}
\end{align*}
for some $z_1\in\C$. Now recall that $2=\varphi(z_0)$, which implies $z_1=z_0-\frac{2}{c}\frac{|z_0|^2-1}{|z_0|^2}$, so
$$\varphi(z)=cz\dfrac{z-z_0+\frac{2}{c}\frac{|z_0|^2-1}{|z_0|^2}}{z-\overline{z_0}^{-1}}.$$

In particular we have shown that if $\Omega\in\QD\left(\frac{\alpha}{w-2}\right)$ for some $\alpha\not\geq0$, then $\Omega$ may be represented by a Riemann map of the above form.

\subsubsection{Existence and Uniqueness of the One-Parameter Family}
The fact that $\QD\left(\frac{\alpha}{w-2}\right)$ is nonempty when $2+\Re(\alpha)>\alpha$ was established in \ref{subsubsec:PnePtQDExistenceCriterion}. Thus the existence of the continuous monotone family $\left\{\Omega_t\right\}_t$ follows directly from the existence of a backward Hele-Shaw chain (\ref{lemma:HSChains}). For uniqueness, suppose $\exists \Omega_t$ and $\Omega_t'$ with $A(\Omega_t)=A(\Omega_t')=t>0$. Then by lemmas \ref{lemma:ClassicQDDroplet} and \ref{lemma:HSChainZeroCoincidence}, the complements of both of these domains, $K_t$ and $K_t'$, are local droplets of $Q$, which has a unique local minimum in $K_t\cap K_t'$. If we localize to an open nbhd of $K_t\cup K_t'$ then, by Frostman's theorem \ref{lemma:Frostman}, $K_t=K_t'$ q.e, so $\Omega_t=\Omega_t'$.

Now all we need establish is the existence of such a maximal $t_\ast$. Firstly, Equation \ref{eqn:UnboundedFaberTransformHFormula} gives
$$\dfrac{\alpha}{w-2}=h(w)=\Phi_{\varphi}\left(\AnalyticInNoBracket{\varphi^{\#}}{\D}\right)(w)=\dfrac{(cz_0|z_0|^2-2)(c\overline{z_0}-2)}{|z_0|^2(w-2)},$$
(we compute using partial fractions and Equation \ref{eqn:FTFormulae}). This implies
\begin{equation}\label{eqn:OnePtQDConservedQuantity}
    \alpha|z_0|^2=(cz_0|z_0|^2-2)(c\overline{z_0}-2).
\end{equation}
Then, setting $z_0=c^{-1}re^{i\theta}$, we find that $0=r^4-2e^{i\theta}r^{3}-\alpha r^2-2c^2re^{-i\theta}+4c^2$. Then, considering the asymptotics as $c\to\infty$, we see that either $r\in O(c^{\frac{2}{3}})$ or $r\in O(1)$. Thus $|z_0|\lesssim c^{-\frac{1}{3}}$, so $z_0\xrightarrow{c\to\infty}0$ but $|z_0|>1$, so no solution exists for $c$ sufficiently large. Thus there exists a maximal $t_\ast>0$.

\subsubsection{Cusp Formation}
We would like to show that (a) $\partial\Omega_{t_\ast}$ contains a single $(3,2)$ cusp and (b) $\partial\Omega_t$ is smooth whenever $0<t<t_\ast$. Recall that $\Omega$ has a Riemann map in the form of Equation \ref{eqn:classicQDClassification}. Thus for (a) all we need show is that if $\partial\Omega_t=f_t(\partial\D)$ (where $f_t(z)=z\frac{z-z_0-\alpha}{z-\overline{z_0}^{-1}}$ for some $\alpha=\alpha(t)\in\R$ and $z_0=z_0(t)\in\D\IntComp$) is a continuous family of bounded curves and there exists some maximal $t_\ast>0$ such that $f_t$ is univalent in $\D\IntComp$, then $\partial\Omega_{t_\ast}$ contains a single $(3,2)$ cusp. In particular, for (a) it is sufficient to show that

\begin{enumerate}
    \item If $f_t'(z_\ast)=0$ for some $z_\ast\in\partial\D$, then $f_t(\partial\D)$ has a $(3,2)-$cusp at $f_t(z_\ast)$.
    \item If $f_t$ continuously transitions out of univalence at $t=t_\ast$, then there exists $z_\ast\in\partial\D$ such that $f_t'(z_\ast)=0$.
\end{enumerate}

{\bf(1)} $f_t'(z)=z_0\frac{z(z\overline{z_0}-2)+z_0+\alpha}{(z\overline{z_0}-1)^2}$, so $0=f_t'(z_\ast)=z_\ast(z_\ast\overline{z_0}-2)+z_0+\alpha$. Considering this equation along with it's conjugate, we obtain $z_0+\alpha=2-\overline{z_0}$, so $f_t(z)=z\frac{z-2+\overline{z_0}}{z-\overline{z_0}^{-1}}$. Once again setting $f_t'(z_\ast)=0$, and solving for $z_\ast$, we find that $z_\ast=1$ or $z_\ast=\frac{2-\overline{z_0}}{\overline{z_0}}$. However, in the latter case, we find that $\Re(z_0)=1$, in which case $f_t$ is not univalent (in fact $f_t(\partial\D)$ is a slit). Thus $z_\ast=1$ so, expanding $f_t(e^{i\theta})$ about $\theta=0$, we obtain $\overline{z_0}^{-1}f_t(e^{i\theta})-1=\frac{\theta^2}{1-\overline{z_0}}+\frac{i\theta^3}{(1-\overline{z_0})^2}+O(\theta^4)$, confirming that the cusp is of type $(3,2)$.\\

{\bf(2)} Suppose that $f_t$ continuously transitions out of univalence at $t=t_\ast$, so that for $t>t_\ast$, there exist $z\neq w\in\partial\D$, such that $f_t(z)-f_t(w)=0$. Then
\begin{align*}
0&=\dfrac{f_t(z)-f_t(w)}{z-w}=-\overline{z_0}\dfrac{z_0+\alpha+w(z\overline{z_0}-1)-z}{(z\overline{z_0}-1)(w\overline{z_0}-1)}\tab\implies\tab z_0+\alpha+w(z\overline{z_0}-1)-z=0.
\end{align*}
This is a linear equation in $w$, so it admits at most one solution. Thus $f_t(\partial\D)$ can have at most one self intersection point, and this is a simple root, so the intersection is transversal. That is, for each $\epsilon>0$, $f_{t_\ast+\epsilon}$ is not univalent in $\D\IntComp$ so $f_{t_\ast+\epsilon}(\partial\D\IntComp)$ is a figure-eight.

By the Whitney-Graustein theorem, which states that two regular closed curves in the plane have a regular homotopy between them iff they have the same degree, there is no regular homotopy between the figure eight and a simple closed curve.\cite{Kranjc1993} Therefore a singularity in the derivative must form on the boundary during the transition. In particular, the derivative of $f_{t_\ast}$ has a singularity at some point $z_\ast\in\partial\D$. Then, because $|z_0|\neq1$, the form of $f_t$ implies that $f_{t_\ast}'(z_\ast)\neq\infty$, so $f_{t_\ast}'(z_\ast)=0$. We conclude that the only way that $f_t$ can continuously transition out of univalence is with the formation of a $(3,2)-$cusp on $f(\partial\D)$.\\

The claim (b) of the smoothness of $\partial\Omega$ for $0<t<t_\ast$ then follows immediately from the fact that if a droplet in a Hele-Shaw chain contains a $(3,2)$ cusp then it is maximal (\cite{Chang_2013} \S4.1).
\qed

\subsubsection{Univalence Criterion}

For univalence of $\varphi$, it is sufficient to check that $\varphi(z)\neq\varphi(w)$ for all distinct $z,w\in\partial\D$. The sufficiency of injectivity on the boundary is a classic result, sometimes referred to as ``Darboux's theorem'':
\begin{lemma}\label{thm:DarbouxUnivalence}
Let $f$ be an analytic map on a Jordan domain $\Omega$ which extends continuously to the boundary. If $f$ is injective on $\partial\Omega$ then $f$ is univalent in $\Omega$.
\end{lemma}

\begin{proof}[Proof of Lemma \ref{thm:DarbouxUnivalence}]
We shall prove the result only for bounded domains, as the proof in the unbounded case is analogous. Suppose that $f\in\A(\Omega)$ is continuous and injective on $\partial\Omega$, so $f(\partial\Omega)$ is a Jordan curve. Also let $U$ be the region enclosed by $f(\partial\Omega)$. By the argument principle, the number of solutions to $f(z)=w_0$ for each $w_0\in \Ch$ is given by
\begin{align*}
    \oint_{\partial\Omega}\dfrac{f'(z)dw}{f(z)-w_0}=\oint_{f(\partial\Omega)}\dfrac{dw}{w-w_0}=\1_{U}(w_0).
\end{align*}
That is, each $w_0\in U$ is the image under $f$ of a unique point in $\Omega$ and no point in $U^c$ is in the image of $f$. In particular, $f$ univalently maps $\Omega$ into the region enclosed by $f(\partial\Omega)$, so $f$ is univalent in $\Omega$.
\end{proof}

In our case, if $\varphi(z)=\varphi(w)$ but $z\neq w$, then one can show that
\begin{align*}
0&=\varphi(z)-\varphi(w)=\left(z-w\right)\left(c-\dfrac{\left(\overline{z_0}-z_0^{-1}\right)\left(2-c z_0\right)}{(z\overline{z_0}-1)(w\overline{z_0}-1)}\right)\;\implies\;\dfrac{z\overline{z_0}-1}{\left(|z_0|^2-1\right)\left(\frac{2}{cz_0}-1\right)}=\dfrac{1}{w\overline{z_0}-1}.
\end{align*}
But the left and right hand sides of this equation are simply Mobius transformations, so the image of $\partial\D$ under each is a circle. Thus it's necessary and sufficient to check that these circles don't intersect. A pair or circles $\partial\D_{r_1}(x_1)$ and $\partial\D_{r_2}(x_2)$ intersect iff $(r_1-r_2)^2\leq|x_1-x_2|^2\leq(r_1+r_2)^2$. In this case, $r_1=|z_0|\left(|z_0|^2-1\right)^{-1}\left|\frac{2}{cz_0}-1\right|^{-1}$, $x_1=-\left(|z_0|^2-1\right)^{-1}\left(\frac{2}{cz_0}-1\right)^{-1}$ and $r_2=|z_0|\left(|z_0|^2-1\right)^{-1}$, $x_2=\left(|z_0|^2-1\right)^{-1}$. Thus
$$|z_0|^2\dfrac{(|cz_0|-|cz_0-2|)^2}{(|z_0|^2-1)^2|cz_0-2|^2}\leq\dfrac{4}{(|z_0|^2-1)^2|cz_0-2|^2}\leq|z_0|^2\dfrac{(|cz_0|+|cz_0-2|)^2}{(|z_0|^2-1)^2|cz_0-2|^2},$$
%which simplifies to
%$$(|cz_0|-|cz_0-2|)^2\leq\dfrac{4}{|z_0|^2}\leq(|cz_0|+|cz_0-2|)^2.$$
In particular, $\varphi$ is univalent iff $|z_0|>1$ and
$$2<\left||cz_0|-|cz_0-2|\right||z_0|\;\;\text{or}\;\;\left||cz_0|+|cz_0-2|\right||z_0|<2$$
However, $\left||cz_0|+|cz_0-2|\right||z_0|<2$ implies $|z_0|<1$, and $2<\left||cz_0|-|cz_0-2|\right||z_0|$ implies $|z_0|>1$, so we may conclude that $\varphi$ is univalent iff
\begin{equation}\label{eqn:OnePtQDUnivalenceCriterion}
    2<\left||cz_0|-|cz_0-2|\right||z_0|.
\end{equation}

\;\\\\

\begin{figure}[ht]
  \centering
    \includegraphics[scale=.43]{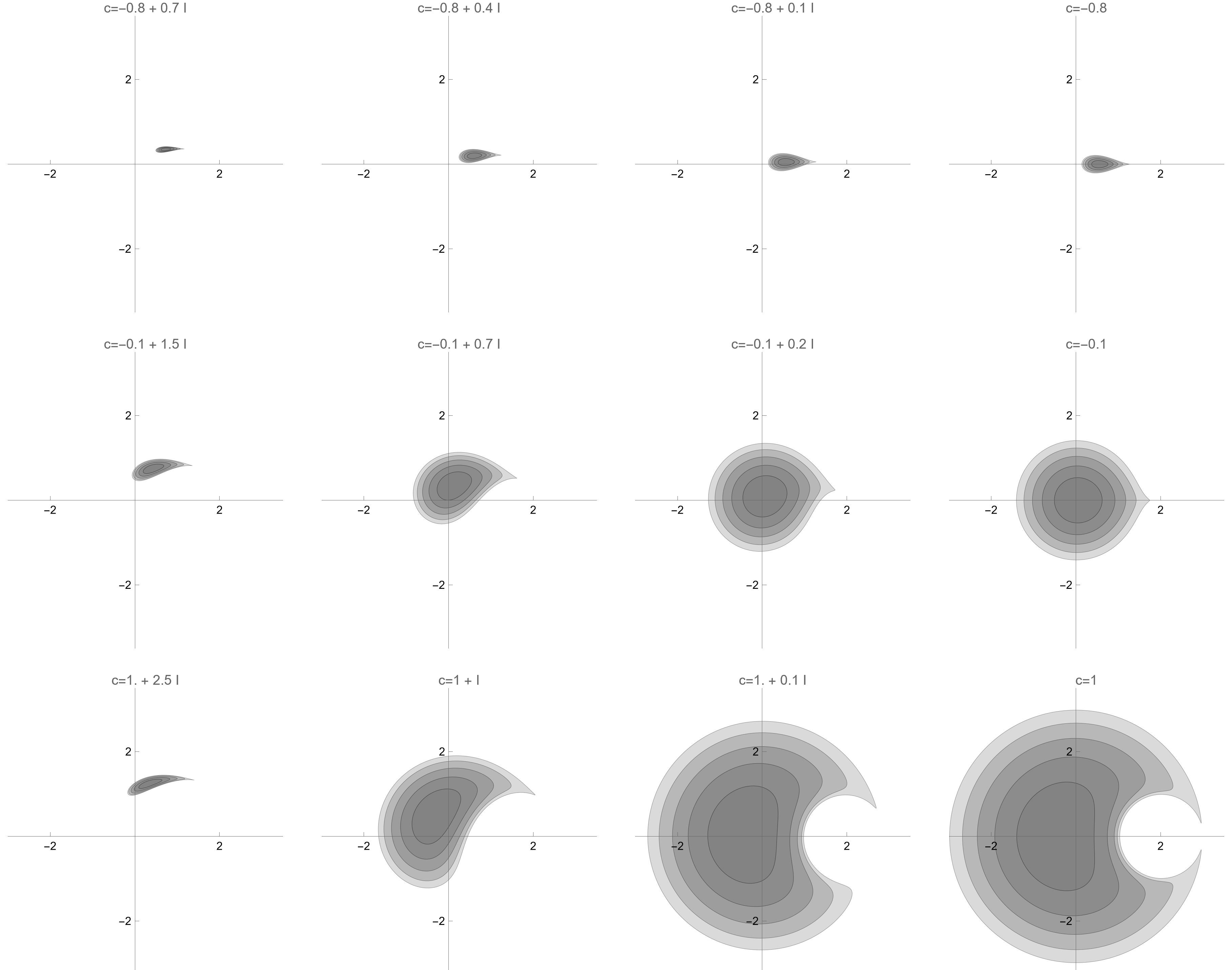}
    \caption{$\Omega_t\in\QD\left(\frac{\alpha}{w-2}\right)$ (complements of shaded) for various values of $\alpha$, and increasing $t$ up to the critical value. Note that the critical value is decreasing as $|w_0|^2+2\Re(\alpha)-2|\alpha|$ approaches $0$.}\label{fig:OnePtUQDFams}\vspace{1.5em}
\end{figure}
\FloatBarrier

\section{Power-Weighted Quadrature Domains}\label{sec:PQDs}
We will begin our study of weighted quadrature domains by discussing the instructive example of {\it power-weighted quadrature domains} (PQDs). Informally, a PQD is a domain $\Omega\subset\Ch$ if it admits a quadrature identity in which the integral over the domain is taken with respect to the weight $\rho_a(w)=|w|^{2(a-1)}$ for some $a>0$. The choice of $\rho_a$ is motivated by the correspondence between PQDs and local droplets of potentials of the form
$$Q(w)=\frac{|w|^{2a}}{a^2}-2\Re(H(w)).$$ The exact nature of this correspondence is detailed in Lemmas \ref{lemma:WQDGCE}, \ref{lemma:GenComplementDropletQuad}, and \ref{lemma:GenQuadComplementDroplet}. Before a deeper study of PQDs, we begin with a brief discussion of weighted quadrature domains more generally.

\subsection{Weighted Quadrature Domains}
As the name suggests, a $\rho-$weighted quadrature domain is a domain $\Omega\subset\Ch$ satisfying a quadrature identity in which the area integral is taken with respect to the weight $\rho$. Consider $\rho(w)=|w|^4=\frac{1}{4}\Delta\frac{|w|^{6}}{9}$ as a first example.
\begin{example}\label{ex:WQDFirstex}
If $c>0$, then the unbounded domain $\Omega_c=\{w\in\Ch:|w^3-1|>c^3\}$ is a WQD with respect to the weight $\rho(w)=|w|^4$. To see this, note that $|w^3-1|^2\dEquals c^6$, so $\overline{w}^3\dEquals\frac{c^6}{w^3-1}+1$. Thus for each $f\in\A_0(\Omega)$,
\begin{align*}
\int_{\Omega}f(w)|w|^4dA(w)&=\oint_{\partial\Omega}f(w)\frac{w^2}{3}\overline{w}^3dw\\
&=\oint_{\partial\Omega}f(w)\frac{1}{3}w^2\left(\frac{c^6}{w^3-1}+1\right)dw
\end{align*}
Applying the residue theorem and the fact that the third roots of unity are not in $\Omega$, we find
\begin{align*}
\int_{\Omega}f(w)|w|^4dA(w)&=-\frac{f_3}{3},
\end{align*}
where $f(w)=f_1w^{-1}+f_2w^{-2}+f_3w^{-3}+O(w^{-3})$.

We observe in Figure \ref{fig:3-3Droplets} that the boundary of $\Omega$ forms corners at the origin. Here, the angle is $\frac{\pi}{3}$. This is a special case of a more general phenomenon: the angle of intersection of the boundary of a PQD with the origin is observed to be an integer multiple of $\frac{\pi}{a}$ whenever $a>\frac{1}{2}$.
\begin{figure}[ht]
  \centering\vspace{.2em}
    \includegraphics[scale=.3]{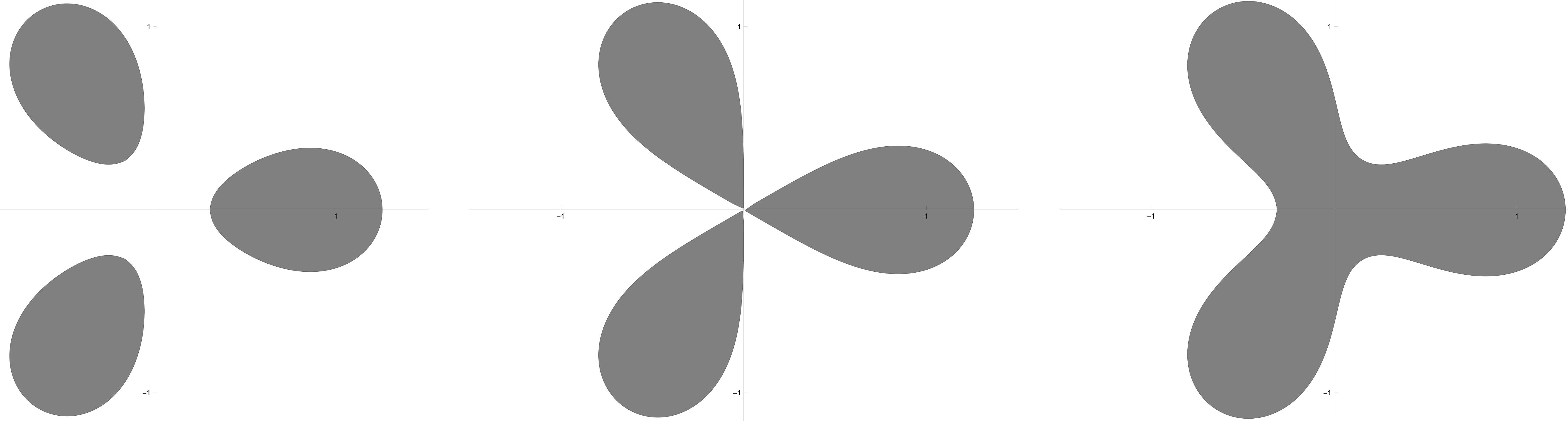}
    \caption{$\Omega\in\QD_{|w|^{4}}\left(\frac{w^{2}}{3};\A_0(\Omega)\right)$ (the complements of the shaded regions) for $c=.95,1,1.05$ transitioning from multiply connected to simply connected (left to right)}\label{fig:3-3Droplets}\vspace{1.5em}
\end{figure}
\end{example}

As for classical quadrature domains, the existence of a {\it weighted quadrature identity} is equivalent to the existence of a rational function $h$ for which Equation \ref{eqn:WQDContourQID} holds over some analytic test class $\mathcal{F}$ (e.g $h(w)=\frac{w^2}{3}$ in example \ref{ex:WQDFirstex}). This motivates the following definition.
\begin{definition}[Weighted quadrature domain]\label{def:rhoWQD}
Let $\Omega\subset\Ch$ be a domain equal to the interior of its closure and $\rho$ a weight function on $\Omega$. $\Omega$ is a $\rho-$weighted quadrature domain with respect to the $\rho-$integrable analytic test class $\mathcal{F}$ iff there exists a rational function $h$ for which
\begin{equation}\label{eqn:WQDContourQID}
    \int_{\Omega}f\rho dA=\oint_{\partial\Omega}f(w)h(w)dw
\end{equation}
holds for all $f\in\mathcal{F}$. This is denoted by $\Omega\in\QD_{\rho}(h;\mathcal{F})$.
\end{definition}
While this definition is very general, one can still recover analogs of many of the results for classical QDs - at least in certain special cases. The following lemma provides an analogue of the Schwarz function equation (\ref{eqn:QuadCoincidence}) for a broad class of radially symmetric weights.
\begin{lemma}\label{lemma:WQDGCE}
Let $\Omega\in \QD_\rho(h;\mathcal{F})$ be a bounded domain where $\rho=\frac{\partial^2F}{\partial w\partial\overline{w}}=\frac{1}{4}\Delta F$ for some real valued $F\in C^2(\Cl(\Omega))$, radially symmetric with respect to the origin. If $\A(\Omega)\subseteq\mathcal{F}$, then there exists $G\in\A(\Omega)$ such that
\begin{equation}\label{eqn:WeightedSFE}
h(w)\dEquals\dfrac{\partial F}{\partial w}+G(w).
\end{equation}
The same holds when $\Omega$ is unbounded, with each $\A(\Omega)$ above replaced by $\A_0(\Omega)$.
\end{lemma}

\begin{proof}[Proof of Lemma \ref{lemma:WQDGCE}]
Let $\Omega\in \QD_\rho(h;\mathcal{F})$ be a bounded domain with $\rho=\frac{1}{4}\Delta F$, where $F\in C^2(\Cl(\Omega))$ is radially symmetric with respect to the origin.
By assumption if $f\in\A(\Omega)$ then $f\in\mathcal{F}$, so applying the quadrature identity and Green's theorem, we obtain
\begin{align*}
    \oint_{\partial\Omega}f(w)\left(h(w)-\dfrac{\partial F}{\partial w}\right)dw&=\oint_{\partial\Omega}f(w)h(w)dw-\oint_{\partial\Omega}f(w)\dfrac{\partial F}{\partial w}dw\\
    &=\int_{\Omega}f\rho dA-\int_{\Omega}f\rho dA=0
\end{align*}
Therefore there exists $G\in\A(\Omega)$ such that $h\dEquals\frac{\partial F}{\partial w}+G$. The proof in the unbounded case is entirely analogous.
\end{proof}

This generalized Schwarz function equation naturally leads to the notion of a {\it generalized Schwarz function}, $S_{\Omega,F}:=h-G\dEquals\frac{\partial F}{\partial w}\in\M(\Omega)$. For example if $\Omega$ is a bounded domain in $\QD_{|w|^{2(a-1)}}(h;\A(\Omega))$, so that $F(w)=\frac{|w|^{2a}}{a^2}$ for some $a>0$, then the lemma implies $S_{\Omega,a^{-2}|w|^{2a}}(w):=h(w)-G(w)\dEquals a^{-1}\overline{w}|w|^{2(a-1)}$ for some $G\in\A(\Omega)$ (see Equation \ref{eqn:PQDCoincidence}). In the remainder of this section, we explore various families of PQDs and their properties.

\subsection{What is a PQD?}

A power-weighted quadrature domains behave very much like classical QDs away from the origin. However, the singularity of $\rho_a(w)=|w|^{2(a-1)}$ at the origin gives rise to behavior which is excluded in the classical case (by Sakai regularity, e.g. Theorem 3.2 of \cite{Lee_2015}): the development of ``corners'' on the boundary (e.g. in Figures \ref{fig:3-3Droplets}, \ref{fig:15_PQD_Nested_Family_PostCrit}, \ref{fig:PQDOnePtUnboundedTable}, and \ref{fig:.4-1-1o3c_Droplets_Finite_All}). It appears that this phenomenon isn't unique to PQDs, but is generic among WQDs with a singular weight.

As in the classical case, the choice of test class $\mathcal{F}$ is of great importance when it comes to the uniqueness of the quadrature function. For example if $\Omega$ were unbounded then for each $\alpha\in\C$,
$$\oint_{\partial\Omega}f(w)\left(\alpha+h(w)\right)dw=\oint_{\partial\Omega}f(w)h(w)dw$$
whenever $f\in L_a^1(\Omega;\rho_a)$ and $a\geq1$. That is, $\QD_{\rho_a}(h;L_a^1(\Omega;\rho_a))=\QD_{\rho_a}(h+\alpha;L_a^1(\Omega;\rho_a))$ for all $\alpha\in\C$. This isn't the case, however, when $\mathcal{F}=\A_0(\Omega)$
%\footnote{The $\rho-$weighted integral of $f\in\A_0(\Omega)$ remains finite for unbounded domains when considered in Cauchy principal value: $$\lim_{R\to\infty}\int_{(\D_R)\IntComp}f(w)|w|^{2(a-1)}dA(w)=\lim_{R\to\infty}\frac{1}{\pi}\int_{R}^{\infty}r^{2a-1}\int_{-\pi}^{\pi}f(re^{i\theta})d\theta dr=2\lim_{R\to\infty}\int_{R}^{\infty}r^{2a-1}\oint_{\partial\D_r}\dfrac{f(w)}{w}dwdr=0$$}
because if $w\in\Omega\IntComp$ then $\xi\mapsto(\xi-w)^{-1}\in\A_0(\Omega)$, so
$$\oint_{\partial\Omega}\dfrac{\alpha+h(\xi)}{\xi-w}d\xi=\alpha+\oint_{\partial\Omega}\dfrac{h(\xi)}{\xi-w}d\xi\neq\oint_{\partial\Omega}\dfrac{h(\xi)}{\xi-w}d\xi.$$
Applying Lemma \ref{lemma:WQDGCE} with $F(w)=\frac{|w|^{2a}}{a^2}$, we find that there exists $G\in\A(\Omega)$ ($\A_0(\Omega)$ when $\Omega$ is unbounded) for which
\begin{equation}\label{eqn:PQDCoincidence}
h(w)\dEquals \dfrac{1}{a}\overline{w}|w|^{2(a-1)}+G(w).
\end{equation}

Note this implies that when $\Omega$ is bounded (resp. unbounded), $h$ is uniquely associated to $\Omega$ if we restrict to the class $\Rat_0(\Omega)$ (resp. $\Rat(\Omega)$): If $\Omega$ is a PQD with $\mathcal{F}=\A(\Omega)$ (resp. $\A_0(\Omega)$) two such quadrature functions $h_1,h_2$ then
$$h_1(w)-h_2(w)\dEquals\left(\dfrac{1}{a}\overline{w}|w|^{2(a-1)}+G_1(w)\right)-\left(\dfrac{1}{a}\overline{w}|w|^{2(a-1)}+G_2(w)\right)=G_1(w)-G_2(w).$$
Hence $h_1-h_2$ continues analytically into $\Omega$. So when $\Omega$ is bounded, $h_1-h_2\in\Rat_0(\Omega)\cap\A(\Omega)=\{0\}$, and when $\Omega$ is unbounded, $h_1-h_2\in\Rat(\Omega)\cap\A_0(\Omega)=\{0\}$, so $h_1=h_2$. This motivates the following definition.
\begin{definition}[Power-weighted quadrature domain]
Take $a>0$ and let $\Omega$ be a bounded (resp. unbounded) domain equal to the interior of its closure. We say that $\Omega$ is a PQD iff there exists $h\in\Rat_0(\Omega)$ (resp. $\Rat(\Omega)$) for which
\begin{equation}\label{eqn:PQDContourQID}
\int_{\Omega}f\rho_a dA=\oint_{\partial\Omega}f(w)h(w)dw
\end{equation}
is satisfied for all $f\in\A(\Omega)$ (resp. $A_0(\Omega)$). This is denoted by $\Omega\in\QD_a(h)$.
\end{definition}
In the language of general weighted quadrature domains, this is equivalent to the statement that $\Omega\in\QD_{\rho_a}(h;\A(\Omega))$ with $h\in\Rat_0(\Omega)$ when $\Omega$ is bounded, and $\Omega\in\QD_{\rho_a}(h;\A_0(\Omega))$, $h\in\Rat(\Omega)$ when $\Omega$ is unbounded. We also denote by $\Omega\in\QD_a$ that there exists an $h$ for which $\Omega\in\QD_a(h)$. There are several equivalent characterizations of PQDs in the spirit of Lemma \ref{lemma:QDChars}:
\begin{theorem}\label{theorem:EquivPQDChars}
If $\Omega\subset\Ch$ is a domain equal to the interior of its closure and $a>0$ then the following are equivalent
\begin{enumerate}
    \item There exists a rational $h$ for which $\Omega\in\QD_a(h)$.
    \item There exists a $G\in\A(\Omega)$ ($\A_0$ when $\Omega$ is unbounded) and a rational $h$ for which $h(w)\dEquals\frac{1}{a}\overline{w}|w|^{2(a-1)}+G(w)$.
    \item There exists a function $S_a\in\M(\Omega)$ for which $S_a(w)\dEquals\frac{1}{a}\overline{w}|w|^{2(a-1)}$. This is the {\it generalized Schwarz function} associated to $\Omega$.
\end{enumerate}
Moreover in this case $h=\AnalyticIn{S_a}{\Omega\IntComp}$, and $G$ is minus the $\rho_a-$weighted Cauchy transform of $\Omega^c$,
\begin{equation}\label{eqn:PQDGCauchyTransFormula}
    G(w)=\int_{\Omega^c}\dfrac{|\xi|^{2(a-1)}}{\xi-w}dA(\xi).
\end{equation}
\end{theorem}
The generalized Schwarz function $S_a\in\M(\Omega)$ is unique when it exists as a consequence of its equality with $ \frac{1}{a}\overline{w}|w|^{2(a-1)}$ on $\partial\Omega$. We'll see in the next section that simply connected PQDs admit an additional characterization in terms of their Riemann map.
\begin{proof}[Proof of Theorem \ref{theorem:EquivPQDChars}]\;\\
$\mathbf{(1)\implies(2)}$: This is immediate from Equation \ref{eqn:PQDCoincidence} (via Lemma \ref{lemma:WQDGCE}).\\
$\mathbf{(2)\implies(3)}$: Rearranging, we find that $\frac{1}{a}\overline{w}|w|^{2(a-1)}\dEquals h(w)-G(w)=:S_a(w)$, which is the generalized Schwarz function for $\Omega$ because $h-G\in\M(\Omega)$. Also, $\AnalyticInNoBracket{S_a}{\Omega\IntComp}=\AnalyticIn{h-G}{\Omega\IntComp}=h$.\\
$\mathbf{(3)\implies(1)}$: Finally, if $\Omega$ is a bounded domain for which such an $S_a$ exists, and $f\in\A(\Omega)$, then
\begin{align*}
\int_{\Omega}f(w)|w|^{2(a-1)}dA(w)&=\oint_{\partial\Omega}f(w)\frac{1}{a}\overline{w}|w|^{2(a-1)}dw=\oint_{\partial\Omega}f(w)S_a(w)dw\\
&=\oint_{\partial\Omega}f(w)\AnalyticIn{S_a(w)}{\Omega\IntComp}dw=\oint_{\partial\Omega}f(w)h(w)dw,
\end{align*}
So $\Omega\in\QD_a(h)$. The argument in the unbounded case is completely analogous. 

Finally, we verify Equation \ref{eqn:PQDGCauchyTransFormula}: First, if $\Omega$ is unbounded then for each $w\in\Omega$,
\begin{align*}
G(w)&=\oint_{\partial\Omega}\dfrac{G(\xi)d\xi}{\xi-w}=-\oint_{\partial\Omega\IntComp}\dfrac{h(\xi)-\frac{1}{a}\overline{\xi}|\xi|^{2(a-1)}}{\xi-w}d\xi=\oint_{\partial\Omega\IntComp}\dfrac{\frac{1}{a}\overline{\xi}|\xi|^{2(a-1)}}{\xi-w}d\xi=\int_{\Omega^c}\dfrac{|\xi|^{2(a-1)}}{\xi-w}dA(\xi).
\end{align*}
The first equality is from Cauchy's integral formula, the second follows from part (2) of the lemma, the third follows from the fact that $h\in\A(\Omega\IntComp)$, and the final equality is Green's formula.
On the other hand, if $\Omega$ is bounded then $\Omega\IntComp$ is unbounded, so we consider the integral in terms of its Cauchy principal value. In particular, we have by Green's theorem that for each $w\in\Omega$,
\begin{align*}
\int_{\Omega^c}\dfrac{|\xi|^{2(a-1)}}{\xi-w}dA(\xi)&:=\lim_{r\to\infty}\int_{\Omega^c\cap\D_r}\dfrac{|\xi|^{2(a-1)}}{\xi-w}dA(\xi)=\lim_{r\to\infty}\oint_{\partial(\Omega^c\cap\D_r)}\dfrac{\frac{1}{a}\overline{\xi}|\xi|^{2(a-1)}}{\xi-w}d\xi
\end{align*}
Integrating over each component of the boundary separately, and applying the fact that $\frac{1}{a}\overline{\xi}|\xi|^{2(a-1)}=h(\xi)-G(\xi)$ on $\partial\Omega$, we find
$$\int_{\Omega^c}\dfrac{|\xi|^{2(a-1)}}{\xi-w}dA(\xi)=\oint_{\partial\Omega\IntComp}\dfrac{h(\xi)-G(\xi)}{\xi-w}d\xi+\lim_{r\to\infty}\dfrac{r^{2a}}{a}\oint_{\partial\D_r}\dfrac{\xi^{-1}}{\xi-w}d\xi.$$
Applying Cauchy's theorem and that $h\in\A_0(\Omega)$, we obtain
$$\int_{\Omega^c}\dfrac{|\xi|^{2(a-1)}}{\xi-w}dA(\xi)=\oint_{\partial\Omega}\dfrac{G(\xi)}{\xi-w}d\xi+\lim_{r\to\infty}\dfrac{r^{2a}}{a}\cdot0=G(w).$$
\end{proof}

\subsection{Boundary Regularity of PQDs}
Recall that the boundary regularity of $\Omega\in\QD$ (Theorem \ref{thm:QDBoundaryRegularity}) was a direct consequence of the existence of an associated Schwarz function and Sakai's regularity theorem \ref{theorem:SakaiRegularity}, which characterizes the types of singularities which can occur on a curve admitting a local Schwarz function. Unfortunately, Sakai's result can't be applied directly to PQDs, as they don't generally admit a Schwarz function. However, a recent result of Vardakis and Volberg \cite{vardakis2022freeboundaryproblemssakais} (Corollary 2.5) sufficiently generalizes Sakai's result for our purposes.
\begin{lemma}\label{lemma:GeneralizedSchwarzToLocal}
    Let $\Omega$ be a bounded domain, $p$ a polynomial, $F$ a function analytic in a neighborhood of $\Cl(\Omega)$, and $T\in\A(\Omega)$. Suppose that for all $w\in\partial\Omega$, we have
    $$T(w)=p(\overline{w})F(w).$$
    Then for each $w_0\in\partial\Omega$ such that $p'(\overline{w_0})\neq0$, there exists $\delta>0$ such that $p^{-1}\left(\frac{T}{F}\right)$ is a local Schwarz function at $w_0$.
\end{lemma}
In particular, if $a\in\Z_{+}$ and $\Omega\in\QD_a(h)$ is a bounded PQD, Theorem \ref{theorem:EquivPQDChars} tells us that there exists a $G\in\A(\Omega)$ such that $a(h(w)-G(w))=\overline{w}^aw^{a-1}$. If we let $r(w)=\prod_{k}(w-a_k)^{n_k}$ (where the $a_k$ are the poles of $h$ and $n_k$ their respective orders), then
$$p(\overline{w})F(w):=\overline{w}^a(w^{a-1}r(w))\dEquals ar(w)(h(w)-G(w))=:T(w),$$
where $T$ is analytic in $\Omega$, $F(w):=w^{a-1}r(w)$ is analytic in a neighborhood of $\Cl(\Omega)$, and $p(w)=w^{a-1}$ is a polynomial. This choice of $T$, $p$, and $F$ satisfy the conditions of Lemma \ref{lemma:GeneralizedSchwarzToLocal}, and we find that $\Omega$ admits a local Schwarz function at every point $w_0\in\partial\Omega$, except possibly at $w_0=0$. We conclude that if $\Omega$ is a bounded PQD then for each $\epsilon>0$, $\partial\Omega\setminus\D_\epsilon$ has finitely many singular points, each of which is either a cusp or a double point.
On the other hand, if $\Omega$ is an unbounded PQD, then choose $w_1\notin\Cl(\Omega)$ and consider the bounded domain $\widetilde{\Omega}=(\Omega-w_1)^{-1}$, so $\partial\widetilde{\Omega}$ satisfies
$$\dfrac{1}{h((w+w_1)^{-1})-G((w+w_1)^{-1})}=a\overline{(w+w_1)}^{a}(w+w_1)^{a-1}$$
for some $G\in\A_0(\Omega)$. If we let $r(w)=\prod_{k}(w-a_k)^{n_k}$ (where the $a_k$ are the poles of the left hand side and $n_k$ their respective orders), then
$$p(\overline{w})F(w):=a(\overline{w}+\overline{w_1})^{a}((w+w_1)^{a-1}r(w))\dEquals\dfrac{r(w)}{h((w+w_1)^{-1})-G((w+w_1)^{-1})}=:T(w),$$
where $T$ is analytic in $\widetilde{\Omega}$, $F(w):=((w+w_1)^{a-1}r(w))$ is analytic in a neighborhood of $\Cl(\widetilde{\Omega})$, and $p(w)=a(w+\overline{w_1})^{a}$ is a polynomial. This choice of $T$, $p$, and $F$ satisfy the conditions of Lemma \ref{lemma:GeneralizedSchwarzToLocal}, and we find that $\widetilde{\Omega}$ admits a local Schwarz function at every point of $\partial\widetilde{\Omega}$, except possibly at $-w_1$. The map $w\mapsto(w-w_1)^{-1}$ is univalent in a neighborhood of $\partial\Omega$, so this implies that $\Omega$ admits a local Schwarz function at every point of $\partial\Omega$, except possibly at $0$. In summary,
\begin{lemma}\label{lemma:PQDBoundaryRegularityv1}
    If $\Omega$ is a PQD then for each $\epsilon>0$, $\partial\Omega\setminus\D_\epsilon$ has finitely many singular points, each of which is either a cusp or a double point.
\end{lemma}

To constrain the behavior at the origin, we require an extra ingredient, that $(\partial\Omega)^a$ is a subset of an algebraic curve:

\begin{theorem}\label{thm:PQDAlgebraic}
If $\Omega\in\QD_a$ for some $a\in\Z_{+}$, there exists a non-constant polynomial $p$, irreducible over $\C$, with real coefficients such that $p(x,y)=0$ for all $x+i y\in\partial\Omega^a$. (c.f. \cite{AharonovShapiro}, Theorem 3). In particular, $\partial\Omega$ is a subset of the $a$th root of an algebraic curve.
\end{theorem}

\begin{proof}[Proof of Theorem \ref{thm:PQDAlgebraic}]
By Theorem \ref{theorem:EquivPQDChars}, $\Omega$ admits a generalized Schwarz function, $S_a$ which is analytic in $\Omega$ except at finitely many poles such that $T(w):=\frac{S_a(w)}{\frac{1}{a}w^{a-1}}\dEquals\overline{w}^a$. Thus
$$f(w):=\frac{1}{2}(w^a+T(w)),\text{ and }g(w):=\frac{1}{2i}(w^a-T(w))$$
are analytic except at finitely many poles, and admit continuous real extensions to $\partial\Omega$, $f(w)\dEquals\Re(w^a)$, $g(w)\dEquals\Im(w^a)$. Now suppose $\Omega$ is bounded. We apply Theorem 2 of \cite{AharonovShapiro}:

{\it Let $\Omega\subset\C$ be a bounded domain and suppose $f$ and $g$ are holomorphic in $\Omega$ except at finitely many poles. Further suppose that $f$ and $g$ admit real, continuous, extensions to $\partial\Omega$. Then there exists a non-trivial polynomial $p(x,y)$, irreducible over $\C$, with real coefficients such that $p(f(z),g(z))=0$ on $\partial\Omega$.}

Therefore there exists a real non-constant, $\C$-irreducible polynomial $p$ such that $p(f(w),g(w))=0$ on $\partial\Omega$. Thus, $p(\Re(w^a),\Im(w^a))=0$ on $\partial\Omega$. This is equivalent to $p(x,y)=0$ for all $x+iy\in\partial\Omega^a$. The argument is similar, but slightly more delicate in the unbounded case; see the proof of Theorem 3 in \cite{AharonovShapiro}.
\end{proof}

Theorem \ref{thm:PQDAlgebraic} straightforwardly implies that $\partial\Omega$ has finitely many singular points, so we conclude

\begin{theorem}\label{theorem:PQDBoundaryRegularityv2}
    If $\Omega\in\QD_a$ for some $a\in\Z_{+}$ then $\partial\Omega$ has finitely many singular points, where each non-zero singular point is either a cusp or a double point.
\end{theorem}

\subsection{Simply Connected Power-Weighted Quadrature Domains with a=2}\label{subsec:2PQDFTMethod}

As a first step towards generalizing Theorems \ref{thm:ClassicBQDRationalRiemannIffQD} and \ref{thm:ClassicUQDRationalRiemannIffQD} to the power-weighted setting, it is instructive to consider the case in which $a=2$ and $0\notin\Cl(\Omega)$. We further divide our discussion into, first, the case of unbounded PQDs and, second, the case of bounded PQDs.

\subsubsection{Unbounded PQDs}\label{subsubsec:2UPQDFTMethod}
Suppose that $\Omega\in\QD_2$ is a simply connected unbounded PQD and quadrature function $h$ and Riemann map $\varphi:\D\IntComp\rightarrow\Omega$. Further suppose that $0\notin\Cl(\Omega)$. Composing $\varphi$ with the Schwarz function equation for PQDs, $h(w)\dEquals\frac{1}{2}w\overline{w}^2+G(w)$ (Theorem \ref{theorem:EquivPQDChars}), we obtain
$$\overline{\varphi}^2\dEquals \left(\frac{h(w)-G(w)}{\frac{1}{2}w}\right)\circ\varphi$$
Expanding $\varphi$ in series, $\varphi(z)=cz+f_0+\left<z^{-1}\right>$, we see that
\begin{align*}
    \varphi^2(z)&=c^2z^2+2cf_0z+(f_0^2+2cf_1)+\left<z^{-1}\right>\\
    &=\mathring{W_2}(z)+(f_0^2+2cf_1)+\left<z^{-1}\right>
\end{align*}
where $\mathring{W_2}(z):=W_2(z)-W_2(0)$ and $W_2$ is the second inverse Faber polynomial (Equation \ref{eqn:FaberPolyFormulae}) so, in the disk, $(\varphi^2)^{\#}=\mathring{W_2}^{\#}+(f_0^2+2cf_1)+\left<z^{-1}\right>$. Therefore $(\varphi^2)^{\#}-\mathring{W_2}^{\#}\in\A(\D)$ which implies
$$(\varphi^2)^{\#}-\mathring{W_2}^{\#}=\AnalyticIn{(\varphi^2)^{\#}-\mathring{W_2}^{\#}}{\D}=\AnalyticIn{(\varphi^2)^{\#}}{\D}=\AnalyticIn{\overline{\varphi}^2}{\D},$$
where the last line follows from the fact that $\overline{\varphi}\dEquals\varphi^{\#}$. Rearranging and substituting the Schwarz function equation, we obtain
\begin{align*}
    (\varphi^2)^{\#}&=\mathring{W_2}^{\#}+\AnalyticIn{\overline{\varphi}^2}{\D}=\mathring{W_2}^{\#}+\AnalyticIn{\left(\frac{h(w)-G(w)}{\frac{1}{2}w}\right)\circ\varphi}{\D}.
\end{align*}
Equation \ref{eqn:ExteriorFaberTransformProjectionExtension} (using the boundary regularity from Theorem \ref{theorem:PQDBoundaryRegularityv2}) implies
$$\AnalyticIn{\left(\frac{h(w)-G(w)}{\frac{1}{2}w}\right)\circ\varphi}{\D}=\Phi_{\varphi}^{-1}\left(\AnalyticIn{\frac{h(w)-G(w)}{\frac{1}{2}w}}{\Omega\IntComp}\right)=\Phi_{\varphi}^{-1}\left(2\frac{h(w)-h(0)}{w}\right).$$
Hence, we obtain an implicit solution to the inverse problem,
\begin{equation}\label{eqn:2UPQDRiemannMapFormula}
    \varphi^2=\mathring{W_2}+2\Phi_{\varphi}^{-1}\left(\dfrac{h(w)-h(0)}{w}\right)^{\#}.
\end{equation}
Moreover, as the Faber transform of a rational function is rational, we conclude that
\begin{corollary}
    If $0\notin\Omega$ is a simply connected unbounded domain then $\Omega\in\QD_2$ iff the square of its Riemann map extends to a rational function.
\end{corollary}
The reverse direction follows from the fact that $\frac{w}{2}(\varphi^2)^\#\circ\psi(w)$ is a generalized Schwarz function for $\Omega$ (Theorem \ref{theorem:EquivPQDChars}). Reflecting Equation \ref{eqn:2UPQDRiemannMapFormula} and applying the Faber transform to both sides of the equation, yields a partial solution to the direct problem,
\begin{equation}
    \dfrac{h(w)-h(0)}{w}=\dfrac{1}{2}\Phi_{\varphi}\left(\left(\varphi^2-\mathring{W_2}\right)^{\#}\right)(w).
\end{equation}
Note that this leaves $h(0)$ undetermined.

Another, more direct, formula for $h$ can be obtained by composing $\varphi$ with a different arrangement of the Schwarz function equation, $\frac{1}{2}\varphi\overline{\varphi}^2\dEquals(h-G)\circ\varphi$. Taking the Faber transform of this (after projecting onto the part analytic in $\D$) and applying Equation \ref{eqn:ExteriorFaberTransformProjectionExtension} yields
\begin{align*}
    \Phi_\varphi\left(\AnalyticIn{\frac{1}{2}\varphi\overline{\varphi}^2}{\D}\right)&=\Phi_\varphi\left(\AnalyticIn{\left(h-G\right)\circ\varphi}{\D}\right)=\AnalyticIn{h-G}{\Omega\IntComp}=h
\end{align*}
Noting that $\overline{\varphi}\dEquals\varphi^{\#}$ and again applying Equation \ref{eqn:ExteriorFaberTransformProjectionExtension}, we obtain a solution to the direct problem
\begin{equation}\label{eqn:2UPQDHFormula1}
    h(w)=\Phi_\varphi\left(\AnalyticIn{\dfrac{1}{2}\varphi(\varphi^2)^{\#}}{\D}\right)(w)=\AnalyticIn{\dfrac{1}{2}w(\varphi^2)^{\#}\circ\psi(w)}{\Omega\IntComp} 
\end{equation}

Finally, consider the arrangement of the Schwarz function equation, $\frac{1}{2}\varphi^2\overline{\varphi}^2\dEquals(w(h(w)-G(w)))\circ\varphi$. Taking the Faber transform of this (after projecting onto the part analytic in $\D$) and applying Equation \ref{eqn:ExteriorFaberTransformProjectionExtension} yields
\begin{align*}
    \Phi_\varphi\left(\AnalyticIn{\frac{1}{2}\varphi^2\overline{\varphi}^2}{\D}\right)&=\Phi_\varphi\left(\AnalyticIn{\left(w(h(w)-G(w))\right)\circ\varphi}{\D}\right)=\AnalyticIn{wh(w)-wG(w)}{\Omega\IntComp}=wh(w)+\Res{\infty}G.
\end{align*}
One can show via Equation \ref{eqn:PQDGCauchyTransFormula} that $\Res{\infty}G=t:=\int_{\Omega^c}|w|^{2(a-1)}dA(w)$, the weighted area of $\Omega^c$. Hence,
\begin{equation}
    h(w)=\dfrac{1}{w}\Phi_\varphi\left(\AnalyticIn{\dfrac{1}{2}\varphi^2(\varphi^2)^{\#}}{\D}\right)(w)-\dfrac{t}{w}=\dfrac{1}{w}\AnalyticIn{\dfrac{1}{2}w^2(\varphi^2)^{\#}\circ\psi(w)}{\Omega\IntComp}-\dfrac{t}{w}
\end{equation}

\paragraph{Linear PQDs Not Containing Zero}\label{par:NZLinearPQDs}
We'll now apply these results to characterize the set of $\Omega\in\QD_2$ having a linear quadrature function. The conclusion of this analysis is summarized in the following theorem.
\begin{theorem}
Fix $\alpha_0\in\C$, $\alpha_1\in\C\setminus\{0\}$, and a simply connected domain $0\notin\Omega\in\QD_2(\alpha_0+\alpha_1w)$. Then $\Omega$ is unique modulo conformal radius unless $\alpha_0=0$, in which case $\pm\Omega$ are the only such domains.
\end{theorem}

Let $\Omega$ be a simply connected domain not containing zero such that $\Omega\in\QD_2(h)$, where $h(w)=\alpha_0+w$ for some $\alpha_0\in\C$ (by change of variables, we can wlog assume $\alpha_1=1$ as long as $\alpha_1\neq0$). Applying Equation \ref{eqn:2UPQDRiemannMapFormula}, we find that
\begin{align*}
    \varphi^2(z)&=\mathring{W_2}^{\#}(z)+2\Phi_{\varphi}^{-1}(1)=c^2z^2+2cf_0z+2.
\end{align*}
To solve for the undetermined coefficient $f_0$, we apply Equation \ref{eqn:2UPQDHFormula1},
\begin{align*}
    \alpha_0+w&=h(w)=\AnalyticIn{\dfrac{1}{2}w\left(\dfrac{c^2}{\psi^2(w)}+\dfrac{2c\overline{f_0}}{\psi(w)}+2\right)}{\Omega\IntComp}=\AnalyticIn{w+c^2\overline{f_0}+O(w^{-1})}{\Omega\IntComp}=w+c^2\overline{f_0}.
\end{align*}
Hence, $f_0=c^{-2}\overline{\alpha_0}$, and we find
\begin{align*}
    \varphi^2(z)&=c^2z^2+2c^{-1}\overline{\alpha_0}z+2.
\end{align*}
In particular, if $\Omega\in\QD_2(\alpha_0+w)$ is a simply connected domain not containing $0$ and $\varphi$ is the associated Riemann map (uniquely associated via the normalization $\varphi(\infty)=\infty$, $\varphi'(\infty)=c>0$), then $\varphi$ satisfies the above equation in $\D\IntComp$.

To determine uniqueness, suppose there were two such domains $\Omega_1$ and $\Omega_2$ with equal conformal radius $c>0$ and associated Riemann maps $\varphi_1$ and $\varphi_2$. The prior equation tells us that $\varphi_1^2(z)=\varphi_2^2(z)$ for all $z\in\D\IntComp$, so $1=\left(\frac{\varphi_1(z)}{\varphi_2(z)}\right)^2$, which implies $\frac{\varphi_1(z)}{\varphi_2(z)}=\pm1$ for all $z\in\D\IntComp$. As $\varphi_1$ and $\varphi_2$ are non-zero and analytic in $\D\IntComp$, we conclude that $\varphi_2=\pm\varphi_1$. In particular, $\Omega_2=\pm\Omega_1$. If $\Omega_2=-\Omega_1$ then, applying the quadrature identity to the Cauchy kernel, we find that for each $w\in(\Omega_2)\IntComp$,
\begin{align*}
    h(w)&=\oint_{\partial\Omega_2}\dfrac{h(\xi)d\xi}{\xi-w}=\int_{\Omega_2}\dfrac{|\xi|^2dA(\xi)}{\xi-w}=-\int_{\Omega_1}\dfrac{|\xi|^2dA(\xi)}{\xi-(-w)}=-\oint_{\partial\Omega_1}\dfrac{h(\xi)d\xi}{\xi-(-w)}=-h(-w)
\end{align*}

Later on we will discuss these domains in the more general case of positive real $a$ and when $0\in\Omega$.

\subsubsection{Simply Connected Bounded PQDs with a=2}\label{subsubsec:2BPQDFTMethod}
Suppose that $\Omega\in\QD_2(h)$ is simply connected and bounded with Riemann map $\varphi:\D\rightarrow\Omega$. Further suppose that $0\notin\Cl(\Omega)$. Composing $\varphi$ with the Schwarz function equation for PQDs (Theorem \ref{theorem:EquivPQDChars}) $h(w)\dEquals\frac{1}{2}w\overline{w}^2+G(w)$ (for some $G\in\A(\Omega)$), we obtain $\overline{\varphi}^2\dEquals \left(\frac{h(w)-G(w)}{\frac{1}{2}w}\right)\circ\varphi$. Expanding $\varphi$ in series, $\varphi(z)=\varphi(0)+\left<z\right>$, we see that $(\varphi^2)^{\#}=\left(\varphi^2(0)+\left<z\right>\right)^{\#}=\overline{\varphi^2(0)}+\left<z^{-1}\right>$. Therefore $(\varphi^2)^{\#}-\overline{\varphi^2(0)}\in\A_0(\D\IntComp)$ which implies
$$(\varphi^2)^{\#}-\overline{\varphi^2(0)}=\AnalyticIn{(\varphi^2)^{\#}-\overline{\varphi^2(0)}}{\D\IntComp}=\AnalyticIn{(\varphi^2)^{\#}}{\D\IntComp}=\AnalyticIn{\overline{\varphi}^2}{\D\IntComp},$$
where the last line follows from the fact that $\overline{\varphi}\dEquals\varphi^{\#}$. Rearranging and substituting the Schwarz function equation, we obtain
\begin{align*}
    (\varphi^2)^{\#}&=\overline{\varphi^2(0)}+\AnalyticIn{\overline{\varphi}^2}{\D\IntComp}=\overline{\varphi^2(0)}+\AnalyticIn{\left(\frac{h(w)-G(w)}{\frac{1}{2}w}\right)\circ\varphi}{\D\IntComp}.
\end{align*}
Equation \ref{eqn:InteriorFaberTransformProjectionExtension} (using the boundary regularity from Theorem \ref{theorem:PQDBoundaryRegularityv2}) implies
$$\AnalyticIn{\left(\frac{h(w)-G(w)}{\frac{1}{2}w}\right)\circ\varphi}{\D\IntComp}=\Phi_{\varphi}^{-1}\left(\AnalyticIn{\frac{h(w)-G(w)}{\frac{1}{2}w}}{\Omega\IntComp}\right)=\Phi_{\varphi}^{-1}\left(2\frac{h(w)-h(0)}{w}\right).$$
Hence, we obtain an implicit solution to the inverse problem,
\begin{equation}
    \varphi^2=\varphi^2(0)+2\Phi_{\varphi}^{-1}\left(\dfrac{h(w)-h(0)}{w}\right)^{\#}.
\end{equation}
Reflecting and applying the Faber transform to both sides of the equation, yields a partial solution to the direct problem,
\begin{equation}
    \dfrac{h(w)-h(0)}{w}=\dfrac{1}{2}\Phi_{\varphi}\left((\varphi^2)^{\#}-\overline{\varphi^2(0)}\right)(w).
\end{equation}
Moreover, the constraint $h(\infty)=0$ implies $h(0)$ is equal to the residue at $\infty$ of the right hand side.\\

Another, more direct, formula for $h$ can be obtained by composing $\varphi$ with a different arrangement of the Schwarz function equation, $\frac{1}{2}\varphi\overline{\varphi}^2\dEquals(h-G)\circ\varphi$. Taking the Faber transform of this (after projecting onto the part analytic in $\D\IntComp$) and applying Equation \ref{eqn:InteriorFaberTransformProjectionExtension} yields
\begin{align*}
    \Phi_\varphi\left(\AnalyticIn{\frac{1}{2}\varphi\overline{\varphi}^2}{\D\IntComp}\right)&=\Phi_\varphi\left(\AnalyticIn{\left(h-G\right)\circ\varphi}{\D\IntComp}\right)\\
    &=\AnalyticIn{h-G}{\Omega\IntComp}=h
\end{align*}
Noting that $\overline{\varphi}\dEquals\varphi^{\#}$ and again applying Equation \ref{eqn:InteriorFaberTransformProjectionExtension}, we obtain a solution to the direct problem
\begin{equation}
    h(w)=\Phi_\varphi\left(\AnalyticIn{\dfrac{1}{2}\varphi(\varphi^2)^{\#}}{\D\IntComp}\right)(w)=\AnalyticIn{\dfrac{1}{2}w(\varphi^2)^{\#}\circ\psi(w)}{\Omega\IntComp} 
\end{equation}

$\frac{1}{2}\varphi^2\overline{\varphi}^2\dEquals(wh(w)-wG(w))\circ\varphi$. Taking the Faber transform of this (after projecting onto the part analytic in $\D\IntComp$) and applying Equation \ref{eqn:InteriorFaberTransformProjectionExtension} yields
\begin{align*}
    \Phi_\varphi\left(\AnalyticIn{\frac{1}{2}\varphi^2\overline{\varphi}^2}{\D\IntComp}\right)&=\Phi_\varphi\left(\AnalyticIn{\left(wh(w)-wG(w)\right)\circ\varphi}{\D\IntComp}\right)\\
    &=\AnalyticIn{wh(w)-wG(w)}{\Omega\IntComp}\\
    &=wh(w)+\Res{\infty}h.
\end{align*}
By applying the quadrature identity to the constant function shows $-\Res{\infty}h=t:=\int_{\Omega}|w|^{2(a-1)}dA(w)$, the weighted area of $\Omega$. Hence,
\begin{equation}
    h(w)=\dfrac{1}{w}\Phi_\varphi\left(\AnalyticIn{\dfrac{1}{2}\varphi^2(\varphi^2)^{\#}}{\D\IntComp}\right)(w)+\dfrac{t}{w}=\dfrac{1}{w}\AnalyticIn{\dfrac{1}{2}w^2(\varphi^2)^{\#}\circ\psi(w)}{\Omega\IntComp}+\dfrac{t}{w}
\end{equation}

\;\\

We are now prepared to consider two different generalizations of the Riemann mapping characterization and formulae for simply connected classical QDs (Theorems \ref{thm:ClassicBQDRationalRiemannIffQD} and \ref{thm:ClassicUQDRationalRiemannIffQD}) to PQDs for arbitrary $a>0$. In particular, we exhibit Faber transform formulae in terms of both sum (Mittag-Leffler) and product (inner-outer/Nevanlinna) decompositions of the Riemann map $\varphi$.

Each of these approaches come with their own benefits and drawbacks. The sum representation yields relatively simple formulae, but only applies for integer values of $a$. On the other hand, the product representation yields more specific and unwieldy formulae, but applies for all $a>0$. We will begin with a discussion of the simpler sum representation.

\subsection{Faber Sum Identities for PQDs}\label{subsec:PQDAddFTMethod}

Fix $a\in\Z_{+}$ and let $\Omega$ be a simply connected unbounded PQD with quadrature function $h$ and Riemann map $\varphi:\D\IntComp\rightarrow\Omega$. Note that $\varphi^a(z)$ can be rewritten as $p+g$, where $p\in\A(\D)$ is a polynomial of degree $a$ and $g\in\A_0(\D\IntComp)$. Recalling the definition of Faber polynomials (Equation \ref{eqn:FTFormulae}), we find that $p$ is, in fact, the $a$th inverse Faber polynomial $W_a$,
\begin{align*}
p&=\AnalyticIn{p}{\D}=\AnalyticIn{\varphi^{a}-g}{\D}=\AnalyticIn{\varphi^{a}}{\D}=W_a.
\end{align*}
As a consequence, $\AnalyticIn{(\varphi^{a})^{\#}}{\D}=(\varphi^{a})^{\#}-\mathring{W_a}^{\#}\in\A(\D)$, where $\mathring{W_a}:=W_a-W_a(0)$, so we may compute its Faber transform:
\begin{align*}
    \Phi_{\varphi}\left((\varphi^{a})^{\#}-\mathring{W_a}^{\#}\right)&=\Phi_{\varphi}\left(\AnalyticIn{(\varphi^{a})^{\#}}{\D}\right).
\end{align*}
Note that $(\varphi^{a})^{\#}\dEquals\overline{\varphi^{a}}$, so rearranging Equation \ref{eqn:PQDCoincidence} and substituting, we obtain $$(\varphi^{a})^{\#}\dEquals\left(\dfrac{h(w)-G(w)}{\frac{1}{a}w^{a-1}}\right)\circ\varphi,$$
for some $G\in\A_0(\Omega)$. Then, as $\AnalyticInNoBracket{}{\D}$ and the Faber transform only respect the boundary values of their argument, we may substitute to obtain
\begin{align*}
    \Phi_{\varphi}\left((\varphi^{a})^{\#}-\mathring{W_a}^{\#}\right)&=\Phi_{\varphi}\left(\AnalyticIn{\left(\dfrac{h(w)-G(w)}{\frac{1}{a}w^{a-1}}\right)\circ\varphi}{\D}\right).
\end{align*}
We recognize the argument of the RHS as the inverse Faber transform. Applying Equation \ref{eqn:ExteriorFaberTransformProjectionExtension} yields
\begin{align*}
    \Phi_{\varphi}\left((\varphi^{a})^{\#}-\mathring{W_a}^{\#}\right)(w)&=\AnalyticIn{\left(\dfrac{h(w)-G(w)}{\frac{1}{a}w^{a-1}}\right)\circ\varphi\circ\psi(w)}{\Omega\IntComp}\\
    &=\AnalyticIn{\dfrac{h(w)-G(w)}{\frac{1}{a}w^{a-1}}}{\Omega\IntComp}.
\end{align*}
Applying the inverse Faber transform to both sides of the above equation and rearranging, we obtain
\begin{equation}
    \varphi^{a}(z)=\mathring{W_a}(z)+\Phi_{\varphi}^{-1}\left(\AnalyticIn{\dfrac{h(w)-G(w)}{\frac{1}{a}w^{a-1}}}{\Omega\IntComp}\right)^{\#}(z).
\end{equation}
Furthermore, the Faber transform of a meromorphic function is rational, so we may conclude that $\varphi^a\in\Rat(\D)$. This proves the backward direction of the first of the two main theorems of this section:
\begin{theorem}\label{thm:GenUPQDLefflerFTFormula}
Take $a\in\Z_{+}$ and let $\Omega$ be an unbounded simply connected domain. Then there exists $r\in\Rat(\D\IntComp)$ for which $\varphi^a(z)=W_a(z)-W_a(0)+r^{\#}(z)$ is a Riemann map for $\Omega$ if and only if $\Omega\in\QD_a(h)$ for some $h\in\Rat(\Omega)$.
\end{theorem}

\begin{theorem}\label{thm:GenBPQDLefflerFTFormula}
Take $a\in\Z_{+}$ and let $\Omega$ be a bounded simply connected domain. Then there exists $r\in\Rat_0(\D)$ for which $\varphi^a(z)=\varphi^{a}(0)+r^{\#}(z)$ is a Riemann map for $\Omega$ if and only if $\Omega\in\QD_a(h)$ for some $h\in\Rat_0(\Omega)$.
\end{theorem}
In the course of the proof of Theorem \ref{thm:GenBPQDLefflerFTFormula}, we also obtain the following pair of equations relating $h$ and $\varphi$.
\begin{corollary}\label{cor:GenPQDLefflerFTFormulae}
Whenever $r$ and $h$, as in Theorems \ref{thm:GenUPQDLefflerFTFormula} and \ref{thm:GenBPQDLefflerFTFormula}, exist, they satisfy the following pair of identities
\begin{equation}\label{eqn:GenPQDLefflerFTFormulae}
\begin{aligned}
r(z)=\Phi_{\varphi}^{-1}\left(\AnalyticIn{\dfrac{h(w)-G(w)}{\frac{1}{a}w^{a-1}}}{\Omega\IntComp}\right)(z),\tab\tab h(w)=\AnalyticIn{\dfrac{1}{a}w^{a-1}(\varphi^a)^{\#}\circ\psi(w)}{\Omega\IntComp}.
\end{aligned}
\end{equation}
Where $G\in\A(\Omega)$ when $\Omega$ is bounded and $G\in\A_0(\Omega)$ when $\Omega$ is unbounded.
\end{corollary}
Note that the $G(w)$ term in the left equation drops out whenever $0\notin\Omega$. We now prove Theorem \ref{thm:GenBPQDLefflerFTFormula} and Corollary \ref{cor:GenPQDLefflerFTFormulae}. The proof of Theorem \ref{thm:GenUPQDLefflerFTFormula} will be omitted, as it is analogous.

\begin{proof}[Proof of Theorem \ref{thm:GenBPQDLefflerFTFormula} and Corollary \ref{cor:GenPQDLefflerFTFormulae}]\;\\
For the forward direction, suppose there exists $r\in\Rat_0(\D)$ for which $\varphi^{a}(z)=\varphi^{a}(0)+r^{\#}(z)$. Then note that for $w\in\partial\Omega$,
\begin{align*}
    \dfrac{1}{a}\overline{w}|w|^{2(a-1)}&=\dfrac{1}{a}w^{a-1}\overline{w}^a=\dfrac{1}{a}w^{a-1}\overline{\varphi^a\circ\psi(w)}\\
    &=\dfrac{1}{a}w^{a-1}\overline{\left(\varphi^{a}(0)+r^{\#}\circ\psi(w)\right)}\\
    &=\dfrac{1}{a}w^{a-1}\left(\overline{\varphi^{a}(0)}+r\circ\psi(w)\right)=:S_{a}(w)
\end{align*}
extends meromorphically to $\Omega$, so $S_a$ is a generalized Schwarz function for $\Omega$. Thus $\Omega\in\QD_a(h)$ for some $h\in\Rat_0(\Omega)$ by Theorem \ref{theorem:EquivPQDChars}.\\

For the reverse direction, suppose that $\Omega\in\QD_a(h)$ for some $h\in\Rat_0(\Omega)$. Then, as $\AnalyticIn{\overline{\varphi^{a}(0)}}{\D\IntComp}=0$ and $(\varphi^{a})^{\#}-\overline{\varphi^{a}(0)}\in\A_0(\D\IntComp)$, we find that $\AnalyticIn{(\varphi^{a})^{\#}}{\D\IntComp}=(\varphi^{a})^{\#}-\overline{\varphi^{a}(0)}$, so
\begin{align*}
    \Phi_{\varphi}\left((\varphi^{a})^{\#}-\overline{\varphi^{a}(0)}\right)&=\Phi_{\varphi}\left(\AnalyticIn{(\varphi^{a})^{\#}}{\D\IntComp}\right).
\end{align*}
Rearranging Equation \ref{eqn:PQDCoincidence} and substituting as we did in the bounded case, we obtain
\begin{align*}
    \Phi_{\varphi}\left((\varphi^{a})^{\#}-\overline{\varphi^{a}(0)}\right)&=\Phi_{\varphi}\left(\AnalyticIn{\left(\dfrac{h(w)-G(w)}{\frac{1}{a}w^{a-1}}\right)\circ\varphi}{\D\IntComp}\right).
\end{align*}
We recognize the argument of the RHS as the inverse Faber transform. Applying Equation \ref{eqn:InteriorFaberTransformProjectionExtension} yields
\begin{align*}
    \Phi_{\varphi}\left((\varphi^{a})^{\#}-\overline{\varphi^{a}(0)}\right)(w)&=\AnalyticIn{\left(\dfrac{h(w)-G(w)}{\frac{1}{a}w^{a-1}}\right)\circ\varphi\circ\psi(w)}{\Omega\IntComp}=\AnalyticIn{\dfrac{h(w)-G(w)}{\frac{1}{a}w^{a-1}}}{\Omega\IntComp}.
\end{align*}
Taking the inverse Faber transform of both sides and noting that the RHS is meromorphic in $\Omega$, so that its inverse Faber transform is rational, we obtain the desired result.\\

Finally, we verify the formula for the quadrature function. By Theorem \ref{theorem:EquivPQDChars}, $h(w)\dEquals\frac{1}{a}w^{a-1}\overline{w}^{a}+G(w)$ for some $G\in\A(\Omega)$, so
$$h(w)=\AnalyticIn{h(w)}{\Omega^{\ast}}=\AnalyticIn{\frac{1}{a}w^{a-1}\overline{w}^{a}+G(w)}{\Omega^{\ast}}=\AnalyticIn{\frac{1}{a}w^{a-1}\overline{w}^{a}}{\Omega^{\ast}},$$
where the first equality follows from the fact that $h\in\A_0(\Omega_{\ast})$ and the last from the fact that $G\in\A(\Omega)$. Thus, noting that $\overline{w}^{a}\dEquals(\varphi^a)^{\#}\circ\psi(w)$ we find that $h(w)=\AnalyticIn{\frac{1}{a}w^{a-1}(\varphi^a)^{\#}\circ\psi(w)}{\Omega^{\ast}}$. The argument for unbounded PQDs is analogous.
\end{proof}

As a demonstration of the utility of this approach, we consider a simple one point PQD for $a=2$.
\subsubsection{One Point PQDs with a=2 not Containing Zero}\label{subsubsec:NZOnePtPQDEx}

Let $0\notin\Omega\in\QD_2(h)$ be a simply connected bounded domain with $h(w)=\frac{\alpha}{w-1}$ for some $\alpha>0$. Let $\varphi:\D\rightarrow\Omega$ be the unique Riemann map associated to $\Omega$ for which $\varphi(0)=1$ and $\varphi'(0)>0$. Then by Theorem \ref{thm:GenBPQDLefflerFTFormula} and corollary \ref{cor:GenPQDLefflerFTFormulae}, there exists $r\in\Rat_0(\D)$ for which $\varphi^2=1+r^{\#}$ is a Riemann map for $\Omega$ and $r$ is given by Equation \ref{eqn:GenPQDLefflerFTFormulae} for some $G\in\A(\Omega)$. Thus,
\begin{align*}
    r(z)&=\Phi_{\varphi}^{-1}\left(\AnalyticIn{2\dfrac{\frac{\alpha}{w-1}-G(w)}{w}}{\Omega\IntComp}\right)(z)=2\Phi_{\varphi}^{-1}\left(\AnalyticIn{\dfrac{\alpha}{w-1}-\dfrac{\alpha+G(w)}{w}}{\Omega\IntComp}\right)(z)\\
    &=2\alpha\Phi_{\varphi}^{-1}\left(\dfrac{1}{w-1}\right)(z)=\dfrac{2\alpha\psi'(1)}{z-\psi(1)},
\end{align*}
where the last equality follows from Equation \ref{eqn:FTFormulae}. Also $\psi(1)=\varphi^{-1}(1)=0$ and $\psi'(1)=\varphi'(0)^{-1}$, so $r(z)=\frac{2\alpha}{\varphi'(0)z}$, and we find $\varphi^2(z)=1+\frac{2\alpha}{\varphi'(0)}z$. Thus
$$2\varphi'(0)=2\varphi(0)\varphi'(0)=(\varphi^2)'(0)=\frac{2\alpha}{\varphi'(0)},$$
so $\varphi'(0)=\sqrt{\alpha}$, and we conclude
$$\varphi^2(z)=1+2\sqrt{\alpha}z.$$
In summary, 

\begin{theorem}
If $\alpha>0$ and $\Omega\in\QD_2\left(\frac{\alpha}{w-1}\right)$ is a simply connected domain not containing $0$, then $\Omega$ is the unique connected component of the square root of $\D_{2\sqrt{\alpha}}(1)$ containing $1$.
\end{theorem}

\noindent In \S\ref{subsec:OnePtPQDs} we provide a classification of simply connected one point PQDs more generally.

\subsection{Faber Product Identities for PQDs}\label{subsec:PQDMultFTMethod}
As the title of this subsection suggests, this portion of the paper is concerned with formulae obtained by decomposing the Riemann map into a product. Out of a host of possible factorizations, those using {\it Blaschke factors} proved by far the most fruitful. For each $\lambda\in\D$, the associated Blaschke factor is defined by
$$b_{\lambda}(z):=\frac{\overline{\lambda}}{|\lambda|}\frac{z-\lambda}{\overline{\lambda}z-1}.$$
(also $b_0(z):=z$ by standard convention). $b_\lambda$ satisfies the following important properties:
\begin{enumerate}
    \item $|b_\lambda(z)|=1$ for all $z\in\partial\D$,
    \item $b_\lambda^{\#}(z)=b_{\lambda}(z)^{-1}$ for all $z\in\C$,
    \item $\lim_{|\lambda|\to1}b_{\lambda}(z)=1$,
    \item $\lim_{|z|\to\infty}b_{\lambda}(z)=|\lambda|^{-1}>0$.
\end{enumerate}
When $|\lambda|>1$, we also call $b_\lambda$ a Blaschke factor (in the exterior disk).

Let $\Omega$ be a simply connected domain and $\varphi$ the associated Riemann map (with the standard normalization). We define the function $\varphi_{\rm out}$ as follows. If $\infty\notin\Omega$ then
\begin{equation}\label{eqn:BoundedOuterFactorization}
\begin{alignedat}{2}
    \varphi(z)&=\varphi_{\rm out}(z),\quad&&\text{when }0\notin\Omega,\\
    \varphi(z)&=b_{z_0}(z)\varphi_{\rm out}(z),\quad&&\text{when }0\in\Omega,
\end{alignedat}
\end{equation}
where $z_0\in\D$ is the unique root of $\varphi$. On the other hand, if $\infty\in\Omega$ then
\begin{equation}\label{eqn:UnboundedOuterFactorization}
\begin{alignedat}{2}
    \varphi(z)&=z\varphi_{\rm out}(z),\quad&&\text{when }0\notin\Omega,\\
    \varphi(z)&=zb_{z_0}(z)\varphi_{\rm out}(z),\quad&&\text{when }0\in\Omega,
\end{alignedat}
\end{equation}

where $z_0\in\D\IntComp$ is the unique root of $\varphi$. The formulae (\ref{eqn:BoundedOuterFactorization}) and (\ref{eqn:UnboundedOuterFactorization}) give the inner-outer factorization of $\varphi$. It is trivial to see that $\varphi$ belongs to the meromorphic Nevanlinna class, which consists of ratios of bounded holomorphic functions. Thus $\varphi_{\rm out}$ is exactly the {\it outer part} of $\varphi$ in the sense of the Nevanlinna theory. The key property of $\varphi_{\rm out}$ for our purposes is the fact that it is (by construction) finite, non-zero, and analytic, and hence $\varphi_{\rm out}^a$ is single-valued and analytic. 

On the other hand, we denote by $\varphi_{\rm in}:=\frac{\varphi}{\varphi_{\rm out}}$ the {\it inner part} of $\varphi$. Equations \ref{eqn:BoundedOuterFactorization} and \ref{eqn:UnboundedOuterFactorization} yield by inspection
\begin{equation}\label{eqn:PQDInnerFactors}
\begin{alignedat}{2}
    \varphi_{\rm in}(z)&=1,\quad&&\text{when }\infty\notin\Omega,\;0\notin\Omega,\\
    \varphi_{\rm in}(z)&=b_{z_0}(z),\quad&&\text{when }\infty\notin\Omega,\;0\in\Omega,\\
    \varphi_{\rm in}(z)&=z,\quad&&\text{when }\infty\in\Omega,\;0\notin\Omega,\\
    \varphi_{\rm in}(z)&=zb_{z_0}(z),\quad&&\text{when }\infty\in\Omega,\;0\in\Omega.
\end{alignedat}
\end{equation}
In particular, $\varphi_{\rm in}$ is a product of Blaschke factors (and their reciprocals). This leads to the key property of $\varphi_{\rm in}$ for our purposes: that $|\varphi_{\rm in}(z)|=1$ for all $z\in\partial\D$ and $\varphi_{\rm in}^{\#}(z)=\varphi_{\rm in}(z)^{-1}$ for all $z\in\C$.\\

For the remainder of this discussion, fix $a>0$ and let $\Omega$ be a simply connected domain with Riemann map $\varphi=\varphi_{\rm in}\varphi_{\rm out}$ (unless otherwise stated). In this case, we obtain the following simple characterization of simply connected PQDs.
\begin{theorem}\label{thm:SCPQDCharacterization}
$\Omega\in\QD_a$ iff $\varphi^a_{\rm out}$ extends to a rational function.
\end{theorem}

\begin{proof}[Proof of Theorem \ref{thm:SCPQDCharacterization}]
For the reverse direction, suppose that $\varphi_{\rm out}^a$ extends to a rational function. Then, as $|\varphi_{\rm in}|\dEquals1$, we find that $\left|\varphi\right|^{2a}=\left|\varphi_{\rm in}\right|^{2a}\left|\varphi_{\rm out}\right|^{2a}\dEquals\left|\varphi_{\rm out}^a\right|^2\dEquals \varphi_{\rm out}^a(\varphi_{\rm out}^a)^{\#}$. Hence,
\begin{align*}
\dfrac{1}{a}\overline{w}|w|^{2(a-1)}&=\dfrac{\left|\varphi\right|^{2a}\circ\psi(w)}{a w}\dEquals\dfrac{\left((\varphi_{\rm out}^a)^{\#}\varphi_{\rm out}^a\right)\circ\psi(w)}{aw}=:S_a(w)
\end{align*}
is meromorphic in $\Omega$, and thus a generalized Schwarz function associated to $\Omega$ and $a$. We conclude by Theorem \ref{theorem:EquivPQDChars} that $\Omega\in\QD_a$.

Now we only need demonstrate that $\Omega\in\QD_a$ implies $\varphi_{\rm out}^a$ extends to a rational function. Note that it is necessary and sufficient to show that $r:=(\varphi_{\rm out}^{a})^{\#}=\left(\frac{\varphi^{\#}}{\varphi_{\rm in}^{\#}}\right)^a$ extends to a rational function. By construction, $r$ is finite, non-zero, analytic (defined in $\D\IntComp$ if $\Omega$ is bounded and $\D$ if $\Omega$ is unbounded), and extends continuously to the boundary. Hence for each $w\in\partial\Omega$,
\begin{align*}
    r\circ\psi(w)&=\left(\dfrac{\varphi^{\#}\circ\psi(w)}{\varphi_{\rm in}^{\#}\circ\psi(w)}\right)^{a}\\
    &=\left(\overline{w}\varphi_{\rm in}\circ\psi(w)\right)^{a}\\
    &=aw\left(\dfrac{1}{a}\overline{w}|w|^{2(a-1)}\right)\left(\dfrac{\varphi_{\rm in}\circ\psi(w)}{w}\right)^{a}
\end{align*}
If $\Omega$ is bounded then, by the definition of interior the Faber transform, we have $\Phi_{\varphi}\left(\AnalyticInNoBracket{r}{\D\IntComp}\right)(z)=\AnalyticIn{r\circ\psi(w)}{\Omega\IntComp}$. Moreover by Theorem \ref{theorem:EquivPQDChars}, $\frac{1}{a}\overline{w}|w|^{2(a-1)}\dEquals h(w)-G(w)$. Hence, by Equation \ref{eqn:InteriorFaberTransformProjectionExtension},
\begin{align*}
\Phi_{\varphi}\left(\AnalyticInNoBracket{r}{\D\IntComp}\right)(w)&=\AnalyticIn{aw(h(w)-G(w))\left(\dfrac{\varphi_{\rm in}\circ\psi(w)}{w}\right)^{a}}{\Omega\IntComp}.
\end{align*}
Setting $Y(w):=a\left(\frac{\varphi_{\rm in}\circ\psi(w)}{w}\right)^{a}\in\A(\Omega)$, we have $wG(w)Y(w)\in\A(\Omega)$ so
\begin{align*}
\Phi_{\varphi}\left(\AnalyticInNoBracket{r}{\D\IntComp}\right)(w)&=\AnalyticIn{wh(w)Y(w)}{\Omega\IntComp}-\AnalyticIn{wG(w)Y(w)}{\Omega\IntComp}=\AnalyticIn{wh(w)Y(w)}{\Omega\IntComp}.
\end{align*}
(where the last equality follows from Theorem \ref{theorem:EquivPQDChars}). As $r\in\A(\D\IntComp)$ is bounded, $r-r(\infty)\in\A_0(\D\IntComp)$, so $\AnalyticInNoBracket{r}{\D\IntComp}=r-r(\infty)$. Applying the inverse Faber transform, we obtain
\begin{align*}
r&=r(\infty)+\Phi_{\varphi}^{-1}\left(\AnalyticIn{wh(w)Y(w)}{\Omega\IntComp}\right).
\end{align*}
The argument of the inverse Faber transform here is in $\M(\Omega)$, so its transform is rational, and we may conclude that $r$ is rational.

On the other hand, if $\Omega$ is unbounded then by reasoning analogous to the bounded case, we find that
\begin{align*}
\Phi_{\varphi}\left(r\right)(w)&=\AnalyticIn{w(h(w)-G(w))Y(w)}{\Omega\IntComp}.
\end{align*}
As $wG(w)Y(w)\in\A(\Omega)$, $G(w)=O(w^{-1})$ and $Y(w)=O(1)$ about $\infty$, we find that
\begin{align*}
\Phi_{\varphi}\left(r\right)(w)&=\AnalyticIn{wh(w)Y(w)}{\Omega\IntComp}-\AnalyticIn{wG(w)Y(w)}{\Omega\IntComp}=\AnalyticIn{wh(w)Y(w)}{\Omega\IntComp}+C
\end{align*}
for some $C\in\C$. Applying the inverse Faber transform, we obtain
\begin{align*}
r&=\Phi_{\varphi}^{-1}\left(\AnalyticIn{wh(w)Y(w)}{\Omega\IntComp}\right)+C
\end{align*}
The argument of the inverse Faber transform here is in $\M(\Omega)$, so its transform is rational, and we conclude that $r$ is rational. The desired conclusion follows.
\end{proof}
It follows immediately that if $\Omega\in\QD_a$ then there exists a rational function $r$ such that
\begin{equation}\label{eqn:PQDRiemannMapProdFactorization}
    \varphi(z)=\varphi_{\rm in}(z)r^{\#}(z)^{\frac{1}{a}}.
\end{equation}

We furthermore obtain a formula for $h$ in terms of this representation of $\varphi$, hence solving the direct problem for PQDs. In particular if $\Omega\in\QD_a(h)$ then by Theorem \ref{theorem:EquivPQDChars},
$$h(w)\dEquals\dfrac{|w|^{2a}}{aw}+G(w)=\dfrac{|\varphi|^{2a}\circ\psi(w)}{aw}+G(w)\dEquals\dfrac{\left(rr^{\#}\right)\circ\psi(w)}{aw}+G(w).$$
Noting that $\AnalyticIn{G}{\Omega\IntComp}=0$ and $\AnalyticIn{h}{\Omega\IntComp}=h$, we find
\begin{equation}\label{eqn:PQDHFormulaOld}
    h(w)=\dfrac{1}{a}\AnalyticIn{\dfrac{\left(rr^{\#}\right)\circ\psi(w)}{w}}{\Omega\IntComp}
\end{equation}

More generally, in the course of the proof of Theorem \ref{thm:SCPQDCharacterization}, we obtained an implicit solution to the inverse problem for PQDs. In the discussion to follow we sharpen this result, first for bounded PQDs, then for unbounded PQDs. By similar reasoning, we also obtain another solution to the direct problem for PQDs.

\subsubsection{Bounded PQDs}

\begin{theorem}\label{thm:GenBPQDInvProbFormula}
If $\Omega\in\QD_a(h)$ with Riemann map $\varphi$ as in Equation \ref{eqn:PQDRiemannMapProdFactorization} then
\begin{equation}\label{eqn:GenBPQDInvProbFormula}
    r=a\Phi_{\varphi}^{-1}\left(\AnalyticIn{wh(w)\left(\dfrac{\varphi_{\rm in}\circ\psi(w)}{w}\right)^{a}}{\Omega\IntComp}\right)+C,
\end{equation}
where $C=\overline{\varphi^a(0)}$ when $0\notin\Omega$ and $C=\frac{\overline{\varphi^a(0)}}{|z_0|^a}$ when $0\in\Omega$.
\end{theorem}
The proof of Theorem \ref{thm:SCPQDCharacterization} already tells us that Equation \ref{eqn:GenBPQDInvProbFormula} holds for some $C\in\C$. In this case, the value of $C$ is easily obtained by evaluating of $\varphi$ at $0$. We can obtain by similar analysis a solution to the direct problem.

\begin{theorem}\label{thm:BPQDDirectProblemSol}
If $\Omega\in\QD_a(h)$ with Riemann map $\varphi$ as in Equation \ref{eqn:PQDRiemannMapProdFactorization} then
\begin{equation}\label{eqn:BPQDDirectProblemSol}
h(w)=\dfrac{1}{a w}\Phi_{\varphi}\left(\AnalyticIn{rr^{\#}}{\D\IntComp}\right)(w)+\dfrac{t}{w},
\end{equation}
where $t=\int_{\Omega}|w|^{2(a-1)}dA(w)$.
\end{theorem}
\begin{proof}[Proof of Theorem \ref{thm:BPQDDirectProblemSol}]
Recall from the proof of Theorem \ref{thm:SCPQDCharacterization} that $\frac{(rr^{\#})\circ\psi(w)}{aw}=\frac{\left((\varphi_{\rm out}^a)^{\#}\varphi_{\rm out}^a\right)\circ\psi(w)}{aw}:=S_a(w)$ is a generalized Schwarz function for $\Omega$. Thus, by Theorem \ref{theorem:EquivPQDChars}, $h(w)\dEquals\frac{(rr^{\#})\circ\psi(w)}{aw}+G(w)$, where $h$ is the quadrature function and $G\in\A(\Omega)$. Hence, by the definition of the Faber transform,
\begin{align*}
    \AnalyticIn{wh(w)}{\Omega\IntComp}&=\dfrac{1}{a}\AnalyticIn{(rr^{\#})\circ\psi(w)}{\Omega\IntComp}+\AnalyticIn{wG(w)}{\Omega\IntComp}=\Phi_{\varphi}\left(\AnalyticIn{rr^{\#}}{\D\IntComp}\right)(w).
\end{align*}
Moreover, note that if $w\in\Omega\IntComp$, then
\begin{align*}
\AnalyticIn{wh(w)}{\Omega\IntComp}&=\oint_{\partial\Omega\IntComp}\dfrac{\xi h(\xi)}{\xi-w}d\xi=-\oint_{\partial\Omega}h(\xi)d\xi+w\oint_{\partial\Omega\IntComp}\dfrac{h(\xi)}{\xi-w}d\xi
\end{align*}
Applying the quadrature identity to the first term, and Cauchy's integral formula to the second, we obtain
\begin{align*}
\AnalyticIn{wh(w)}{\Omega\IntComp}&=-\int_{\Omega}|\xi|^{2(a-1)}dA(\xi)+wh(w)
\end{align*}
Thus,
\begin{align*}
    wh(w)&=\AnalyticIn{wh(w)}{\Omega\IntComp}+\int_{\Omega}|\xi|^{2(a-1)}dA(\xi)=\Phi_{\varphi}\left(\AnalyticIn{rr^{\#}}{\D\IntComp}\right)(w)+t.
\end{align*}
\end{proof}

Applying the fact that the inverse Faber transform takes poles to their image under $\psi:=\varphi^{-1}$, we obtain the following corollary

\begin{corollary}\label{corr:GenBPQDFTFormula}
Let $h$ be a rational function with poles $\{p_k\}$ of respective orders $\{n_k\}$. Then the $r$ in Theorem \ref{thm:GenBPQDInvProbFormula} is given by
\begin{equation}\label{eqn:GenBPQDFTFormula}
r(z)=C+z^{-1}p^{\#}(z)+\sum_{k}\sum_{j=1}^{n_k}\dfrac{\beta_{k,j}}{\left(z-z_k\right)^{j}}.
\end{equation}
for some constants $\{\beta_{k,j}\}$, $\varphi(z_k)=p_k$ and the same value of $C$ as in the Theorem. If $\varphi(0)=p_j$ for some $j$, then the sum over $k$ skips $j$ and $p$ is a polynomial of degree:
$$\deg(p)=\begin{cases}
   n_j-1 & p_j\neq0\\
   n_j-2 & p_j=0
\end{cases}$$
$p=0$ otherwise.
\end{corollary}

\subsubsection{Unbounded PQDs}

\begin{theorem}\label{thm:GenUPQDInvProbFormula}
If $\Omega\in\QD_a(h)$ then it has a Riemann map $\varphi$ as in Equation \ref{eqn:PQDRiemannMapProdFactorization} with
\begin{equation}\label{eqn:GenUPQDInvProbFormula}
    r=a\Phi_{\varphi}^{-1}\left(\AnalyticIn{wh(w)\left(\dfrac{\varphi_{\rm in}\circ\psi(w)}{w}\right)^{a}}{\Omega\IntComp}\right)+C,
\end{equation}
where, $C=\frac{a}{c^a}t$ when $0\notin\Omega$ and $C=\frac{a}{|cz_0|^a}t$ when $0\in\Omega$, and $t=\int_{\Omega^c}|w|^{2(a-1)}dA(w)$.
\end{theorem}

\begin{proof}[Proof of Theorem \ref{thm:GenUPQDInvProbFormula}]
The proof of Theorem \ref{thm:SCPQDCharacterization} already tells us that Equation \ref{eqn:GenUPQDInvProbFormula} holds for some $C\in\C$. The proof also shows that 
$$C=-\AnalyticIn{wG(w)Y(w)}{\Omega\IntComp},$$
where $Y(w)=a\left(\frac{\varphi_{\rm in}\circ\psi(w)}{w}\right)^{a}$ is non-zero, bounded, and analytic in $\Omega$. Expanding Equation \ref{eqn:PQDGCauchyTransFormula} about $\infty$, we find that $wG(w)=-t+g(w)$, where $t=\int_{\Omega^c}|w|^{2(a-1)}dA(w)$ and $g\in\A_0(\Omega)$, so
\begin{align*}
    C&=t\AnalyticIn{Y(w)}{\Omega\IntComp}-\AnalyticIn{g(w)Y(w)}{\Omega\IntComp}=t\AnalyticIn{Y(w)}{\Omega\IntComp},
\end{align*}
because $g(w)Y(w)\in\A_0(\Omega)$. Finally note that, as $Y(w)$ is non-zero, bounded, and analytic in $\Omega$, $Y(w)-Y(\infty)\in\A_0(\Omega)$ so
\begin{align*}
    C&=t\AnalyticIn{Y(w)}{\Omega\IntComp}-\AnalyticIn{g(w)Y(w)}{\Omega\IntComp}=t\AnalyticIn{Y(\infty)}{\Omega\IntComp}+t\AnalyticIn{Y(w)-Y(\infty)}{\Omega\IntComp}=tY(\infty).
\end{align*}
The desired formula for $C$ follows.
\end{proof}

We can obtain by similar analysis a solution to the direct problem.

\begin{theorem}\label{thm:UPQDDirectProblemSol}
If $\Omega\in\QD_a(h)$ with Riemann map $\varphi$ as in Equation \ref{eqn:PQDRiemannMapProdFactorization} then
\begin{equation}\label{eqn:UPQDDirectProblemSol}
h(w)=\dfrac{1}{a w}\Phi_{\varphi}\left(\AnalyticIn{rr^{\#}}{\D}\right)(w)-\dfrac{t}{w},
\end{equation}
where $t=\int_{\Omega^c}|w|^{2(a-1)}dA(w)$.
\end{theorem}
\begin{proof}[Proof of Theorem \ref{thm:UPQDDirectProblemSol}]
Recall from the proof of Theorem \ref{thm:SCPQDCharacterization} that $\frac{(rr^{\#})\circ\psi(w)}{aw}:=S_a(w)$ is a generalized Schwarz function for $\Omega$. Thus, by Theorem \ref{theorem:EquivPQDChars}, $h(w)\dEquals\frac{(rr^{\#})\circ\psi(w)}{aw}+G(w)$, where $h$ is the quadrature function and $G\in\A_0(\Omega)$. Hence, by the definition of the Faber transform,
\begin{align*}
    wh(w)&=\AnalyticIn{wh(w)}{\Omega\IntComp}=\dfrac{1}{a}\AnalyticIn{(rr^{\#})\circ\psi(w)}{\Omega\IntComp}+\AnalyticIn{wG(w)}{\Omega\IntComp}=\Phi_{\varphi}\left(\AnalyticIn{rr^{\#}}{\D\IntComp}\right)(w)-t.
\end{align*}
Where the last equality follows from expansion of Equation \ref{eqn:PQDGCauchyTransFormula} about $\infty$, which demonstrates that $wG(w)=-t+g(w)$, where $t=\int_{\Omega^c}|w|^{2(a-1)}dA(w)$ and $g\in\A_0(\Omega)$, so $\AnalyticIn{wG(w)}{\Omega\IntComp}=\AnalyticIn{-t+g(w)}{\Omega\IntComp}=-t$.
\end{proof}

\begin{example}\label{ex:15_PQD_Nested_Family_PostCrit}
Consider the map $\varphi:\D\IntComp\rightarrow\Ch$ given by $\varphi(z)=cz\left(1-\frac{\gamma}{z}\right)^{\frac{1}{a}}$, where $|\gamma|<1$ and $a,c>0$. Note that $\Omega=:\varphi(\D\IntComp)$ is an unbounded domain not containing $0$ on account of the assumption that $|\gamma|<1$. Hence, when $\varphi$ is univalent, it is a Riemann map for $\varphi$ with $\varphi_{\rm out}^a(z)=r^{\#}(z)=c^a\left(1-\frac{\gamma}{z}\right)$ rational so Theorem \ref{thm:SCPQDCharacterization} tells us that $\Omega\in\QD_a$. We may now apply Theorem \ref{thm:UPQDDirectProblemSol} to solve the direct problem of finding the quadrature function associated to $\Omega$:
\begin{align*}
h(w)&=\dfrac{1}{aw}\Phi_{\varphi}\left(\AnalyticIn{c^{a}\left(1-\frac{\gamma}{z}\right)c^{a}\left(1-\overline{\gamma}z\right)}{\D}\right)(w)-\dfrac{t}{w}\\
&=\dfrac{c^{2a}}{aw}\Phi_{\varphi}\left(\AnalyticIn{1+|\gamma|^2-\frac{\gamma}{z}-\overline{\gamma}z}{\D}\right)(w)-\dfrac{t}{w}\\
&=\dfrac{c^{2a}}{aw}\Phi_{\varphi}\left(1+|\gamma|^2-\overline{\gamma}z\right)(w)-\dfrac{t}{w}\\
&=\dfrac{c^{2a}}{aw}\left(1+|\gamma|^2-\overline{\gamma}(wc^{-1}+f_0)\right)-\dfrac{t}{w}\\
&=-\dfrac{c^{2a-1}}{a}\overline{\gamma}+\dfrac{c_1}{w}.
\end{align*}
However, $h\in\A(\Omega\IntComp)$ and $0\in\Omega\IntComp$, so $c_1=0$ and we find that $\Omega\in\QD_a(\alpha)$, for $\gamma=-\frac{\overline{\alpha}a}{c^{2a-1}}$. We will characterize the univalence of $\varphi$ in \S \ref{subsec:BasicMonomialPQDClass}. Several examples of $\Omega$ are provided in Figure \ref{fig:15_PQD_Nested_Family_PostCrit}. Note that in the limit as $|\gamma|\to1$, $\partial\Omega$ approaches $0$, forming a ``corner'' there. 
\end{example}

\begin{figure}[ht]
  \centering
    \includegraphics[height=0.4\linewidth,width=.4\linewidth]{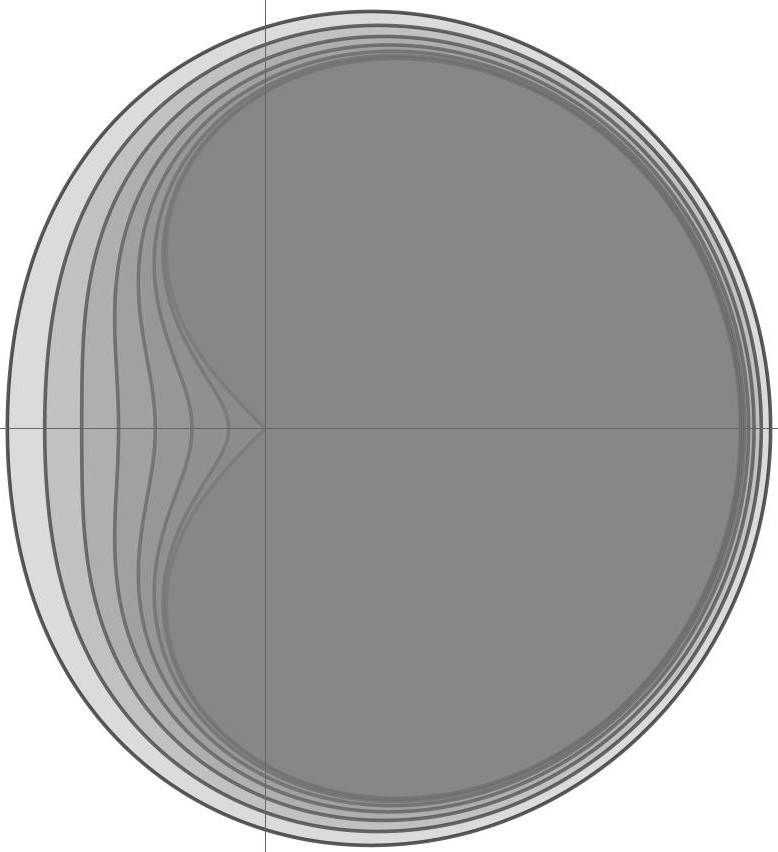}\includegraphics[height=0.4\linewidth,width=.4\linewidth]{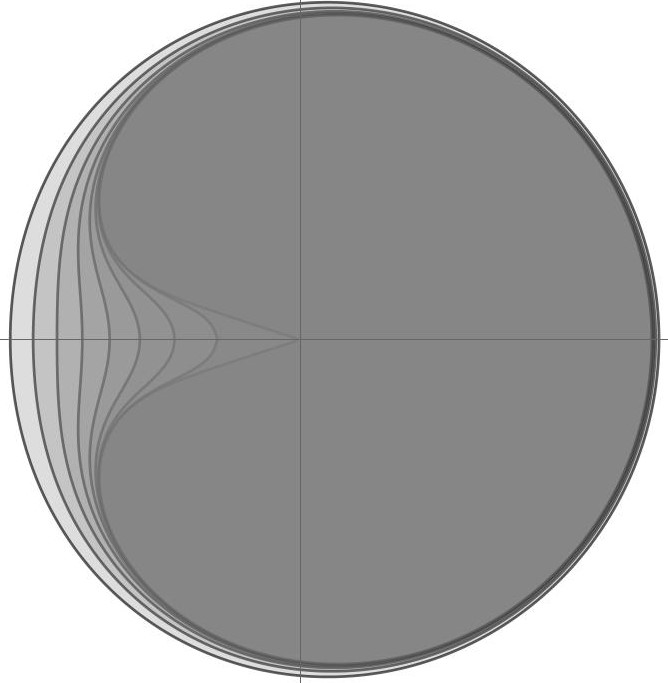}
    \caption{Two families of simply connected unbounded $\Omega\in\QD_a\left(\frac{1}{2}\right)$ (the complements of the shaded regions) with $a=2,5$ (Example \ref{ex:15_PQD_Nested_Family_PostCrit})}\label{fig:15_PQD_Nested_Family_PostCrit}\vspace{1.5em}
\end{figure}
\FloatBarrier

We furthermore obtain a detailed characterization of the pole structure of $r$ in terms of $h$.
\begin{corollary}\label{corr:GenUPQDFTFormula}
Let $h$ be a rational function with poles $\{p_k\}$ of respective orders $\{n_k\}$, and a pole at $\infty$ of order $n$. Then the $r$ in Theorem \ref{thm:GenUPQDInvProbFormula} is given by
\begin{equation}\label{eqn:GenUPQDFTFormula}
r(z)=C+zp(z)+z\sum_{k}\sum_{j=1}^{n_k}\dfrac{\beta_{k,j}}{\left(z-z_k\right)^{j}},
\end{equation}
for some constants $\{\beta_{k,j}\}$, $\varphi(z_k)=p_k$, a polynomial $p$ of degree $n$ ($p=0$ if $h(\infty)=0$); and $C$ as in the theorem. Moreover, if $h$ has a pole at zero then the corresponding term in the above sum is of one degree lower. In particular, {\it the poles of $h$ are the images of those of $r$ under $\varphi$}.
\end{corollary}

\subsection{Characterization of Simply Connected One Point PQDs}\label{subsec:OnePtPQDs}
In the prior discussion, we showed that simply connected PQDs admit representations in terms of rational functions. In the following analysis, we characterize this representation for one point PQDs of arbitrary order. Lemmas \ref{lemma:PQDOnePtClassHigherOrder} and \ref{lemma:PQDOnePtInftyClassHigherOrder} cover the case in which the quadrature node is finite (i.e. the quadrature function is a rational function with a unique pole). Lemma \ref{lemma:PolynomialUPQDClass} covers the case in which the quadrature node is at $\infty$ (i.e. the quadrature function is a polynomial).

We begin with a classification of simply connected one point PQDs of order one with coefficient $\alpha>0$, $\QD_a\left(\frac{\alpha}{w-w_0}\right)$. There are several cases to consider:
\begin{enumerate}
    \item $\Omega$ is bounded:
    \begin{enumerate}
        \item $0\notin\Omega$,
        \item $0\in\Omega$ and $w_0\neq0$,
        \item $0\in\Omega$ and $w_0=0$,
    \end{enumerate}
    \item $\Omega$ is unbounded:
    \begin{enumerate}
        \item $0\notin\Omega$,
        \item $0\in\Omega$, but $w_0\neq0$.\footnote{By symmetry, $\Omega$ cannot be unbounded if $w_0=0$.}
    \end{enumerate}
\end{enumerate}
Case 1 is covered by \cite{DragnevLeggSaff} for integer $a$ and summarized by the following theorem
\begin{theorem}
For each $n\in\Z_{+}$, $w_0\in\C$, and $\alpha>0$, there exists a bounded simply connected $\Omega\in\QD_n\left(\frac{\alpha}{w-w_0}\right)$. In addition,
\begin{enumerate}[a.]
\item If $|w_0|^{2n}\geq n^2\alpha$, then $\Omega$ is an nth root of $\D_{n\sqrt{\alpha}}(w_0^n)$, specifically the connected component containing $w_0$.
\item If $w_0=0$, then $\Omega=\D_{\sqrt[2n]{\alpha n}}(0)$,
\item If $0<|w_0|^{2n}<n^2\alpha$, then $\Omega$ is the image of the unit disk under a map of the form $\varphi(z)=\dfrac{\delta z}{(1-z\overline{z_0})^{\frac{1}{n}}}$, where $\delta$ is a constant and $\varphi(z_0)=w_0$.
\end{enumerate}
Cases $(a)$ and $(c)$ correspond to $0\notin\Omega$ and $0\in\Omega$ respectively.
\end{theorem}
By applying the following univalence criterion (Lemma \ref{lemma:starlikeUnivalence}), we show furthermore that in case (c), $\varphi$ is univalent in $\D$ iff $|z_0|<1$. An easy corollary of Theorem 6.6 in \cite{Pommerenke1975} is that
\begin{lemma}\label{lemma:starlikeUnivalence}
An analytic function $f:\D\rightarrow\Ch$ for which $\varphi(0)=0$ is univalent and starlike with respect to the origin iff
$$\Re\left(z\dfrac{\varphi'(z)}{\varphi(z)}\right)>0$$
for all $z\in\D$.
\end{lemma}
In our case, $\varphi(z)=\frac{\delta z}{(1-z\overline{z_0})^{\frac{1}{n}}}$ is analytic in $\D$ iff $|z_0|<1$ so, wlog replacing $\overline{z_0}$ with $|z_0|$ and applying the lemma and partial fractions we obtain
\begin{align*}
n\Re\left(z\dfrac{\varphi'(z)}{\varphi(z)}\right)&=\Re\left(z\dfrac{(\varphi^n)'(z)}{\varphi^n(z)}\right)=n-1+\Re\left(\dfrac{1}{1-z|z_0|}\right).
\end{align*}
Taking $z|z_0|=re^{i\theta}$ and noting that $n\geq1$, we find
\begin{align*}
    n\Re\left(z\dfrac{\varphi'(z)}{\varphi(z)}\right)&\geq\dfrac{1}{2}\left(1+\dfrac{1-r^2}{r^2-2r\cos(\theta)+1}\right)\geq\dfrac{1}{2}\left(1+\dfrac{1-r^2}{r^2+2r+1}\right)>0
\end{align*}
for all $0\leq r<1$. The desired conclusion follows.\\

We now consider the unbounded case. As in the discussion of classical one point QDs (\S\ref{sec:OnePtQDClass}), it is instructive to consider the situation from the perspective of potential theory. In particular, Lemmas \ref{lemma:GenComplementDropletQuad} and \ref{lemma:GenQuadComplementDroplet} to follow characterize the relationship between these two perspectives.

\begin{lemma}\label{lemma:GenComplementDropletQuad}
If $K$ is a local droplet of a potential $Q(w)=\frac{|w|^{2a}}{a^2}-2\Re(H(w))$ for some $a>0$ and $H$ with rational derivative, then $K^c$ is a disjoint union of PQDs for which $h:=H'$ is the sum of their quadrature functions. 
\end{lemma}
This is an easy consequence of the existence of the coincidence equation for local droplets (Equation \ref{eqn:ceq}) and Theorem \ref{theorem:EquivPQDChars}. 

\begin{lemma}\label{lemma:GenQuadComplementDroplet}
If $K\subset\C$ is a compact subset such that $K^c$ is a finite disjoint union of PQDs, $\Omega_j\in\QD_a(h_j)$, $a>0$, then $K$ is a local droplet. Furthermore, the external potential associated to the droplet is given by
$$Q(w)=\dfrac{|w|^{2a}}{a^2}-2\Re(H(w)),\;H(w)=2\Re\left(\int_{w_l}^{w}h(\xi)d\xi\right)+\text{const}(w_l)\text{ in some open nbhd of }K_l$$
(where $h=\sum h_j$ and $w_l$ is any point in the $l$th connected component of $K$, $K_l$). In particular, $Q$ is unique up to several additive constants.
\end{lemma}
\begin{proof}[Proof of Lemma \ref{lemma:GenQuadComplementDroplet}]\label{proof:GenQuadComplementDroplet}
Let $K^c=\Omega_\infty\cup\bigcup_{j=1}^{N}\Omega_j$ be a disjoint union of PQDs with respective quadrature functions $h_j$. Also let $h=\sum_{j}h_j$. Applying Frostman's theorem (Lemma \ref{lemma:FrostmanPt2}), it is sufficient to construct an open nbhd $\mathcal{O}\supset K$ and harmonic $H:\mathcal{O}\rightarrow\R$ such that $\frac{\partial H}{\partial w}=h$ and
\begin{equation}\label{eqn:GenQuadCompDropIdPf}
    \dfrac{|w|^{2a}}{a^2}-H(w)+U^\mu(w)=0,\;\forall w\in K,
\end{equation}
where $d\mu=\1_K\frac{\Delta Q}{2}dA$. We shall choose $\mathcal{O}$ as an $\epsilon-$nbhd of $K$ for a to be determined value of $\epsilon$. We require that $\mathcal{O}$ satisfies each of the following
\begin{enumerate}
    \item $h$ is holomorphic in $\mathcal{O}$,
    \item each connected component of $\mathcal{O}$ contains precisely one connected component of $K$, 
    \item every loop in $\mathcal{O}$ is homotopic to a loop in $K$.
\end{enumerate}
By assumption, $\partial K$ has finitely many singular points. Thus there exists $\epsilon>0$ s.t. (2) and (3) hold.
Note that, for each $\Omega_j$ and each $w\in (\Omega_j)\IntComp$, the quadrature identity tells us that
$$C_Q^{\Omega_j}(w)=\int_{\Omega_j}\dfrac{|\xi|^{2(a-1)}}{w-\xi}dA(\xi)=\oint_{\partial\Omega_j}\dfrac{h_j(\xi)}{w-\xi}d\xi=\oint_{\partial(\Omega_j)\IntComp}\dfrac{h_j(\xi)}{\xi-w}d\xi=h_j(w),$$
(see Equation \ref{sec:PotentialTheoryHeleShaw} for the initial discussion of $C_Q$.)
and this extends to $\partial\Omega_j$. Thus, for each $w\in\bigcap_{j}\Omega_j^{c}=K$, 
$$C_Q^\Omega(w)=\int_{\Omega}\dfrac{|\xi|^{2(a-1)}}{w-\xi}dA(\xi)=\sum_{j}\int_{\Omega_j}\dfrac{|\xi|^{2(a-1)}}{w-\xi}dA(\xi)=\sum_{j}C_Q^{\Omega_j}(w)=\sum_jh_j(w)=h(w).$$
But then note that
\begin{align*}
C_Q^\C(w)=\int_{\C}\dfrac{|\xi|^{2(a-1)}}{w-\xi}dA(\xi)&=\lim_{R\to\infty}\int_{\D_R}\dfrac{|\xi|^{2(a-1)}}{w-\xi}dA(\xi)=\dfrac{1}{a}\overline{w}|w|^{2(a-1)}-\dfrac{1}{a}\lim_{R\to\infty}\oint_{\partial\D_R}\dfrac{\overline{\xi}|\xi|^{2(a-1)}}{\xi-w}d\xi\\
&=\dfrac{1}{a}\overline{w}|w|^{2(a-1)}-\dfrac{1}{a}\lim_{R\to\infty}R^{2a}\oint_{\partial\D_R}\dfrac{\xi^{-1}}{\xi-w}d\xi=\dfrac{1}{a}\overline{w}|w|^{2(a-1)}.
\end{align*}
So $C_Q^K(w)+C_Q^\Omega(w)=C_Q^\C(w)=\frac{1}{a}\overline{w}|w|^{2(a-1)}$. Thus $h(w)=\frac{1}{a}\overline{w}|w|^{2(a-1)}-C_{Q}^K(w)$ on $K$. Then note that if $\gamma$ is a loop in $K$,
\begin{align*}
2\Re\left(\oint_{\gamma}h(w)dw\right)&=\oint_{\gamma}\left(\dfrac{1}{a}\overline{w}|w|^{2(a-1)}-C_{Q}^{K}(w)\right)dw+\oint_{\gamma}\left(\overline{\dfrac{1}{a}\overline{w}|w|^{2(a-1)}}-\overline{C_{Q}^{K}(w)}\right)d\overline{w}\\
&=\oint_{\gamma}\partial\left(\dfrac{|w|^{2a}}{a^2}+U^{\mu}(w)\right)+\oint_{\gamma}\overline{\partial}\left(\dfrac{|w|^{2a}}{a^2}+U^{\mu}(w)\right)\\
&=\oint_{\gamma}d\left(\dfrac{|w|^{2a}}{a^2}+U^{\mu}(w)\right)=0.
\end{align*}
So, by Morera's theorem, $h$ is holomorphic in $K$ and thus in $\mathcal{O}$ because $h$ is a sum of quadrature functions on the components of $K^c$, so, by (3), this result extends to $\mathcal{O}$.

As a consequence, if we pick a point $w_l$ in each connected component $\mathcal{O}_l$ of $\mathcal{O}$ and set
$$H(w):=2\Re\left(\int_{w_l}^{w}h(\xi)d\xi\right)+\dfrac{|w_l|^{2a}}{a^2}+U^\mu(w_l),\;\forall w\in\mathcal{O}_l,$$
is a real, well-defined harmonic function in $\mathcal{O}$ and $\frac{\partial H}{\partial w}=h$ on $K$. Finally, we need to verify Equation \ref{eqn:GenQuadCompDropIdPf}. By construction, the identity holds at each $w_l$. Furthermore, we have
$$\dfrac{\partial}{\partial w}\left(\dfrac{|w|^{2a}}{a^2}-H(w)+U^\mu(w)\right)=\dfrac{1}{a}\overline{w}|w|^{2(a-1)}-h(w)-C^K(w)=0\;\text{ on }K,$$
and
$$\dfrac{\partial}{\partial \overline{w}}\left(\dfrac{|w|^{2a}}{a^2}-H(w)+U^\mu(w)\right)=0\;\text{ on }K.$$
Equation \ref{eqn:GenQuadCompDropIdPf} follows.

An additional consequence of this (from Frostman's theorem \ref{lemma:FrostmanPt2} and Equation \ref{eqn:GenQuadCompDropIdPf}) is that $K$ is a local droplet of the potential $Q(w)=\frac{|w|^{2a}}{a^2}-H(w)$.
\end{proof}

\subsubsection{Simple One Point PQDs}

By lemma \ref{lemma:GenQuadComplementDroplet}, if $\Omega\in\QD_a\left(\frac{\alpha}{w-w_0}\right)$ is unbounded and simply connected, then $\Omega^{c}$ is a local droplet of the potential
$$Q(w)=\dfrac{|w|^{2a}}{a^2}-2\Re(\alpha\ln(w-w_0)).$$
Thus one might expect a $1-$parameter family of PQDs $\{\Omega_t\}_{t}$, parameterized by the weighted area of the droplets in the corresponding weighted Hele-Shaw chain. This turns out to be the case. However, the situation is more complicated than the classical case discussed in \S\ref{sec:OnePtQDClass} on account of the singularity of the weight at $0$ - this phenomenon also manifests itself in Lemmas \ref{lemma:PolynomialUPQDClass}, \ref{lemma:PQDOnePtClassHigherOrder}, and \ref{lemma:PQDOnePtInftyClassHigherOrder}. In particular, there are apparently four distinct families to consider, depicted in Figure \ref{fig:PQDOnePtUnboundedTable}.\\

\begin{figure}[ht]
  \centering
\begin{table}[H]
\centering
\begin{tblr}{
  cells={valign=m,halign=c},
  row{1}={bg=lightgray,font=\bfseries,rowsep=8pt},
  column{1}={bg=lightgray,font=\bfseries},
  colspec={QQQ},
  hlines,
  vlines
}
 & $\alpha>0$ & $\alpha<0$\\ 
$0\notin\Omega_{t_\ast}$ & \includegraphics[scale=.66,valign=c
]{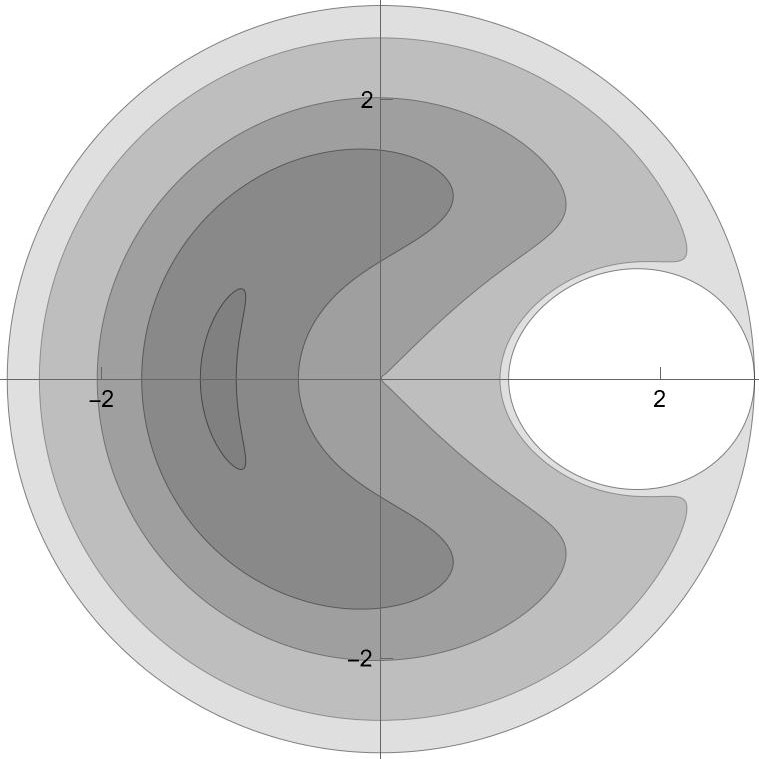} & \includegraphics[scale=.66,valign=c
]{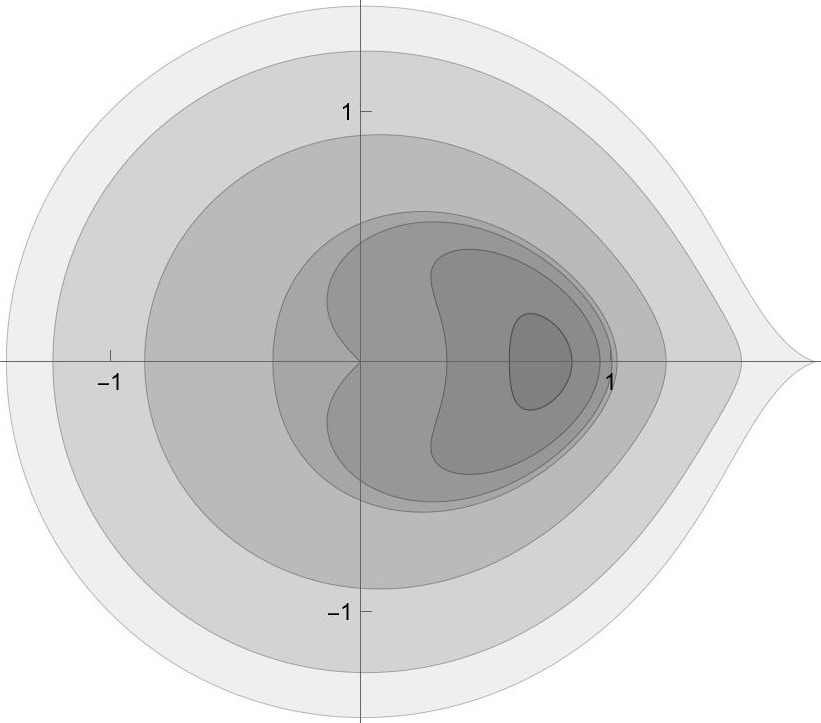}\\ 
$0\in\Omega_{t_\ast}$ & \includegraphics[scale=.66,valign=c
]{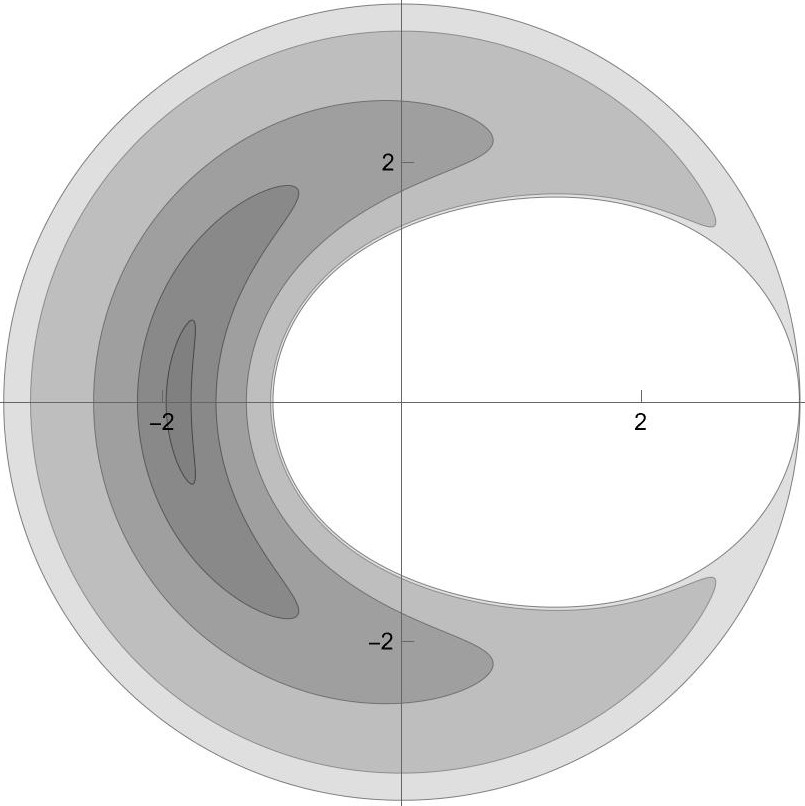} & \includegraphics[scale=.66,valign=c
]{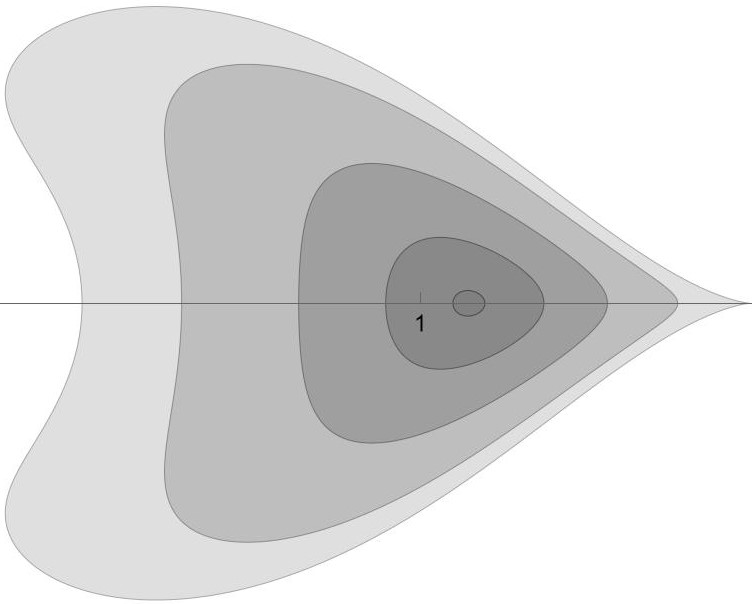}\\
\end{tblr}
\end{table}\vspace{.5em}
\caption{The families of one point unbounded PQDs (complements of the shaded regions) in $\QD_2\left(\frac{\alpha}{w-2}\right)$ with $\alpha=\frac{5}{2}$, $\alpha=\frac{25}{2}$ in column $1$ and $\alpha=-\frac{1}{4}$, $\alpha=-\frac{3}{5}$ in column $2$. Rows 1 and 2 were obtained using Theorems \ref{thm:unbd1ptPQDclassificationNoZero} and \ref{thm:unbd1ptPQDclassificationZero} respectively.}\label{fig:PQDOnePtUnboundedTable}\vspace{1.5em}
\end{figure}
If $\alpha>0$ then there seems to be a unique simply connected $1-$parameter family
$$\{\Omega_t\}_{0<t\leq t_{\ast}}\overset{?}{=}\QD_a\left(\frac{\alpha}{w-w_0}\right)\cap\left\{\Omega\subset\Ch:\Omega\text{ is unbounded}\right\},$$
where $t_{\ast}=\frac{w_0^a}{a}(w_0^a+2a\sqrt{\alpha})+(a-1)\alpha$.

The family and the two phases can be characterized in terms of their Riemann maps, given by Equations \ref{eqn:OnePtPUQDNoZero} and \ref{eqn:OnePtPUQDZero} respectively.

\begin{theorem}\label{thm:unbd1ptPQDclassificationNoZero}
Fix $a>0$ and let $0\notin\Omega$ be an unbounded simply connected domain with conformal radius $c>0$. Then there exist $\alpha,w_0\in\C\setminus\{0\}$ for which $\Omega\in\QD_a\left(\frac{\alpha}{w-w_0}\right)$ if and only if there exists $z_0\in\D\IntComp$ such that $\Omega=\varphi(\D\IntComp)$, where $\varphi:\D\IntComp\rightarrow\Ch$ is a univalent conformal map of the form
\begin{equation}\label{eqn:OnePtPUQDNoZero}
\varphi(z)=cz\left(1+\dfrac{|z_0|^2-1}{z\overline{z_0}-1}\left(\beta^{a}-1\right)\right)^{\frac{1}{a}},
\end{equation}
for which $\alpha=\frac{c^{2a}}{a^2}(1-\overline{\beta}^a)\left(|z_0|^2(1+\beta^a(a-1))-a\beta^a\right)$ and $\left|1-(|z_0|^2-1)(\beta^{a}-1)\right|>|z_0|$, where $\beta=\frac{w_0}{c z_0}$. Also note that $\varphi(z_0)=w_0$.
\end{theorem}

\begin{proof}[Proof of Theorem \ref{thm:unbd1ptPQDclassificationNoZero}]\label{proof:unbd1ptPQDclassificationNoZero}\;\\
The equivalence follows directly from Lemma \ref{lemma:PQDOnePtInftyClassHigherOrder}. To obtain Equation \ref{eqn:OnePtPUQDNoZero}, apply Lemma \ref{lemma:PQDOnePtInftyClassHigherOrder} so that $\varphi(z)=cz\left(\frac{z-z_1}{z-\overline{z_0}^{-1}}\right)^{\frac{1}{a}}$, with $\varphi(z_0)=w_0$, which implies $z_1=z_0-\frac{|z_0|^2-1}{\overline{z_0}}\left(\frac{w_0}{c z_0}\right)^a$. Substitution and simplification yields Equation \ref{eqn:OnePtPUQDNoZero}. The final relation follows straightforwardly from Equation \ref{eqn:UPQDDirectProblemSol}.
\end{proof}

\begin{theorem}\label{thm:unbd1ptPQDclassificationZero}
Fix $1\neq a>0$, and let $0\in\Omega$ be an unbounded simply connected domain with conformal radius $c>0$. Then there exist $\alpha,w_0\in\C\setminus\{0\}$ for which $\Omega\in\QD_a\left(\frac{\alpha}{w-w_0}\right)$ if and only if there are $z_0,z_1\in\D\IntComp$ such that $\Omega=\varphi(\D\IntComp)$, where $\varphi:\D\IntComp\rightarrow\Ch$ is a conformal map of the form
\begin{equation}\label{eqn:OnePtPUQDZero}
\varphi(z)=cz|z_1|b_{z_1}(z)\left(1+\dfrac{|z_0|^2-1}{z\overline{z_0}-1}\left(\beta^a-1\right)\right)^{\frac{1}{a}},
\end{equation}
for which $\beta=\frac{w_0}{cz_0}\frac{z_0-\overline{z_1}^{-1}}{z_0-z_1}$, and satisfies the relation
$$\alpha=\dfrac{c^{2a}}{a^2}\dfrac{|z_1|^{2(a-1)}}{|z_0|^2-1}\left(a\overline{z_0}(z_0(\overline{z_1}z_0-2)+z_1)-\dfrac{|z_0-z_1|^2}{1-|z_0|^{-2}}\right).$$
Also note that $\varphi(z_0)=w_0$, and $\varphi(z_1)=0$.
\end{theorem}

\begin{proof}[Proof of Theorem \ref{thm:unbd1ptPQDclassificationZero}]\label{proof:unbd1ptPQDclassificationZero}\;\\
The equivalence follows directly from lemma \ref{lemma:PQDOnePtInftyClassHigherOrder}. To obtain Equation \ref{eqn:OnePtPUQDZero} note that, by lemma \ref{lemma:PQDOnePtInftyClassHigherOrder},
$$\varphi(z)=cz|z_1|b_{z_1}(z)\left(\frac{z-\overline{z_1}^{-1}}{z-\overline{z_0}^{-1}}\right)^{\frac{1}{a}},$$
with $\varphi(z_0)=w_0$, which implies $\overline{z_0}^{-1}=z_0-\frac{z_0\overline{z_1}-1}{\overline{z_1}}\left(\frac{cz_0\overline{z_1}}{w_0}\frac{z_0-z_1}{z_0\overline{z_1}-1}\right)^a=z_0-\frac{z_0\overline{z_1}-1}{\overline{z_1}}\beta^a$. Substitution and simplification yields Equation \ref{eqn:OnePtPUQDZero}. The final relation follows from Equation \ref{eqn:UPQDDirectProblemSol}.
\end{proof}

Before continuing on to our discussion of Monomial PQDs, we state the following univalence criterion particularly applicable to PQDs.

\begin{lemma}\label{lemma:PolynomialPowerUnivalence}
Take $a>0$, $k\in\Z_{+}$, $\{z_j\}_{j=1}^{k}$ nonzero constants, and assume $\varphi:\D\IntComp\rightarrow\Ch$ is given by $\varphi(z)=z\left(\left(1-\frac{z_1}{z}\right)\hdots\left(1-\frac{z_k}{z}\right)\right)^{\frac{1}{a}}$. Then,
\begin{enumerate}
\item If $k\leq 2a$ then $\varphi$ is univalent if and only if each $|z_j|\leq1$. In this case, $\varphi(\D\IntComp)$ is starlike with respect to infinity.
\item If $k>2a$, then $\varphi$ is univalent and $\varphi(\D\IntComp)$ is starlike with respect to infinity if and only if $|z_j|\leq1$ and
$$a+\Re\left(\sum_{j=1}^{k}\dfrac{z_j}{z-z_j}\right)>0,\;\;\;\forall z\in\D\IntComp.$$
\end{enumerate}
\end{lemma}

\begin{proof}[Proof of lemma \ref{lemma:PolynomialPowerUnivalence}]
This is a straightforward consequence of the following result:\\
{\it
Let $a>0$, $k,l\in\Z_+$ and, $\{z_j\}_{j=1}^{k}$, $\{w_j\}_{j=1}^{l}$ nonzero constants. If $\varphi:\D\IntComp\rightarrow\Ch$ is given by
$$\varphi(z)=z\left(\dfrac{\left(1-\frac{z_1}{z}\right)\hdots\left(1-\frac{z_k}{z}\right)}{\left(1-\frac{w_1}{z}\right)\hdots\left(1-\frac{w_l}{z}\right)}\right)^{\frac{1}{a}}$$
then $\varphi$ is univalent and $\varphi(\D\IntComp)$ is starlike with respect to infinity if and only if each $|z_j|,|w_j|\leq1$ and
$$a+\Re\left(\sum_{j=1}^{k}\dfrac{z_j}{z-z_j}-\sum_{j=1}^{l}\dfrac{w_j}{z-w_j}\right)>0,\;\;\;\forall z\in\D\IntComp.$$
}
It follows from Lemma \ref{lemma:starlikeUnivalence} that an analytic function $f:\D\IntComp\rightarrow\Ch$ is univalent and $\varphi(\D\IntComp)$ is starlike with respect to infinity if and only if $\Re\left(z\frac{f'(z)}{f(z)}\right)>0$ for all $z\in\D\IntComp$. Moreover, by the chain rule, $\Re\left(z\frac{f'(z)}{f(z)}\right)>0$ if and only if $\Re\left(z\frac{(f^a)'(z)}{f^a(z)}\right)>0$. Thus, its sufficient to show that each $|z_j|,|w_j|\leq1$ and $$a+\Re\left(\sum_{j=1}^{k}\dfrac{z_j}{z-z_j}-\sum_{j=1}^{l}\dfrac{w_j}{z-w_j}\right)>0,\;\;\;\forall z\in\D\IntComp$$
if and only if $\varphi$ is analytic in $\D\IntComp$ and $\Re\left(z\frac{(\varphi^a)'(z)}{\varphi^a(z)}\right)>0$ in $\D\IntComp$.\\
First, note that each $|z_j|,|w_j|\leq1$ if and only if $\varphi$ is analytic in $\D\IntComp$. Second, we have that
$$\varphi^a(z)=z^{a-k+l}\frac{\left(z-z_1\right)\hdots\left(z-z_k\right)}{\left(z-w_1\right)\hdots\left(z-w_l\right)},$$
so
\begin{align*}
\Re\left(z\frac{(\varphi^a)'(z)}{\varphi^a(z)}\right)&=\Re\left(z\dfrac{a-k+l}{z}+\sum_{j=1}^{k}\dfrac{z}{z-z_j}-\sum_{j=1}^{l}\dfrac{z}{z-w_j}\right)=a+\Re\left(\sum_{j=1}^{k}\dfrac{z_j}{z-z_j}-\sum_{j=1}^{l}\dfrac{w_j}{z-w_j}\right).
\end{align*}
The desired result follows.
\end{proof}

\subsection{Classification of Basic Monomial PQDs}\label{subsec:BasicMonomialPQDClass}
In Example \ref{ex:15_PQD_Nested_Family_PostCrit}, we showed that the image $\Omega$ of $\D\IntComp$ under the map
\begin{equation}\label{eqn:BasicPQDNoZeroRiemannMap}
\varphi(z)=cz\left(1-\dfrac{\gamma}{z}\right)^{\frac{1}{a}}
\end{equation}
is a PQD with quadrature function $\alpha=-\frac{c^{2a-1}}{a}\overline{\gamma}$ whenever $|\gamma|<1$, $a,c>0$, and $\varphi$ is univalent in $\D\IntComp$.

Using this result as a starting point, our aim in this section is to fully classify the simply connected domains $\Omega\in\QD_a(\alpha)$. We refer to such domains as {\it basic monomial PQDs}. It turns out that all such domains not containing zero admit a Riemann map in the form of Equation \ref{eqn:BasicPQDNoZeroRiemannMap}. The situation is slightly more complicated when $0\in\Omega$, so we deal with this case last. 

\subsubsection{Basic Monomial PQDs Not Containing Zero}

Theorem \ref{thm:ConstMonomialPQDNoZeroClass} below provides a characterization of basic monomial PQDs not containing zero.

\begin{theorem}\label{thm:ConstMonomialPQDNoZeroClass}
Take $a>0$ and $\alpha\in\C\setminus\{0\}$. There exists a simply connected domain $\Omega$, not containing zero, of conformal radius $c$ for which $\Omega\in\QD_a(\alpha)$ if and only if either
\begin{enumerate}
    \item $0<a<\frac{1}{2}$ and $|\gamma|\leq\frac{a}{1-a}$,
    \item $a>\frac{1}{2}$ and $|\gamma|\leq1$,
    \item $a=\frac{1}{2}$ and $|\alpha|\leq2$.
\end{enumerate}
where $\gamma=-\frac{a\overline{\alpha}}{c^{2a-1}}$. In this case, $\Omega$ is unique modulo conformal radius, and $\Omega=\varphi(\D\IntComp)$ with $\varphi$ univalent in $\D\IntComp$ and given by Equation \ref{eqn:BasicPQDNoZeroRiemannMap}.
\end{theorem}
\begin{proof}[Proof of Theorem \ref{thm:ConstMonomialPQDNoZeroClass}]
We proved in Example \ref{ex:15_PQD_Nested_Family_PostCrit} that if $\varphi$ (Equation \ref{eqn:BasicPQDNoZeroRiemannMap}) is a Riemann map for $\Omega$, then $\Omega\in\QD_a(\alpha)$, where $\alpha=-\frac{c^{2a-1}}{a}\overline{\gamma}$. 

Conversely, suppose $\Omega\not\ni0$ is a simply connected domain in $\QD_a(\alpha)$ for some $a>0$ and $\alpha\in\C\setminus\{0\}$ (necessarily unbounded because $\infty$ is a quadrature node). Then by Theorem \ref{thm:GenUPQDInvProbFormula} there exists a rational function $r$ for which $\varphi(z)=zr^{\#}(z)^{\frac{1}{a}}$ is a Riemann map for $\Omega$ and
$$r=a\Phi_{\varphi}^{-1}\left(\AnalyticIn{\alpha w\left(\dfrac{\psi(w)}{w}\right)^{a}}{\Omega\IntComp}\right)+\dfrac{a}{c^{a}}t.$$
Also note that the argument of the Faber transform grows linearly about $\infty$, so there exists $c_0\in\C$ s.t. $w\left(\frac{\psi(w)}{w}\right)^{a}-\frac{w+c_0}{c^a}\in\A_0(\Omega)$, which implies
$$\AnalyticIn{\alpha w\left(\dfrac{\psi(w)}{w}\right)^{a}}{\Omega\IntComp}=\dfrac{\alpha}{c^a}\AnalyticIn{w+c_0}{\Omega\IntComp}+\alpha\AnalyticIn{w\left(\dfrac{\psi(w)}{c^{-1}w}\right)^{a}-\dfrac{w+c_0}{c^a}}{\Omega\IntComp}=\dfrac{\alpha}{c^a}(w+c_0).$$
Substituting this into the above expression and computing the Faber transform, we obtain
\begin{align*}
r^{\#}(z)&=\dfrac{a\overline{\alpha}}{c^{a}}\Phi_{\varphi}^{-1}\left(w+c_0\right)^{\#}(z)+\dfrac{a}{c^{a}}t=a\overline{\alpha}\dfrac{cz^{-1}+\overline{f_0}}{c^{a}}+\dfrac{a}{c^{a}}t=\dfrac{a\overline{\alpha}}{c^{a-1}}z^{-1}+c_1.
\end{align*}
Considering the asymptotics, $\varphi(z)=cz+O(1)$, we find that $c_1=c^a$, so
$$\varphi(z)=cz\left(1-\dfrac{\gamma}{z}\right)^{\frac{1}{a}}.$$

Now all there is to do is demonstrate that $\varphi$ is univalent if and only if one of conditions 1-3 in the theorem is satisfied. By Lemma \ref{lemma:PolynomialPowerUnivalence}, when $a\geq\frac{1}{2}$, $\varphi$ is univalent in $\D\IntComp$ if and only if $|\gamma|\leq1$. Furthermore, when $a=\frac{1}{2}$, $\gamma$ is independent of $c$ with $-2\gamma=\overline{\alpha}$, so $\varphi$ is univalent iff $|\alpha|\leq2$.

On the other hand, if $a<\frac{1}{2}$ then Lemma \ref{lemma:PolynomialPowerUnivalence} tells us that $\varphi$ is univalent in $\D\IntComp$ if and only if $|\gamma|\leq1$ and $a+\Re\left(\frac{\gamma}{z-\gamma}\right)>0$ for all $z\in\D\IntComp$. Rotating $\D\IntComp$ by taking $z\mapsto-\frac{\gamma}{|\gamma|}z$ yields $a-|\gamma|\Re\left(\frac{1}{z+|\gamma|}\right)$ which is minimized as $z\to1$. Thus we obtain the condition $|\gamma|\leq\frac{a}{1-a}$. Or, equivalently $0<c\leq c_{\ast}$, where $c_\ast=\sqrt[2a-1]{|\alpha|(1-a)}$. Thus when $0<a<\frac{1}{2}$, $\varphi$ is univalent iff $|\gamma|\leq\frac{a}{1-a}$.
\end{proof}

\begin{figure}[ht]
  \centering    \vcenteredhbox{\includegraphics[height=.15\textwidth,width=.32\textwidth]{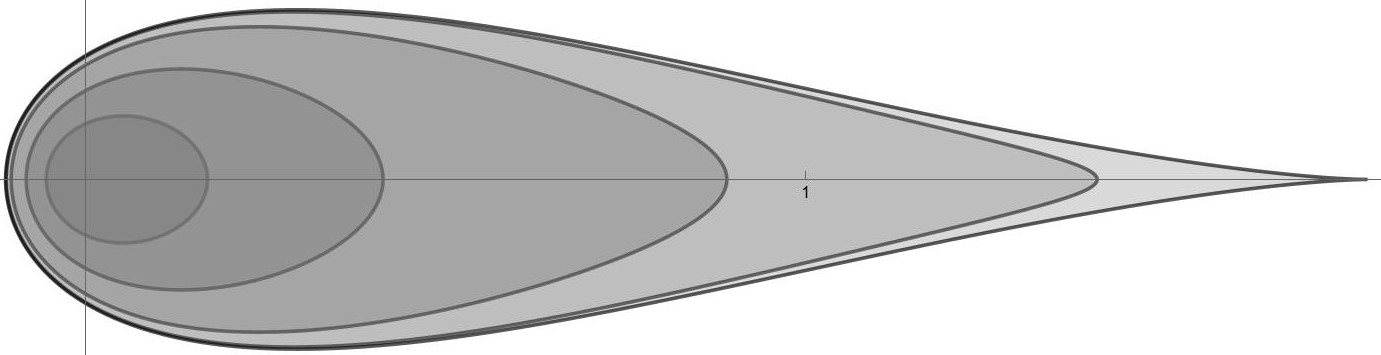}}\;\;\;\;\vcenteredhbox{\includegraphics[height=.15\textwidth,width=.32\textwidth]{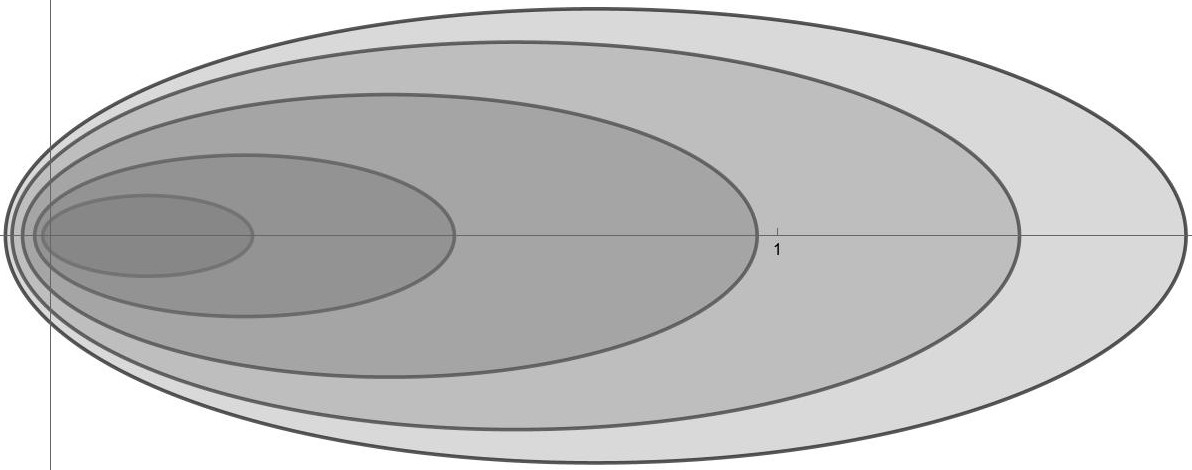}}\;\;\;\;\vcenteredhbox{\includegraphics[height=.3\textwidth]{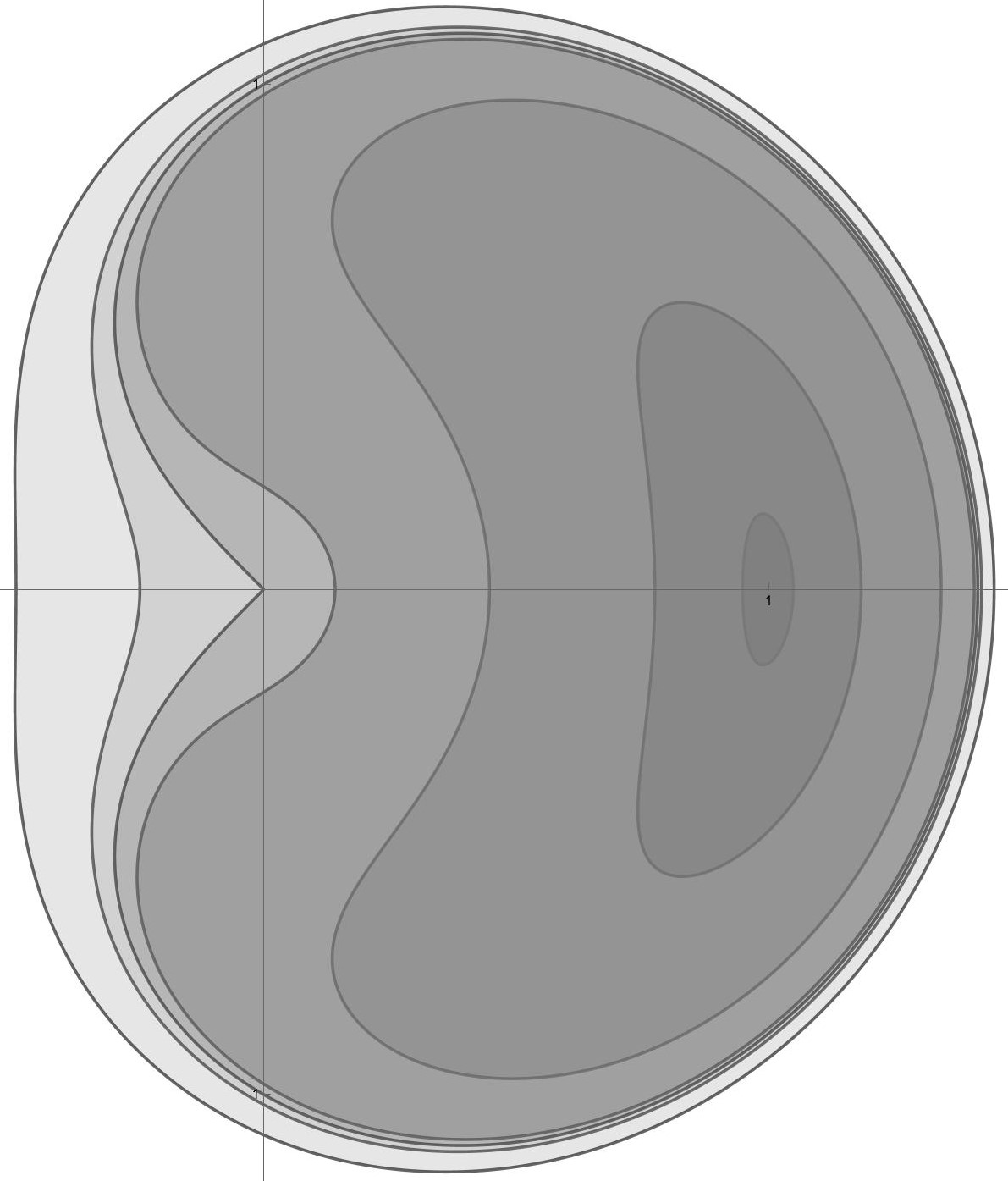}}
    \caption{Example families of $\Omega\in\QD_{a}(\alpha)$ (complements of the shaded regions). Finite family: $\alpha=\frac{25}{12}$ and $a=.4<\frac{1}{2}$ with $c\in\{.03,.1,.2,.25,.3,.32768\}$ (left); traveling wave: $\alpha=\frac{4}{3}$ and $a=\frac{1}{2}$, with $c\in\{.05,.1,.2,.3,.4,.5\}$ (center); two phase: $\alpha=\frac{1}{2}$ and $a=2>\frac{1}{2}$, with $c\in\{.1, .4, .8, .98, 1, 1.02, 1.08\}$ (right).}\label{fig:.4-1-1o3c_Droplets_Finite_All}\vspace{1.5em}
\end{figure}

\subsubsection{Basic Monomial PQDs Containing Zero}
The additional complexity in this case is owed to the fact that when $0\in\Omega$, the formula for the associated Riemann map picks up a Blaschke factor with parameter $z_0$, the unique root of $\varphi$ in $\D\IntComp$. $\varphi$ nevertheless admits a relatively simple representation.
\begin{theorem}\label{thm:ConstMonomialPQDZeroClass}
Take $a>0$ and $\alpha\in\C\setminus\{0\}$. There exists a simply connected domain $\Omega$ of conformal radius $c>0$ containing $0$ for which $\Omega\in\QD_a(\alpha)$ if and only if $|\gamma|\geq1$, with $\gamma$ as in Theorem \ref{thm:ConstMonomialPQDNoZeroClass}. In this case, $\Omega$ is unique modulo conformal radius, and $\Omega=\varphi(\D\IntComp)$ with $\varphi$ given by
\begin{equation}\label{eqn:PQDk1MonomialRiemannMapZero}
\varphi(z)=cz\dfrac{z-\overline{\gamma}^{\frac{1}{2a-1}}}{z-\gamma^{-\frac{1}{2a-1}}}\left(1-\dfrac{\gamma^{\frac{1}{2a-1}}}{z}\right)^{\frac{1}{a}},\;\;\;\;\;\gamma^{\frac{1}{2a-1}}=-\left(\frac{a\overline{\alpha}}{c^{2a-1}}\right)^{\frac{1}{2a-1}}.
\end{equation}
\end{theorem}

\begin{proof}[Proof of Theorem \ref{thm:ConstMonomialPQDZeroClass}]
Let $0\in\Omega\in\QD_a(h)$ be a simply connected PQD of conformal radius $c>0$ with quadrature function $\alpha\in\C\setminus\{0\}$ for some $a>0$. Then by Lemma \ref{lemma:PolynomialUPQDClass}, 
$$\varphi(z)=zb_{z_0}(z)\left(|cz_0|^a\left(1-\dfrac{\overline{z_0}^{-1}}{z}\right)\right)^{\frac{1}{a}}$$
is a Riemann map for $\Omega$ for some $z_0\in\D\IntComp$. Supposing that $\varphi$ is univalent in $\D\IntComp$, Equation \ref{eqn:UPQDDirectProblemSol} implies
\begin{align*}
\alpha&=h(w)=\dfrac{1}{a w}\Phi_{\varphi}\left(\AnalyticIn{|cz_0|^{a}\left(1-\dfrac{\overline{z_0}^{-1}}{z}\right)|cz_0|^{a}\left(1-zz_0^{-1}\right)}{\D}\right)(w)-\dfrac{t}{w}\\
&=\dfrac{|cz_0|^{2a}}{a w}\Phi_{\varphi}\left(\AnalyticIn{1+|z_0|^{-2}-zz_0^{-1}-\dfrac{\overline{z_0}^{-1}}{z}}{\D}\right)(w)-\dfrac{t}{w}\\
&=-\dfrac{z_0^{-1}|cz_0|^{2a}}{a w}c^{-1}w-\dfrac{\widetilde{t}}{w}.
\end{align*}
hence $z_0^{-1}|z_0|^{2a}=-\frac{a \alpha}{c^{2a-1}}=-\overline{\gamma}$. This implies that $z_0\alpha<0$, so we may conclude that $z_0=\overline{\gamma}^{\frac{1}{2a-1}}$, and $\varphi$ takes the form of Equation \ref{eqn:PQDk1MonomialRiemannMapZero}. In particular, whenever $\Omega\ni0$ is a simply connected unbounded domain in $\QD_a(\alpha)$ for $a>0$ and (wlog by change of variables) $\alpha>0$, then $\Omega=\varphi(\D\IntComp)$ with $\varphi$ as above. 

Regarding the univalence of \ref{eqn:PQDk1MonomialRiemannMapZero}, we apply Lemma \ref{thm:DarbouxUnivalence}, which tells us it is sufficient to check injectivity on the boundary. Suppose there existed $z\neq w\in\partial\D$ for which $\varphi(z)=\varphi(w)$. Then if we take $\eta=-\gamma^{\frac{1}{2a-1}}$,
\begin{align*}
|\varphi(z)|&=c\left|\frac{z+\eta}{z+\eta^{-1}}\left(1+\frac{\eta^{-1}}{z}\right)^{\frac{1}{a}}\right|=c\eta\left|\left(1+\frac{\eta^{-1}}{z}\right)^{\frac{1}{a}}\right|=c\eta\left|z+\eta^{-1}\right|^{\frac{1}{a}},
\end{align*}
so $\left|z-\eta^{-1}\right|^2=\left|w-\eta^{-1}\right|^2$. Setting $z=e^{i\theta_0}$ and $w=e^{i\theta_2}$ and expanding, we find that $\cos(\theta_1)=\cos(\theta_2)$, so $w=\overline{z}=z^{-1}$. Thus
\begin{align*}
1&=\dfrac{\varphi(z^{-1})}{\varphi(z)}=cz^{-1}\frac{z^{-1}+\eta}{z^{-1}+\eta^{-1}}\left(1+z\eta^{-1}\right)^{\frac{1}{a}}\left(cz\frac{z+\eta}{z+\eta^{-1}}\left(1+\frac{\eta^{-1}}{z}\right)^{\frac{1}{a}}\right)^{-1}=\frac{\left(1+z^{-1}\eta^{-1}\right)^{2-\frac{1}{a}}}{\left(1+z\eta^{-1}\right)^{2-\frac{1}{a}}}.
\end{align*}
Thus $(1+z^{-1}\eta^{-1})^{2-\frac{1}{a}}=(1+z\eta^{-1})^{2-\frac{1}{a}}$.

$a>\frac{1}{2}$ implies $2-\frac{1}{a}\in(0,2)$, in which case $z\mapsto z^{2-\frac{1}{a}}$ is injective in the closed right half-plane. Moreover, as long as $\eta\geq1$, the images of $\partial\D$ under $z\mapsto 1+z\eta^{-1}$ and $z\mapsto 1+z^{-1}\eta^{-1}$ are in the closed right half plane. Thus $1+z\eta^{-1}=1+z^{-1}\eta^{-1}$, so $z=\pm1$. However recall that $w=z^{-1}$ which, in both cases, implies that $z=w$. Thus $\varphi$ is univalent in $\D\IntComp$ when $|\gamma|\geq1$. On the other hand if $|\gamma|<1$ then $\varphi$ isn't even analytic in $\D\IntComp$. In particular, $\varphi$ is univalent in $\D\IntComp$ if and only if $|\gamma|\geq1$.
\end{proof}

We saw in Theorems \ref{thm:ConstMonomialPQDNoZeroClass} and \ref{thm:ConstMonomialPQDZeroClass} that the family of simply connected $\Omega$ in $\QD_a(\alpha)$ depends crucially on the value of $a$. Another way of understanding this is through Lemmas \ref{lemma:GenComplementDropletQuad} and \ref{lemma:GenQuadComplementDroplet} which tell us $\Omega\in\QD_a(\alpha)$ iff $K:=\Omega^{c}$ is a local droplet of the potential $Q(w)=\frac{|w|^{2a}}{a^2}-2\Re(\alpha w)$. So if $K$ is such a local droplet then, by Lemmas \ref{lemma:HSChains} and \ref{lemma:HSChainZeroCoincidence}, there exists an associated Hele-Shaw chain $\{K_{t}\}_{0<t\leq t_0}$, where $K_{t_0}=K$ such that each element of this chain contains a local minimum of $Q$. Thus $\QD_a(\alpha)$ is nonempty only when $Q$ has a local minimum. For example, in Figure \ref{fig:.4-1-1o3c_Droplets_Finite_All}, we observe that the families of solutions each ``nucleate'' at the unique local minimum of $Q$, located at $0$, $0$, and $\left(a \alpha\right)^{\frac{1}{2a-1}}$ respectively.

\subsection{Classification of Monomial PQDs}
In this section, we are interested in characterizing PQDs with a monomial quadrature function - that is, $\Omega\in\QD_a(\alpha k w^{k-1})$ for some $a>0$, $\alpha\in\C\setminus\{0\}$, and $k\in\Z_{+}$. Firstly note that we can wlog assume that $\alpha>0$ because if $S_a$ is the associated generalized Schwarz function then
$$\widetilde{S_a}(w):=\left(\frac{|\alpha|}{\alpha}\right)^{\frac{1}{k}}S_a\left(w\left(\frac{|\alpha|}{\alpha}\right)^{\frac{1}{k}}\right)\dEquals\left(\frac{|\alpha|}{\alpha}\right)^{\frac{1}{k}}\dfrac{1}{a}\overline{w\left(\frac{|\alpha|}{\alpha}\right)^{\frac{1}{k}}}\left|w\left(\frac{|\alpha|}{\alpha}\right)^{\frac{1}{k}}\right|^{2(a-1)}=\dfrac{1}{a}\overline{w}|w|^{2(a-1)}$$
is a generalized Schwarz function for $\widetilde{\Omega}:=\left(\frac{\alpha}{|\alpha|}\right)^{\frac{1}{k}}\Omega$, so it is a PQD. It's then straightforward to verify that $\widetilde{\Omega}\in\QD_a(|\alpha|kw^{k-1})$.

\begin{theorem}\label{thm:GenMonomialPQDNoZero}
Fix $a>0$, $\alpha\in\C\setminus\{0\}$, $k\in\Z_{+}$. If $0\notin\Omega$ is a $\Z_k-$rotationally symmetric domain $(e^{\frac{2\pi i}{k}}\Omega=\Omega)$, then $\Omega\in\QD_a(\alpha kw^{k-1})$ iff $\Omega^{k}\in\QD_{\frac{a}{k}}(\alpha k^2)$. In this case, $\Omega=\widetilde{\varphi}(\D\IntComp)$, where
\begin{equation}\label{eqn:MonomialPQDNoZeroRiemannMap}
\widetilde{\varphi}(z)=cz\left(1-\frac{\gamma_k}{z^{k}}\right)^{\frac{1}{a}},\;\;\;\;\gamma_k=-\frac{a \overline{\alpha}k}{c^{2a-k}}.
\end{equation}
is univalent in $\D\IntComp$.
\end{theorem}

Toward proving this, we note the following relationship between the $k=1$ case and the general case:
\begin{lemma}\label{lemma:MonomialPQDReduction}
Fix $a>0$, $\alpha\in\C\setminus\{0\}$, and $k\in\Z_+$. If $\Omega\in \QD_{\frac{a}{k}}(\alpha k^2)$ then $\{w:w^k\in\Omega\}\in\QD_a\left(\alpha kw^{k-1}\right)$.
\end{lemma}

\begin{proof}[Proof of lemma \ref{lemma:MonomialPQDReduction}]
Suppose that $\Omega\in\QD_{\frac{a}{k}}(\alpha)$ and define $\widetilde{\Omega}:=\{w:w^k\in\Omega\}$. By Theorem \ref{theorem:EquivPQDChars}, $\Omega$ admits a generalized Schwarz function $S_{\frac{a}{k}}(w)\dEquals \frac{k}{a}\overline{w}|w|^{2\left(\frac{a}{k}-1\right)}$, so
$$S_a(w):=\dfrac{1}{k}w^{k-1}S_{\frac{a}{k}}(w^k)\dEquals \dfrac{1}{k}w^{k-1}\dfrac{k}{a}\overline{w}^k|w|^{2\left(a-k\right)}=\dfrac{1}{a}\overline{w}|w|^{2\left(a-1\right)}$$
is a generalized Schwarz function for $\widetilde{\Omega}$. Thus, by Theorem \ref{theorem:EquivPQDChars}, $\widetilde{\Omega}\in\QD_a$ and
$$\widetilde{h}(w)=\AnalyticIn{S_a(w)}{\widetilde{\Omega}\IntComp}=\AnalyticIn{\dfrac{1}{k}w^{k-1}S_{\frac{a}{k}}(w^k)}{\widetilde{\Omega}\IntComp}=\AnalyticIn{\frac{\alpha}{k}w^{k-1}+\frac{1}{k}w^{k-1}G(w^k)}{\widetilde{\Omega}\IntComp}=\frac{\alpha}{k}w^{k-1}.$$
The result follows from the substitution $\alpha\mapsto \alpha k^2$.
\end{proof}

It is not true however that every element of $\QD_a(\alpha kw^{k-1})$ can be represented this way. To wit,
\begin{example}\label{ex:PQDSymhNoSymK}
Consider the map $\varphi:\D\IntComp\rightarrow\Ch$ given by
$$\varphi(z):=cz\dfrac{z+u}{z+u^{-1}}\sqrt{\left(1-\frac{2\gamma}{u^{3}}z^{-1}\right)\left(1+\frac{1}{u}z^{-1}\right)}=c\left(z+u\right)\sqrt{\dfrac{z-2\gamma u^{-3}}{z+u^{-1}}}$$
where $\gamma=-\frac{2\alpha}{c^2}$, $u=\sqrt{\sqrt{\gamma\left(\gamma+4\right)}-\gamma}$, and $c,\alpha>0$. One can show that this map is univalent in $\D\IntComp$ precisely when $\frac{2\alpha}{c^2}=-\gamma\geq\frac{5\sqrt{13}-1}{4}\approx4.257$ via Lemma \ref{thm:DarbouxUnivalence} and solving the equation $\varphi^2(z)=\varphi^2(\xi)$ over $|z|=|\xi|=1$ (taking its conjugate, we obtain $\varphi^2(z^{-1})=\varphi^2(\xi^{-1})$, so we may eliminate $\xi$ to obtain a polynomial equation in $z$ with coefficients depending only on $\gamma$).

Then note that $\varphi_{\rm out}^2(z)=r^{\#}(z)=u^2c^2\left(1-\frac{2\gamma}{u^{3}}z^{-1}\right)\left(1+\frac{1}{u}z^{-1}\right)$ has the form required by Theorem \ref{thm:SCPQDCharacterization}, so $\Omega:=\varphi(\D\IntComp)$ is a PQD when $\varphi$ is univalent. Applying Equation \ref{eqn:PQDHFormulaOld} with $r(z)=u^2c^2(1-2\gamma u^{-3}z)(1+u^{-1}z)$ we obtain

\begin{align*}
    h(w)&=\dfrac{1}{2}\AnalyticIn{\dfrac{\left(rr^{\#}\right)\circ\psi(w)}{w}}{\Omega\IntComp}=\dfrac{c^4}{2}\AnalyticIn{\dfrac{-2\gamma\psi^2(w)+u^{-3}(u^2-2\gamma)(u^4-2\gamma)\psi(w)+O(1)}{w}}{\Omega\IntComp}\\
    &=\AnalyticIn{-c^2\gamma w+\dfrac{c^3}{2u}\left(u^4+2u^2\gamma-4\gamma\right)+O(w^{-1})}{\Omega\IntComp}\\
    &=2\alpha w+\dfrac{c^3}{2u}\left(u^4+2u^2\gamma-4\gamma\right)
\end{align*}

And $u$, as defined above, is a root of $u^4+2\gamma u^2-4\gamma$, so $h(w)=2\alpha w$. That is, $\Omega\in\Omega_2\left(2\alpha w\right)$. See Figure \ref{fig:PQDSymhNoSymK}. It is of note that $-\Omega$ is also in $\QD_2\left(2\alpha w\right)$ (consider taking $c\mapsto-c$).

\begin{figure}[ht]
  \centering
    \includegraphics[width=.45\textwidth]{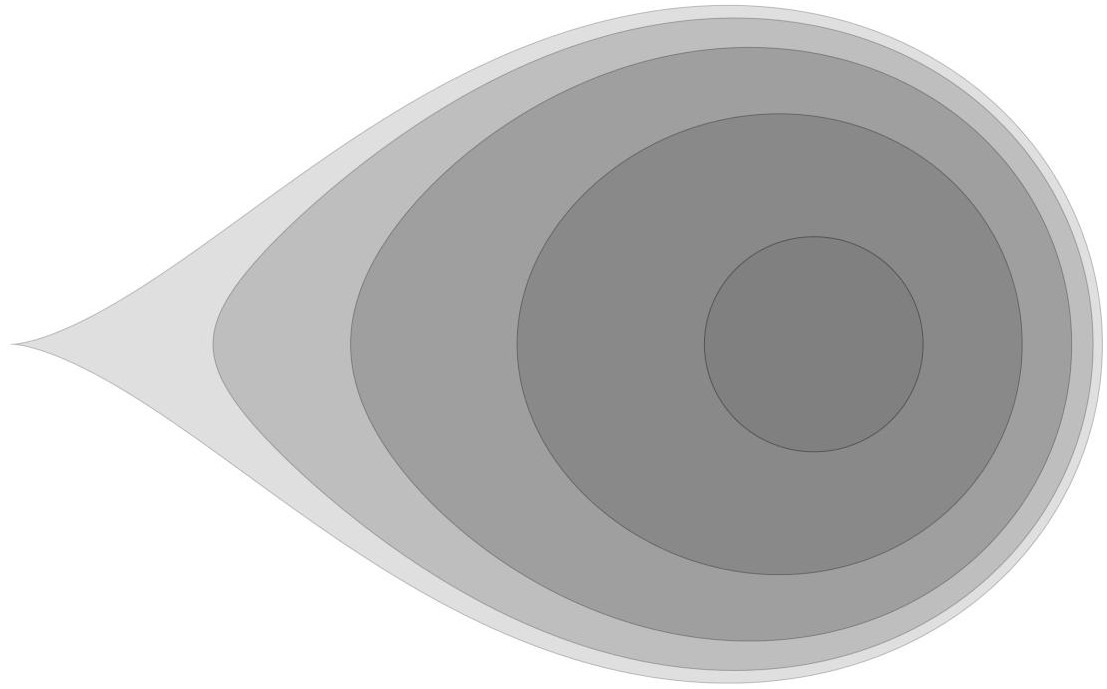}
    \caption{An asymmetric family (complements of the shaded regions) in $\QD_{2}\left(\frac{w}{2}\right)$ for $c\in\{.09, .2, .27, .31, 0.34272\}$ (Example \ref{ex:PQDSymhNoSymK})}\label{fig:PQDSymhNoSymK}
\end{figure}
\end{example}
\;\\
Despite the failure of the converse of Lemma \ref{lemma:MonomialPQDReduction} in the general case, it can nevertheless be recovered with the additional requirement of $\Z_k-$rotational symmetry.
\begin{lemma}[Converse to Lemma \ref{lemma:MonomialPQDReduction}]\label{lemma:MonomialPQDReductionConverse}
Fix $a>0$, $\alpha\in\C\setminus\{0\}$, and $k\in\Z_+$. If $\Omega\in\QD_a(\alpha kw^{k-1})$ and is $\Z_k-$rotationally symmetric ($e^{\frac{2\pi i}{k}}\Omega=\Omega$), then $\Omega^k\in\QD_{\frac{a}{k}}(\alpha k^2)$.
\end{lemma}
\begin{proof}[Proof of Lemma \ref{lemma:MonomialPQDReductionConverse}]

Let $\Omega$ be an unbounded domain as above so that for each $f\in\A_0(\Omega^k)$, $f(w^k)\in\A_0(\Omega)$, and
$$\int_{\Omega^k}f(w)|w|^{2\left(\frac{a}{k}-1\right)}dA(w)=\int_{\Omega}f(w^k)k^2|w|^{2(a-1)}dA(w)=\oint_{\partial\Omega}f(w^k)\alpha k^3w^{k-1}dw=\oint_{\partial\Omega^k}f(w)\alpha k^2dw.$$
Thus $\Omega^k\in\QD_{\frac{a}{k}}(\alpha)$. The bounded case is completely analogous. 
\end{proof}

We're now equipped to prove Theorem \ref{thm:GenMonomialPQDNoZero}.
\begin{proof}[Proof of Theorem \ref{thm:GenMonomialPQDNoZero}]
Let $\Omega$ be a $\Z_k-$rotationally symmetric domain. If $\Omega\in\QD_a(\alpha kw^{k-1})$ then by Lemma \ref{lemma:MonomialPQDReductionConverse}, $\Omega^k\in\QD_{\frac{a}{k}}(\alpha k^2)$. On the other hand, if $\Omega^k\in\QD_{\frac{a}{k}}(\alpha k^2)$ then by Lemma \ref{lemma:MonomialPQDReduction}, $\Omega=\{w:w^k\in\Omega^k\}$ is in $\QD_a(\alpha kw^{k-1})$.

Now, if $\varphi$ is the Riemann map associated to $\Omega^k$ and $0\notin\Omega$ (and thus also $0\notin\Omega^k$) then by Equation \ref{eqn:BasicPQDNoZeroRiemannMap}, there exists a $c>0$ for which $\varphi(z)=c^kz\left(1-\gamma_kz^{-1}\right)^{\frac{k}{a}}$, where $\gamma_k=-\frac{\left(\frac{a}{k}\right) \overline{(\alpha k^2)}}{(c^k)^{2\frac{a}{k}-1}}=-\frac{a \overline{\alpha}k}{c^{2a-k}}$. Moreover, by Lemma \ref{lemma:PolynomialUPQDClass}, there exists a polynomial $p$ of degree $k$ with $p(0)=1$ for which $\widetilde{\varphi}(z)=c'zp^{\#}(z)^{\frac{1}{a}}$. Thus,
\begin{align*}
    (c')^kz^kp^{\#}(z)^{\frac{k}{a}}&=\left(c'zp^{\#}(z)^{\frac{1}{a}}\right)^k=\widetilde{\varphi}^k(z)=\varphi(z^k)=c^kz^k\left(1-\gamma_kz^{-k}\right)^{\frac{k}{a}}.
\end{align*}
So, $(c')^kp^{\#}(z)^{\frac{k}{a}}=c^k\left(1-\gamma_kz^{-k}\right)^{\frac{k}{a}}$. Taking $z\to\infty$, we find that $c'=c$, and so $p^{\#}(z)^{\frac{1}{a}}=\left(1-\gamma_kz^{-k}\right)^{\frac{1}{a}}$. The result follows.
\end{proof}

Figure \ref{fig:3-3Droplets} exhibits each of the cases/phenomena described in Theorems \ref{thm:ConstMonomialPQDNoZeroClass} and \ref{thm:ConstMonomialPQDZeroClass}. Furthermore, by exploiting Lemma \ref{lemma:MonomialPQDReduction} and Equation \ref{eqn:MonomialPQDNoZeroRiemannMap}, we obtain explicit families $\{\Omega_t\}_t\subseteq\QD_a(\alpha kw^{k-1})$, exhibited in Figures \ref{fig:2-6-1o3c_Droplets_Finite_And_Travelling_Wave} and \ref{fig:234-2Droplets}.

\begin{comment}
\begin{figure}[ht]
  \centering
    \includegraphics[width=.99\textwidth]{Images/PQD.5-1-1o3c_Droplets_TravellingWave_Nested.jpg}
    \caption{$\Omega_t\in\QD_{\alpha}(c)$ for $c=\frac{1}{3}$ and $\alpha=\frac{1}{2}$, with $a\in\{.05,.1,.2,.3,.4,.5\}$}\label{fig:.5-1-1o3c_Droplets_TravellingWave}\vspace{1.5em}
\end{figure}

\begin{figure}[ht]
  \centering
    \includegraphics[scale=.85]{Images/2-1_PQD_Nested_Family_PreAndPostCrit.jpg}
    \caption{$\Omega_t\in\QD_{\alpha}(c)$ for $c=2$ and $\alpha=2>\frac{1}{2}$, with $a\in\{.1, .4, .8, .98, 1, 1.02, 1.08\}$}\label{fig:2-1Droplets}\vspace{1.5em}
\end{figure}
\end{comment}
\FloatBarrier

\begin{figure}[ht]
  \centering
    \includegraphics[height=.27\textwidth]{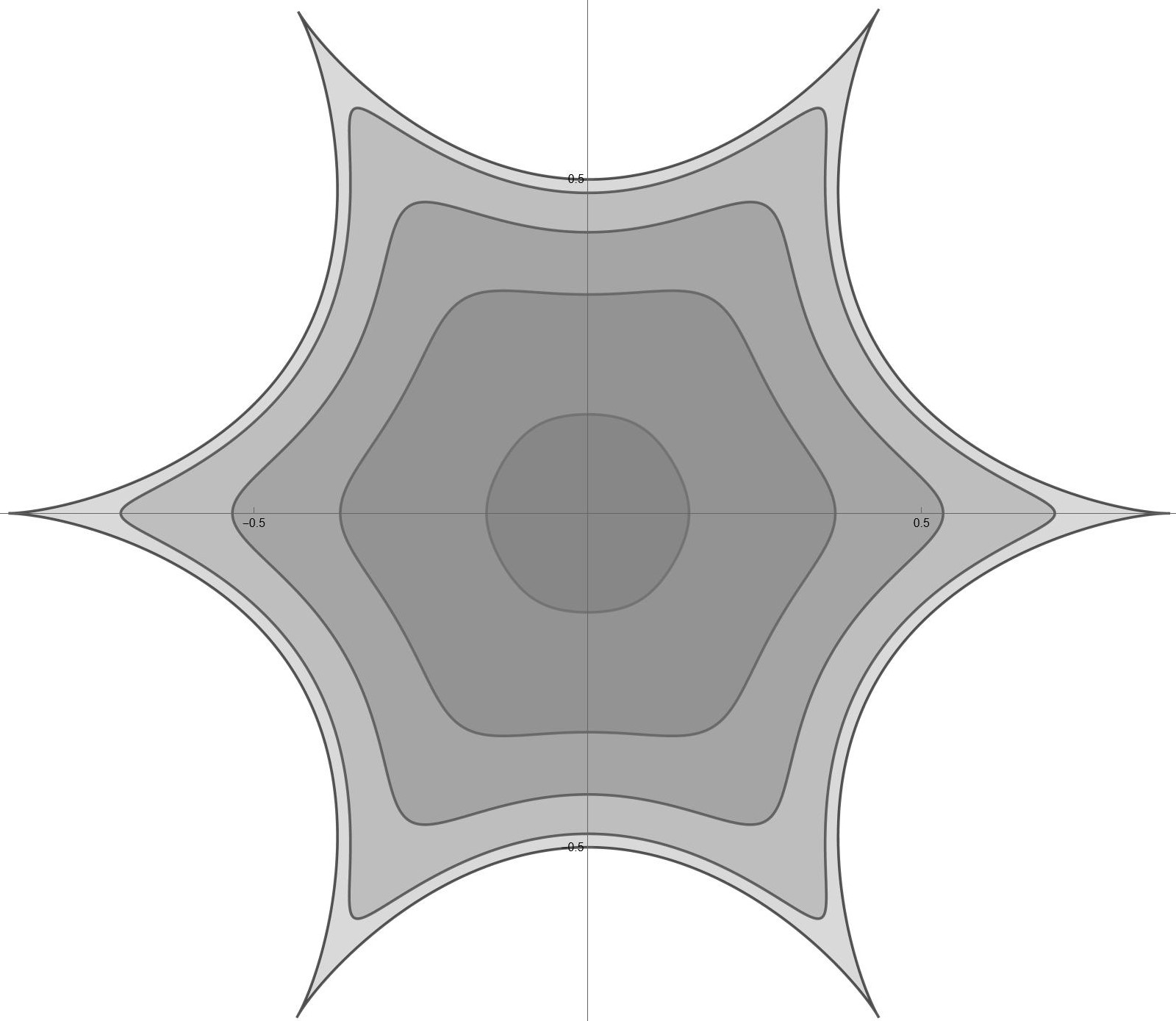}\;\;\;\;\includegraphics[height=.27\textwidth]{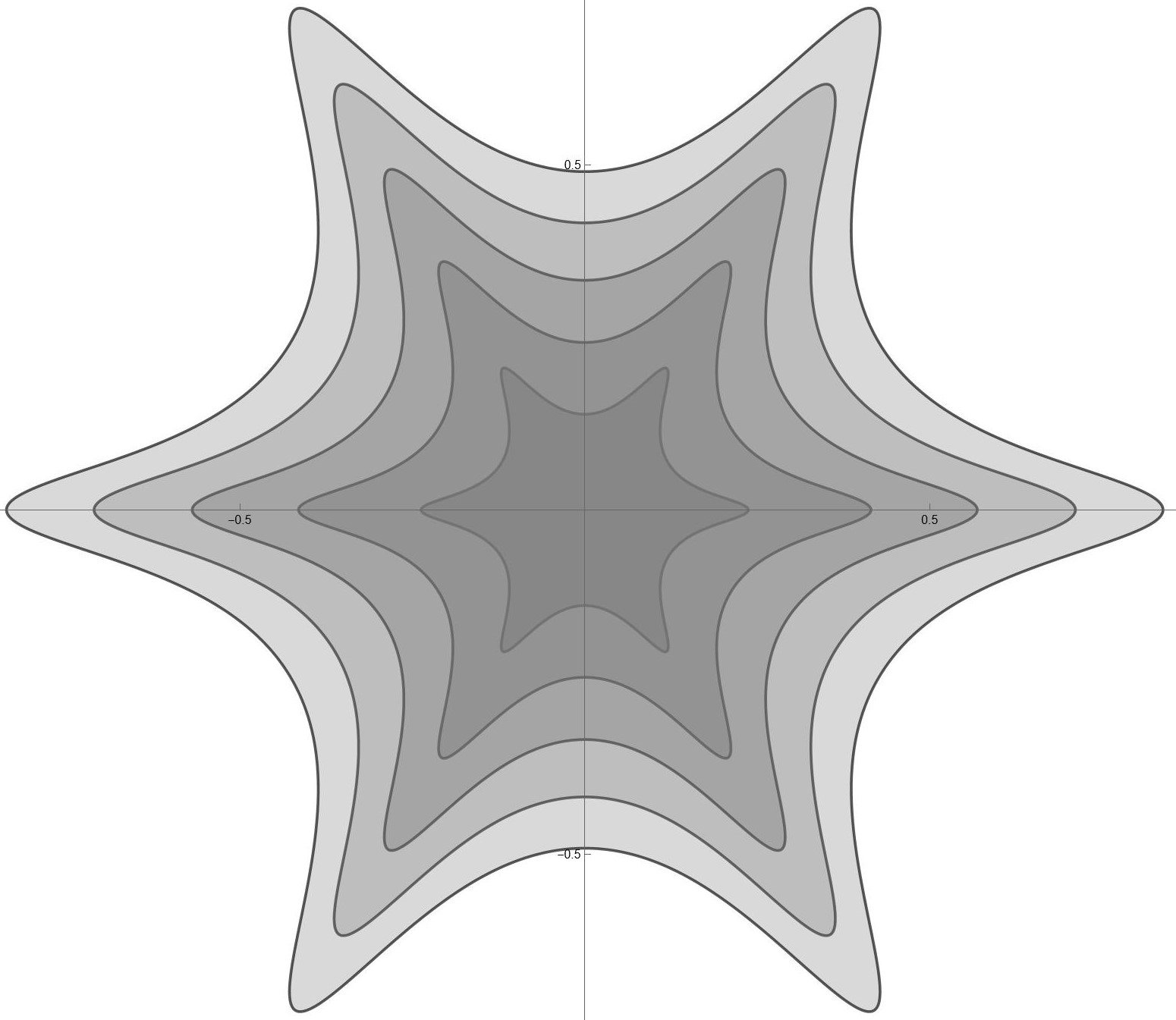}\;\;\;\;\includegraphics[height=.27\textwidth]{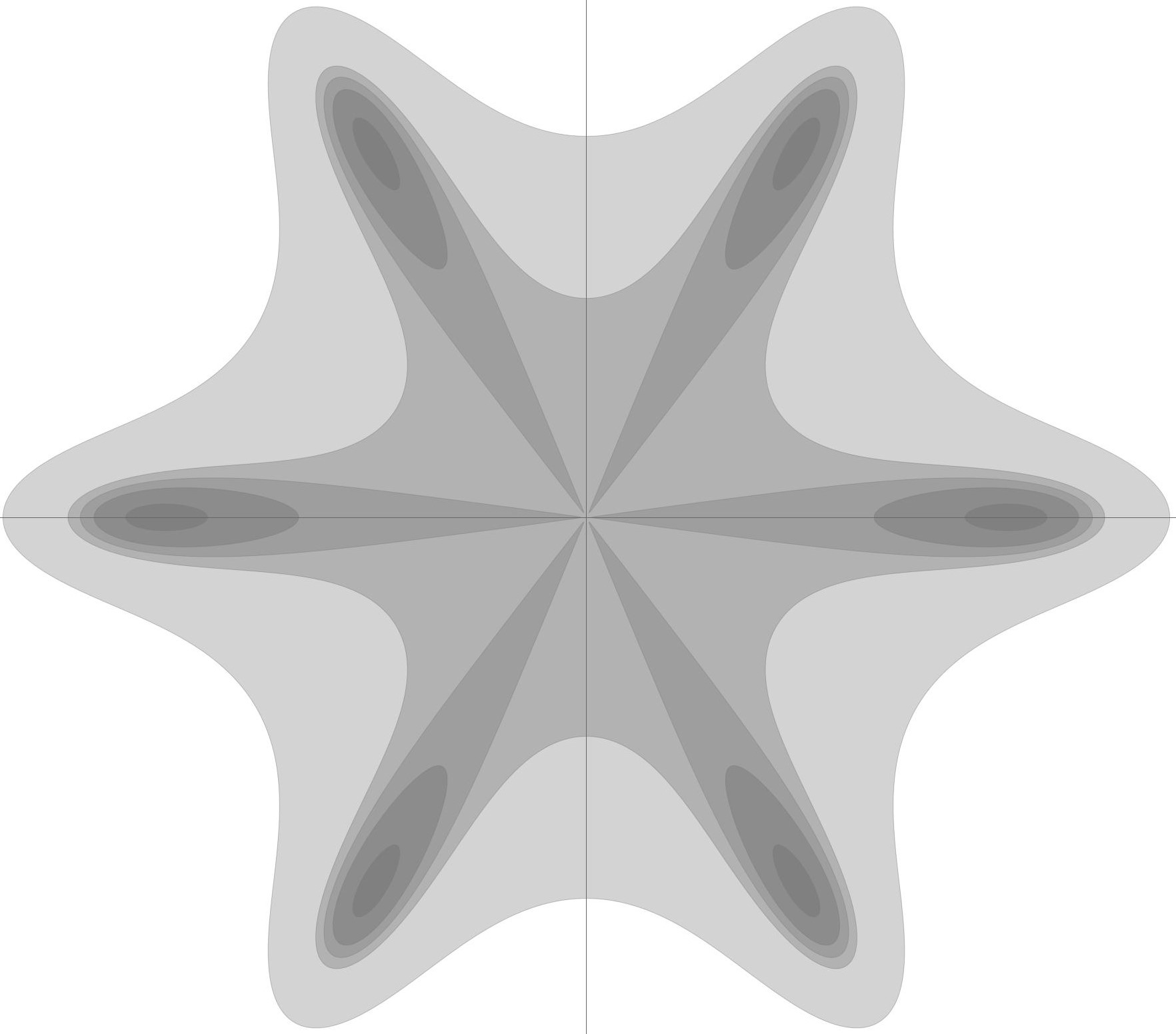}
    \caption{Families $\{\Omega_t\}_t$ of PQDs (complements of the shaded regions) in $\QD_{2}\left(\frac{w^{k-1}}{2}\right)$, $\QD_{3}\left(\frac{2}{9}w^{k-1}\right)$, and $\QD_{4}\left(\frac{w^{k-1}}{8}\right)$ with $k=6$ (left to right). When $a=2$, $k>2a$, so the family has a maximal element. When $a=3$, $k=2a$, so this is the ``traveling wave'' case. Finally, when $a=4$, $k<2a$, so we obtain a two-phase family which is initially multiply connected and is simply connected past the critical time.}\label{fig:2-6-1o3c_Droplets_Finite_And_Travelling_Wave}
\end{figure}

\begin{figure}[ht]
  \centering
    \includegraphics[scale=.6]{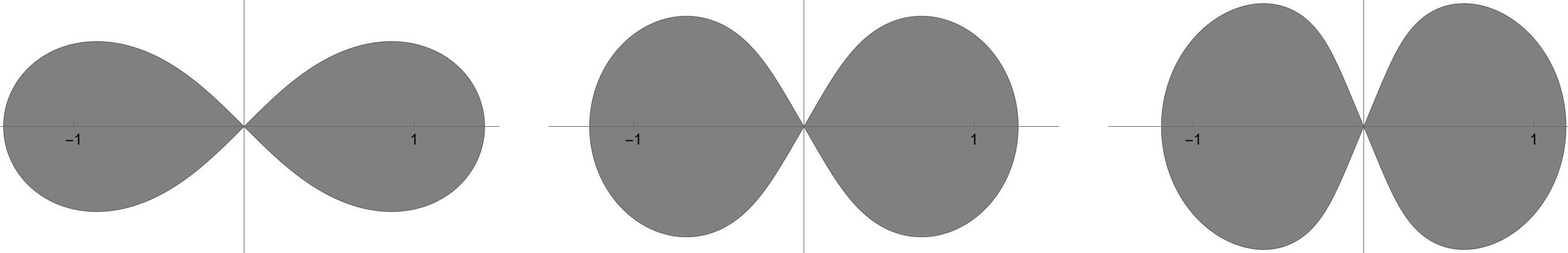}
    \caption{$\Omega\in\QD_2\left(\frac{w}{2}\right)$, $\QD_3\left(\frac{2}{9}w\right)$, and $\QD_\alpha\left(\frac{w}{8}\right)$ (left to right, complements of the shaded regions). The interior angles at $0$ are $\frac{\pi}{a}$ radians.}\label{fig:234-2Droplets}
\end{figure}

\subsection{One Point PQDs of Higher Order}

Let $X_k(w_0)$ be the space of rational functions $=0$ at $\infty$ with a unique pole of order $k\in\Z_{+}$ at $w_0\in\C$.\footnote{i.e. $\sum_{j=1}^{k}\frac{\alpha_j}{(w-w_0)^j}$ for some collection of $\alpha_j\in\C$.} 

\begin{lemma}\label{lemma:PQDOnePtClassHigherOrder}
Fix $a>0$ and let $\Omega$ be a bounded, simply connected domain with Riemann map $\varphi:\D\rightarrow\Omega$ such that $\varphi(0)=w_0$, and $\varphi'(0)>0$.
\begin{enumerate}
\item If $0\notin\Omega$, then there exists $h\in X_k(w_0)$ such that $\Omega\in\QD_a\left(h\right)$ if and only if there exists a polynomial $p$ of degree $k$ such that
\begin{equation}\label{eqn:OnePtUPQDClassMap}
\varphi(z)=p(z)^{\frac{1}{a}}.
\end{equation}
\item If $w_0=0$ (so $0\in\Omega$), then there exists $h\in X_k(0)$ such that $\Omega\in\QD_a\left(h\right)$ if and only if there exists a polynomial $p$ of degree $k-1$ such that 
\begin{equation}\label{eqn:OnePtUPQDClassMapCenterZero}
\varphi(z)=zp(z)^{\frac{1}{a}}.
\end{equation}
\item If $0\in\Omega$ and $w_0\neq0$, then there exists $h\in X_k(w_0)$ such that $\Omega\in\QD_a(h)$ if and only if there exists a polynomial $p$ of degree $k$ with a root at $\overline{z_0}^{-1}$ such that
\begin{equation}\label{eqn:OnePtUPQDClassMapZero}
\varphi(z)=b_{z_0}(z)p(z)^{\frac{1}{a}}.
\end{equation}
\end{enumerate}
\end{lemma}

\begin{proof}[Proof of lemma \ref{lemma:PQDOnePtClassHigherOrder}]\label{proof:PQDOnePtClassHigherOrder}\;\\
We will only carry out the proof of case (1) because cases (2) and (3) follow by similar arguments. Let $\Omega\not\ni0$ be a bounded domain and $\varphi:\D\rightarrow\Omega$ a Riemann map for $\Omega$ for which $\varphi(0)=w_0$ and $\varphi'(0)>0$.

Suppose there exists $h\in X_k(w_0)$ such that $\Omega\in\QD_a\left(h\right)$. Then by Theorem \ref{thm:GenBPQDInvProbFormula} and Corollary \ref{corr:GenBPQDFTFormula}, there exists a polynomial $q$ of degree $k-1$ for which
$\varphi(z)=\left(w_0^a+zq(z)\right)^{\frac{1}{a}}$. Thus there exists a polynomial $p$ of degree $k$ for which $\varphi(z)=p(z)^{\frac{1}{a}}$.

Conversely, suppose there exists a polynomial $p$ of degree $k$ such that $\varphi(z)=p(z)^{\frac{1}{a}}$. Then by Theorem \ref{thm:SCPQDCharacterization}, $\Omega\in\QD_a(h)$ for some rational $h$. Moreover by Theorem \ref{thm:BPQDDirectProblemSol},
\begin{align*}
h(w)&=\dfrac{1}{a w}\Phi_{\varphi}\left(\AnalyticIn{pp^{\#}}{\D\IntComp}\right)(w)+\dfrac{t}{w}=\dfrac{1}{a w}\Phi_{\varphi}\left(\AnalyticIn{\sum_{|j|\leq k}\beta_jz^j}{\D\IntComp}\right)(w)+\dfrac{t}{w}\\
&=\dfrac{1}{a w}\Phi_{\varphi}\left(\sum_{j=1}^{k}\beta_{-j}z^{-j}\right)(w)+\dfrac{t}{w}\\
&=\dfrac{1}{w}\sum_{j=1}^{k}\dfrac{\gamma_j}{(w-\varphi(0))^j}+\dfrac{t}{w}=\sum_{j=1}^{k}\dfrac{\widetilde{\gamma_j}}{(w-w_0)^j}+\dfrac{\widetilde{t}}{w}.
\end{align*}
But then because $h\in\A_0(\Omega\IntComp)$, we find that $\widetilde{t}=0$, so $h\in X_k(w_0)$.
\end{proof}

\begin{lemma}\label{lemma:PQDOnePtInftyClassHigherOrder}
Take $a>0$ and $\Omega$ an unbounded, simply connected domain with conformal radius $c$ admitting a Riemann map $\varphi:\D\IntComp\rightarrow\Omega$ with $\varphi'(\infty)=c$, and $X_k(w_0)$ the space of rational functions $=0$ at $\infty$ with a unique pole of order $k\in\Z_{+}$ at $w_0\in\Omega$. Then
\begin{enumerate}
\item If $0\notin\Omega$, then there exists $h\in X_k(w_0)$ such that $\Omega\in\QD_a\left(h\right)$ if and only if there exists $z_0\in\D\IntComp$ with $\varphi\left(z_0\right)=w_0$, and a polynomial $p$ of degree $k$ with $p(0)=c^a$ and without a root at $z_0$ such that
\begin{equation}
\varphi(z)=z\left(\dfrac{p^{\#}(z)}{\left(1-\frac{\overline{z_0}^{-1}}{z}\right)^{k}}\right)^{\frac{1}{a}}.
\end{equation}
\item If $0\in\Omega$ then there exists $h\in X_k(w_0)$ such that $\Omega\in\QD_a(h)$ if and only if there exist $z_0,z_1\in\D\IntComp$ with $\varphi(z_0)=w_0$, $\varphi(z_1)=0$, and a polynomial $p$ of degree $k$ with $p(0)=|cz_1|^a$ a root at $z_1$ and without a root at $z_0$ such that
\begin{equation}
\varphi(z)=zb_{z_1}(z)\left(\dfrac{p^{\#}(z)}{\left(1-\frac{\overline{z_0}^{-1}}{z}\right)^{k}}\right)^{\frac{1}{a}}.
\end{equation}
\end{enumerate}
\end{lemma}
The proof \ref{proof:PQDOnePtInftyClassHigherOrder} is given below.

\begin{proof}[Proof of lemma \ref{lemma:PQDOnePtInftyClassHigherOrder}]\label{proof:PQDOnePtInftyClassHigherOrder}\;\\
Take $a>0$ and $\Omega$ an unbounded, simply connected domain with conformal radius $c$ admitting a Riemann map $\varphi:\D\IntComp\rightarrow\Omega$ with $\varphi'(\infty)=c$.\\
$\mathbf{(1)}$ \noindent{\bf Forward direction:}\\
Suppose $0\notin\Omega$ and there exists $h\in X_k(w_0)$ such that $\Omega\in\QD_a\left(h\right)$. Then by Theorem \ref{thm:GenUPQDInvProbFormula} and corollary \ref{corr:GenUPQDFTFormula},
\begin{align*}
\varphi(z)&=z\left(c^a+\dfrac{1}{z}\sum_{j=1}^{k}\frac{\beta_j}{(z^{-1}-\overline{z_0})^j}\right)^{\frac{1}{a}}=z\left(\dfrac{c^aq(z)}{\left(z-\overline{z_0}^{-1}\right)^k}\right)^{\frac{1}{a}},
\end{align*}
where $\varphi(z_0)=w_0$, and $q$ is a monic polynomial of degree $k$ with $q(\overline{z_0}^{-1})\neq0$. Thus there exists a polynomial $p$ of degree $k$ with $p(z_0)\neq0$, and $p(0)=c^a$, for which $\varphi(z)=z\left(\frac{p^{\#}(z)}{\left(1-\frac{\overline{z_0}^{-1}}{z}\right)^k}\right)^{\frac{1}{a}}$.\\
$\mathbf{(1)}$ \noindent{\bf Backward direction:}\\
Conversely, suppose there exists $z_0\in\D\IntComp$ with $\varphi\left(z_0\right)=w_0$, and a polynomial $p$ of degree $k$ with $p(0)=c^a$ and without a root at $z_0$ such that $\varphi(z)=z\left(\frac{p^{\#}(z)}{\left(1-\frac{\overline{z_0}^{-1}}{z}\right)^k}\right)^{\frac{1}{a}}$. Then, by Theorem \ref{thm:UPQDDirectProblemSol},
\begin{align*}
h(w)&=\dfrac{1}{a w}\Phi_{\varphi}\left(\AnalyticIn{\dfrac{p^{\#}(z)}{\left(1-\frac{\overline{z_0}^{-1}}{z}\right)^k}\dfrac{p(z)}{\left(1-zz_0^{-1}\right)^k}}{\D}\right)(w)-\dfrac{t}{w}\\
&=\dfrac{1}{a w}\Phi_{\varphi}\left(\AnalyticIn{\dfrac{z_0^kz^kp^{\#}(z)}{\left(z-\overline{z_0}^{-1}\right)^k}\dfrac{p(z)}{\left(z_0-z\right)^k}}{\D}\right)(w)-\dfrac{t}{w}.
\end{align*}
Neither $p$ nor $p^{\#}$ have a root at $z_0$, so the argument of $\Phi_{\varphi}$ is meromorphic in $\D\IntComp$ and has a unique pole of order $k$ at $z_0$ in $\D\IntComp$. Hence, by partial fraction decomposition,
\begin{align*}
h(w)&=\dfrac{1}{a w}\sum_{j=1}^{k}\dfrac{\beta_j}{(w-\psi(z_0))^j}+\dfrac{t}{w}=\sum_{j=1}^{k}\dfrac{\widetilde{\beta_j}}{(w-w_0)^j}+\dfrac{\widetilde{t}}{w}.
\end{align*}
Finally $0\in\Omega\IntComp$ but $h\in\A(\Omega\IntComp)$, so $\widetilde{t}=0$, and we conclude that $h\in X_k(w_0)$.\\\\
The argument is similar for case $\mathbf{(2)}$.
\end{proof}

\subsection{Polynomial PQDs}

Let $\Omega$ be a simply connected unbounded domain not containing zero with conformal radius $c>0$. Now take $a>0$, $k\in\Z_{\geq0}$ and suppose $\Omega\in\QD_a(h)$ for some polynomial $h$ of degree $k$. Then by Theorem \ref{thm:UPQDDirectProblemSol} and corollary \ref{corr:GenUPQDFTFormula}, there exists a polynomial $p$ of degree $k$ for which $\varphi(z)=z\left(c^a+z^{-1}p^{\#}(z)\right)^{\frac{1}{a}}$.

Conversely, suppose $0\notin\Omega$ and there exists a polynomial $p$ of degree $k$ for which $\varphi(z)=z\left(c^a+z^{-1}p^{\#}(z)\right)^{\frac{1}{a}}$. Then by Equation \ref{eqn:UPQDDirectProblemSol},
$$awh(w)=\Phi_{\varphi}\left(\AnalyticIn{rr^{\#}}{\D}\right)(w)-at=\Phi_{\varphi}\left(\AnalyticIn{\left(c^a+z^{-1}p^{\#}(z)\right)\left(c^a+zp(z)\right)}{\D}\right)(w)-at$$
The argument of the Faber transform is a Laurent polynomial of the form $\sum_{|j|\leq k+1}\beta_jz^j$, so its analytic part in $\D$ is a polynomial of degree $k+1$. The Faber transform of a polynomial of degree $k+1$ is a polynomial of degree $k+1$, so we find that $h(w)=\frac{\beta}{w}+q(w)$ for some polynomial $q$ of degree $k$. Finally, $h\in\A(\Omega\IntComp)$ implies $\beta=0$ when $0\notin\Omega$. Thus $h$ is a polynomial of degree $k$. This proves the first case of the following lemma.

\begin{lemma}\label{lemma:PolynomialUPQDClass}
Let $\Omega$ be an unbounded, simply connected domain with conformal radius $c$, admitting a Riemann map $\varphi:\D\IntComp\rightarrow\Omega$. Then for each $a>0$ and $k\in\Z_{\geq0}$,
\begin{enumerate}
\item If $0\notin\Omega$, then there exists a polynomial $h$ of degree $k$ such that $\Omega\in\QD_a(h)$ if and only if there exists a polynomial $p$ of degree $k+1$ with $p(0)=c^a$ such that
\begin{equation}\label{eqn:PolynomialUPQDClassMap}
\varphi(z)=zp^{\#}(z)^{\frac{1}{a}}.
\end{equation}
\item If $0\in\Omega$, then there exists a polynomial $h$ of degree $k$ such that $\Omega\in\QD_a(h)$ if and only if there exists a polynomial $p$ of degree $k+1$ with $p(0)=|cz_0|^a$ and a root at $z_0$ for which
\begin{equation}\label{eqn:PolynomialUPQDClassMapZero}
\varphi(z)=zb_{z_0}(z)p^{\#}(z)^{\frac{1}{a}}.
\end{equation}
\end{enumerate}
\end{lemma}
(recall that $b_{\lambda}(z)=\frac{\overline{\lambda}}{|\lambda|}\frac{z-\lambda}{z\overline{\lambda}-1}$ is a Blaschke factor). Note that, by applying Lemma \ref{lemma:PolynomialPowerUnivalence} to case (1) for $k\leq 2a$, we find that $\varphi$ is univalent in $\D\IntComp$ if and only if each of the roots of $p$ are contained in $\D^{c}$. In particular, the set of unbounded PQDs not containing zero with polynomial quadrature function $h$ of degree $k\leq2a$ is parameterized by the set of polynomials $p$ of degree $k+1$ for which $z\mapsto zp^{\#}(z)^{\frac{1}{a}}$ is univalent in $\D\IntComp$. The proof of case 2 is given below.

\begin{proof}[Proof of lemma \ref{lemma:PolynomialUPQDClass}]\label{proof:PolynomialUPQDClass}\;\\
$\mathbf{(2)}$ \noindent{\bf Forward direction:} Let $0\in\Omega\in\QD_a(h)$ be a simply connected and unbounded domain of conformal radius $c>0$ for some polynomial $h$ of degree $k\in\Z_{\geq0}$ and $a>0$. Then by Theorem \ref{eqn:UPQDDirectProblemSol} and corollary \ref{corr:GenUPQDFTFormula}, there exists a polynomial $p$ of degree $k+1$ with $p(0)=|cz_0|^a$ for which $\varphi(z)=zb_{z_0}(z)p^{\#}(z)^{\frac{1}{a}}$ is a Riemann map for $\Omega$. Then by Equation \ref{eqn:PQDHFormulaOld}, we have that
\begin{align*}
h(w)&=\dfrac{1}{a}\AnalyticIn{\dfrac{\left(pp^{\#}\right)\circ\psi(w)}{w}}{\Omega\IntComp}=\dfrac{1}{a}\AnalyticIn{\dfrac{p\circ\psi(w)}{w}p^{\#}\circ\psi(w)}{\Omega\IntComp}
\end{align*}
Note that $p^{\#}$ has no roots in $\D\IntComp$ (otherwise $\varphi$ wouldn't be analytic in $\D\IntComp$), so if $p(z_0)\neq0$, the argument of the RHS must have a pole at zero, which would imply $h$ has a pole at $0$, a contradiction. Thus $p(z_0)=0$.

\noindent $\mathbf{(2)}$ \noindent{\bf Backward direction:} On the other hand, if there exists a polynomial $p$ of degree $k+1$ with a root at $z_0$ for which $\varphi(z)=zb_{z_0}(z)p^{\#}(z)^{\frac{1}{a}}$ then by Theorem \ref{thm:SCPQDCharacterization}, $\Omega\in\QD_a(h)$ for some rational $h$. We saw above that Equation \ref{eqn:PQDHFormulaOld} tells us
\begin{align*}
h(w)&=\dfrac{1}{a}\AnalyticIn{\dfrac{p\circ\psi(w)}{w}p^{\#}\circ\psi(w)}{\Omega\IntComp},
\end{align*}
the argument of which is meromorphic in $\Omega$ and has no poles in $\Omega$, except a pole of order $k$ at $\infty$. Thus $h$ is a polynomial of degree $k$.
\end{proof}

\begin{example}
$\varphi(z)=z\left(1+\frac{1}{2}z^{-1}+\frac{3}{4}z^{-2}\right)^{\frac{1}{2}}$, is univalent in $\D\IntComp$ by lemma \ref{lemma:PolynomialPowerUnivalence} because its roots, $\frac{-1\pm\sqrt{11}i}{4}$ are contained in $\D$. Then $\Omega:=\varphi(\D\IntComp)$ is an unbounded PQD in $\QD_2\left(\frac{3}{8}w+\frac{1}{4}\right)$.
To see this, first note that if $\varphi(z)=0$ then $|z|=\frac{\sqrt{3}}{2}<1$, so $0\notin\Omega$. Then by Lemma \ref{lemma:PolynomialUPQDClass}, $\Omega\in\QD_2(aw+b)$ and Equation \ref{eqn:PQDHFormulaOld} implies
$$aw+b=\AnalyticIn{\dfrac{\left(pp^{\#}\right)\circ\psi(w)}{2w}}{\Omega\IntComp}=\AnalyticIn{\dfrac{\left(\frac{3}{4}z^2+\frac{7}{8}z+O(1)\right)\circ\psi(w)}{2w}}{\Omega\IntComp}=\AnalyticIn{\dfrac{\frac{3}{4}\psi^2(w)+\frac{7}{8}\psi(w)}{2w}}{\Omega\IntComp}=\dfrac{3}{8}w+\dfrac{1}{4}.$$

\begin{figure}[ht]
  \centering
    \includegraphics[scale=.3]{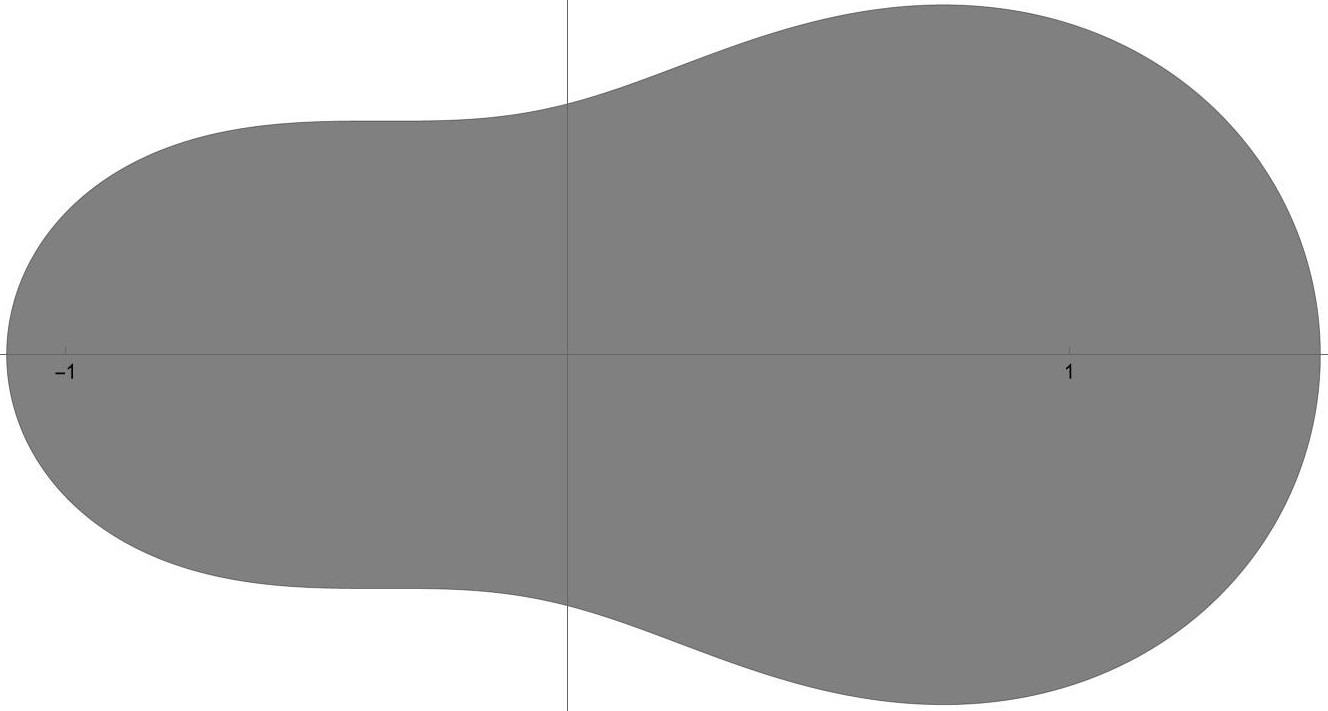}\vspace{.5em}
    \caption{$\Omega\in\QD_2\left(\frac{3}{8}w+\frac{1}{4}\right)$ (complement of shaded region).}\label{fig:PQDExample1}
\end{figure}
\end{example}

\subsection{Linear PQDs}\label{subsec:LinearPQDs}
By Lemma \ref{lemma:PolynomialUPQDClass}, if $\Omega\in\QD_a(\alpha_0+\alpha_1w)$ is simply connected and does not contain $0$, then there exists a polynomial $p$ of degree $2$ with $p(0)=c^a$ for which the associated Riemann map $\varphi:\D\IntComp\rightarrow\Omega$ is given by $\varphi(z)=zp^{\#}(z)^{\frac{1}{a}}$. That is, there exist constants $c_1,c_2\in\C$ for which $p(z)=c^a(1+c_1z+c_2z^2)$. By change of variables we can wlog assume $\alpha_1=1$. Applying Theorem \ref{thm:UPQDDirectProblemSol}, we have that
\begin{align*}
    \alpha_0+w&=\dfrac{1}{aw}\Phi_{\varphi}\left(\AnalyticIn{pp^{\#}}{\D}\right)(w)-\dfrac{t}{w}\\
    &=\dfrac{c^{2a}}{aw}\Phi_{\varphi}\left(u_0+u_1z+c_2z^2\right)(w)-\dfrac{t}{w}\\
    &=\dfrac{c^{2a}}{aw}\left(u_0+u_1F_1(w)+c_2F_2(w)\right)-\dfrac{t}{w},
\end{align*}
where $u_0=1+|c_1|^2+|c_2|^2$ and $u_1=c_1+\overline{c_1}c_2$. Using Equation \ref{eqn:FaberPolyFormulae} we compute $F_1(w)=\frac{w}{c}-\frac{\overline{c_1}}{a}$ and $F_2(w)=\frac{w^2}{c^2}-w\frac{2\overline{c_1}}{ac}+\frac{\overline{c_1}^2-2\overline{c_2}}{a}$. Substituting and equating polynomial coefficients, we obtain $1=\frac{c^{2(a-1)}}{a}c_2$, $\alpha_0=c_1\frac{c^{2a-1}}{a}+\frac{a-2}{a}c\overline{c_1}$. Solving for $c_1$ and $c_2$ yields
\begin{align*}
    c_1&=\frac{a}{c}\dfrac{(a-2)\overline{\alpha_0}-c^{2(a-1)}\alpha_0}{(a-2)^2-c^{4(a-1)}},\\
    c_2&=\dfrac{a}{c^{2(a-1)}}.
\end{align*}
That is,
$$\varphi(z)=cz\left(1+\frac{a}{c}\dfrac{(a-2)\alpha_0-c^{2(a-1)}\overline{\alpha_0}}{(a-2)^2-c^{4(a-1)}}z^{-1}+\dfrac{a}{c^{2(a-1)}}z^{-2}\right)^{\frac{1}{a}}.$$
Note that if we set $a=2$ and $\alpha_0=2\sqrt{2}$, then this simplifies to
$$\varphi(z)=cz\left(1+\dfrac{4\sqrt{2}}{c^3}z^{-1}+\dfrac{2}{c^2}z^{-2}\right)^{\frac{1}{2}}.$$
Furthermore, $0\notin\Omega=\varphi(\D\IntComp)$ for all $c>\sqrt{2}$ and in the limit as $c\rightarrow\sqrt{2}$, the expression simplifies further to $\varphi(z)=\sqrt{2}(z+1)$. In particular, this suggests that $\D_{\sqrt{2}}(\sqrt{2})\IntComp\in\QD_2\left(2\sqrt{2}+w\right)$. However, $0\in\D_{\sqrt{2}}(\sqrt{2})$, so this isn't fully rigorous. It also breaks the observed pattern of $\partial\Omega$ forming angles of $\frac{\pi}{a}$ radians. To confirm that $\D_{\sqrt{2}}(\sqrt{2})\IntComp\in\QD_2\left(2\sqrt{2}+w\right)$, we will verify the quadrature identity directly.  Take $f\in\A_0(\D_{\sqrt{2}}(\sqrt{2}))$, so that
\begin{align*}
    \int_{\D_{\sqrt{2}}(\sqrt{2})\IntComp}f(w)|w|^{2}dA(w)&=\dfrac{1}{2}\oint_{\partial\D_{\sqrt{2}}(\sqrt{2})\IntComp}f(w)w\overline{w}^2dw\\
    &=\dfrac{1}{2}\oint_{\partial\D_{\sqrt{2}}(\sqrt{2})\IntComp}f(w)w\left(\dfrac{2}{w-\sqrt{2}}+\sqrt{2}\right)^2dw\\
    &=\oint_{\partial\D_{\sqrt{2}}(\sqrt{2})\IntComp}f(w)\left(\dfrac{2\sqrt{2}}{(w-\sqrt{2})^2}+\dfrac{6}{w-\sqrt{2}}+2\sqrt{2}+w\right)dw\\
    &=\oint_{\partial\D_{\sqrt{2}}(\sqrt{2})\IntComp}f(w)\left(2\sqrt{2}+w\right)dw,
\end{align*}
confirming the identity.

\subsection{Quadratic PQDs}\label{subsec:QuadraticPQDs}

By Lemma \ref{lemma:PolynomialUPQDClass}, if $\Omega\in\QD_a(\alpha_0+\alpha_1w+\alpha_2w^2)$ is simply connected and does not contain $0$, then there exists a polynomial $p$ of degree $3$ with $p(0)=c^a$ for which the associated Riemann map $\varphi:\D\IntComp\rightarrow\Omega$ is given by $\varphi(z)=zp^{\#}(z)^{\frac{1}{a}}$. That is, there exist constants $c_1,c_2,c_3\in\C$ for which $p(z)=c^a(1+c_1z+c_2z^2+c_3z^3)$. By change of variables we can wlog assume $\alpha_2=1$. Applying Theorem \ref{thm:UPQDDirectProblemSol}, we have that
\begin{align*}
    \alpha_0+\alpha_1w+w^2&=\dfrac{1}{aw}\Phi_{\varphi}\left(\AnalyticIn{pp^{\#}}{\D}\right)(w)-\dfrac{t}{w}\\
    &=\dfrac{c^{2a}}{aw}\Phi_{\varphi}\left(u_0+u_1z+u_2z^2+c_3z^3\right)(w)-\dfrac{t}{w}\\
    &=\dfrac{c^{2a}}{aw}\left(u_0+u_1F_1(w)+u_2F_2(w)+c_3F_3(w)\right)-\dfrac{t}{w},
\end{align*}
Where $u_0=1+|c_1|^2+|c_2|^2+|c_3|^2$, $u_1=c_1+\overline{c_1}c_2+\overline{c_2}c_3$, $u_2=c_2+\overline{c_1}c_3$. Computing $F_1$, $F_2$, and $F_3$ explicitly (using e.g. Equation \ref{eqn:FaberPolyFormulae}) and equating coefficients, we find
\begin{align*}
    \alpha_0&=\dfrac{c^{2a-1}}{2a^3}\left(2a^2c_1+2a(a-2)\overline{c_1}c_2+(a-3)(2a\overline{c_2}-\overline{c_1}^2)c_3\right)\\
    \alpha_1&=\dfrac{c^{2a-2}}{a^2}\left(ac_2+(a-3)\overline{c_1}c_3\right)\\
    1&=\dfrac{c^{2a-3}}{a}c_3
\end{align*}
We can then write $c_3$ in terms of $a$ and $c$:
$$c_3=ac^{3-2a}$$
and we can write $c_2$ in terms of $a$, $c$, and $c_1$:
$$c_2=a \alpha_1 c^{2(1-a)}-(a-3)c^{3-2a}\overline{c_1}.$$
Substituting these into the $\alpha_0$ equation, the problem reduces to that of solving an algebraic equation in $c_1$ (linear in $c_1$ and quadratic in $\overline{c_1}$) with coefficients depending only on $a$, $c$, $\alpha_0$, and $\alpha_1$:
$$0=\left((a-3)(2a-3)c^{2a+3}\right)\overline{c_1}^2-\left(2(a-2)ac^{2a+2}\alpha_1\right)\overline{c_1}+2a\left(((a-3)^2c^6-c^{4a})c_1+ac^{2a+1}\alpha_0-(a-3)ac^5\overline{\alpha_1}\right).$$

For the remainder of the section, we restrict our focus to the special case of $a=2$.

\subsubsection{\texorpdfstring{$\QD_2(\alpha_0+\alpha_1w+w^2)$}{QD2(α0+α1w+w2)}}
When $a=2$, the logarithmic potential associated to a quadratic PQD $\Omega$ (Lemma \ref{lemma:GenQuadComplementDroplet}) is of the form
\begin{align*}
    Q(w)&=\dfrac{|w|^4}{4}-2\Re\left(\alpha_0w+\dfrac{\alpha_1}{2}w^2+\dfrac{w^3}{3}\right),
\end{align*}
plus some additive constant on each connected component of the droplet $K=\Omega^c$. Note that $Q$ is globally admissible because $\Re(w^3)\lesssim|w|^4$ $\implies$ $Q(w)=O(|w|^4)$. Thus by Frostman's theorem (\ref{lemma:Frostman}), for each $t>0$ there exists a unique droplet $K_t$ associated to $Q$, the complement of which is a quadratic PQD. Combining this with Lemma \ref{lemma:HSChains} we find that, for each choice of constants $\alpha_0$ and $\alpha_1$, there is an associated $1-$parameter family of quadratic PQDs with $a=2$.

When $a=2$, the above relations reduce to $c_2=2\alpha_1 c^{-2}+c^{-1} \overline{c_1}$, and $c_3=2c^{-1}$, so that
\begin{align*}
    \varphi(z)&=cz\left(1+\dfrac{\overline{c_1}}{z}+\dfrac{2\overline{\alpha_1}+c_1c}{c^2z^2}+\dfrac{2}{cz^3}\right)^{\frac{1}{2}}
\end{align*}
Moreover, taking the aforementioned algebraic equation in $c_1$, along with its conjugate equation, we may eliminate $\overline{c_1}$ to obtain a quartic in $c_1$,
\begin{align*}
    0&=64(\alpha_1+\overline{\alpha_0})^2 - 128 (c^2-1)^2 (\overline{\alpha_1}+\alpha_0)+64 c (c^2-1)^3c_1 - 16c^2 (\alpha_1+\overline{\alpha_0})c_1^2 + c^4c_1^4.
\end{align*}
By numerically solving the quartic for $c_1$ for various values of $c$, $\alpha_0$ and $\alpha_1$, we obtain six families of quadratic PQDs with $a=2$ which appear to exhibit distinct critical and topological behavior (Figures \ref{fig:QuadraticPQDs12} and \ref{fig:QuadraticPQDs3}).

\begin{figure}[ht]
  \centering
    \includegraphics[height=.33\linewidth,valign=c]{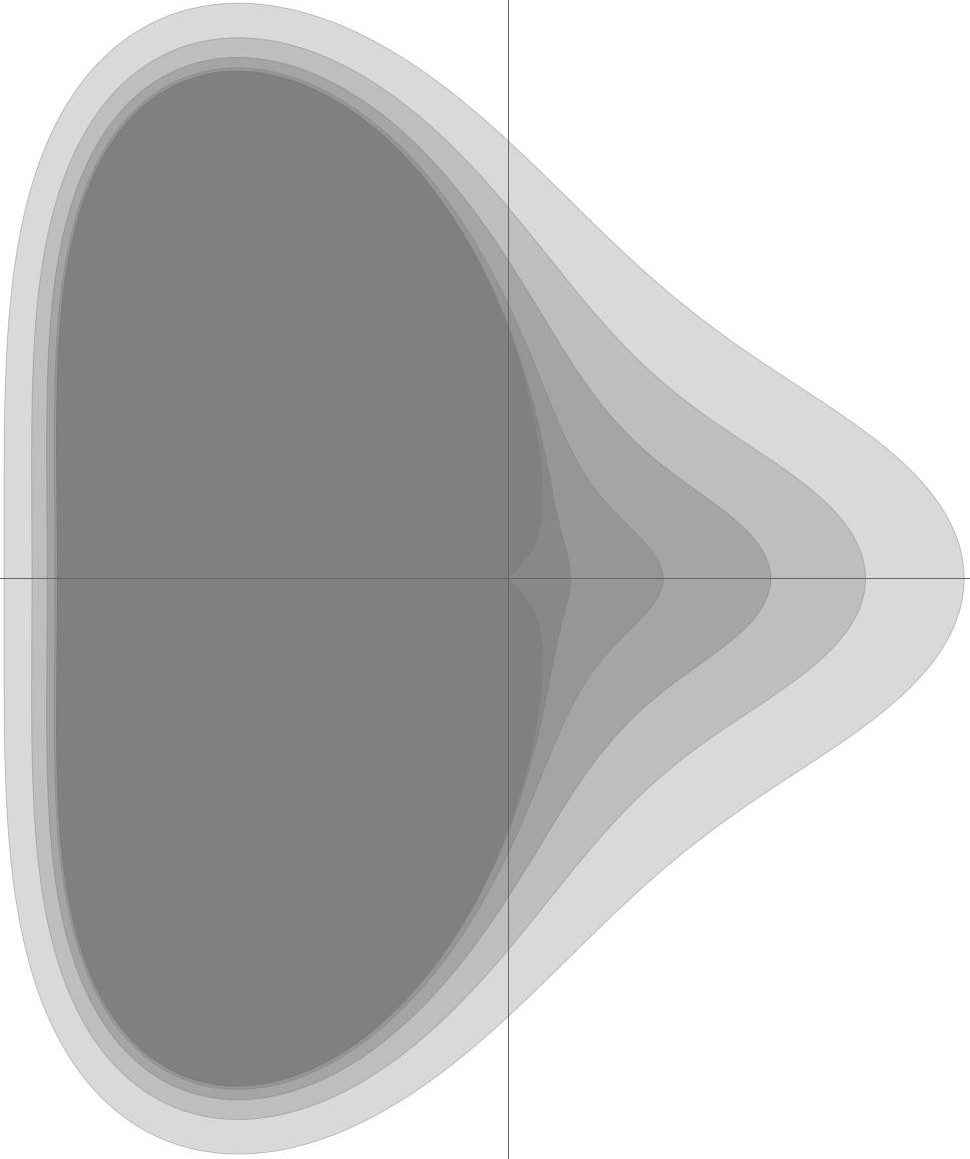}\;\;\;\;\;\;\;\;\;\;\;\;
    \includegraphics[height=.33\linewidth,valign=c]{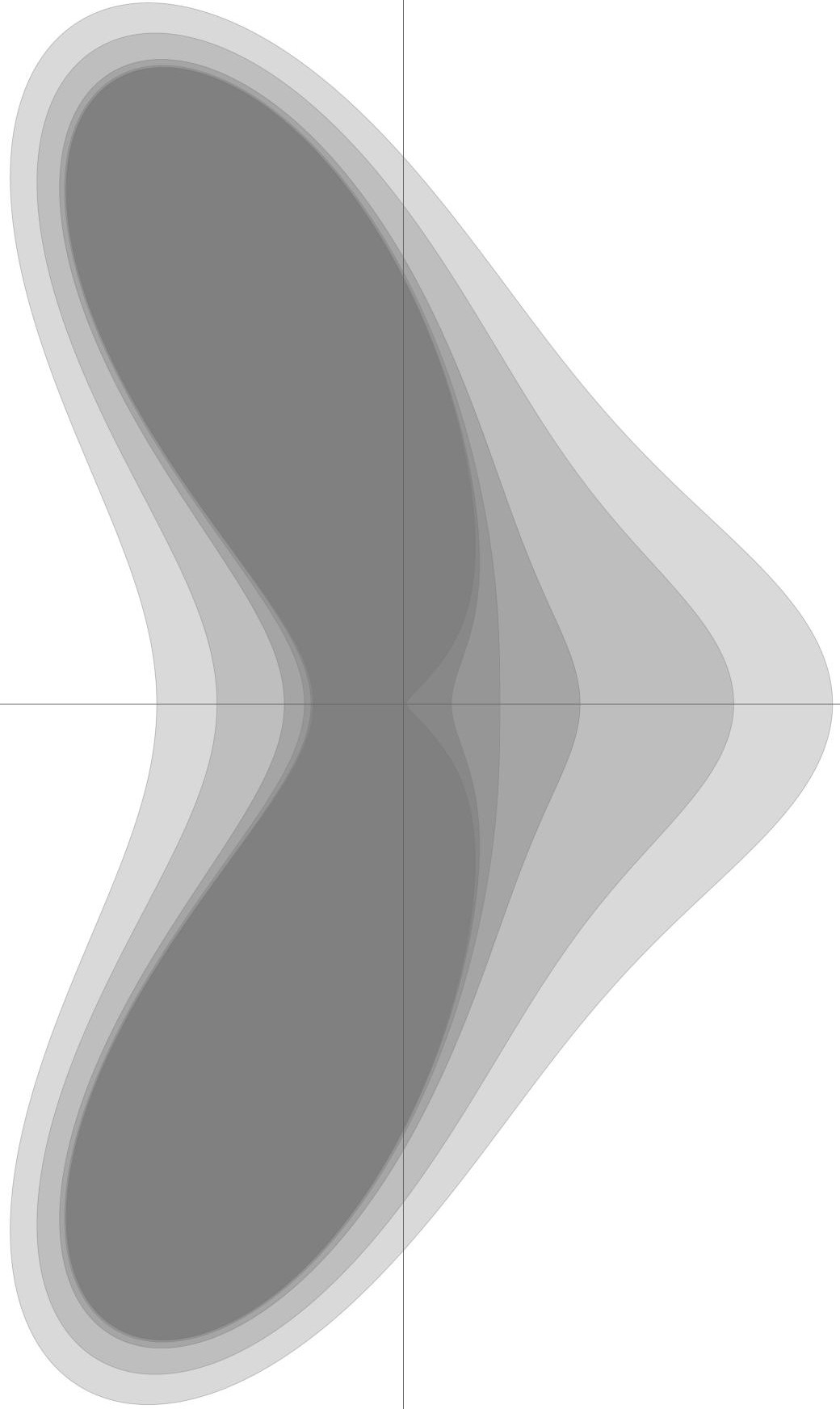}\;\;\;\;\;\;\;\;\;\;\;\;
    \includegraphics[height=.33\linewidth,valign=c]{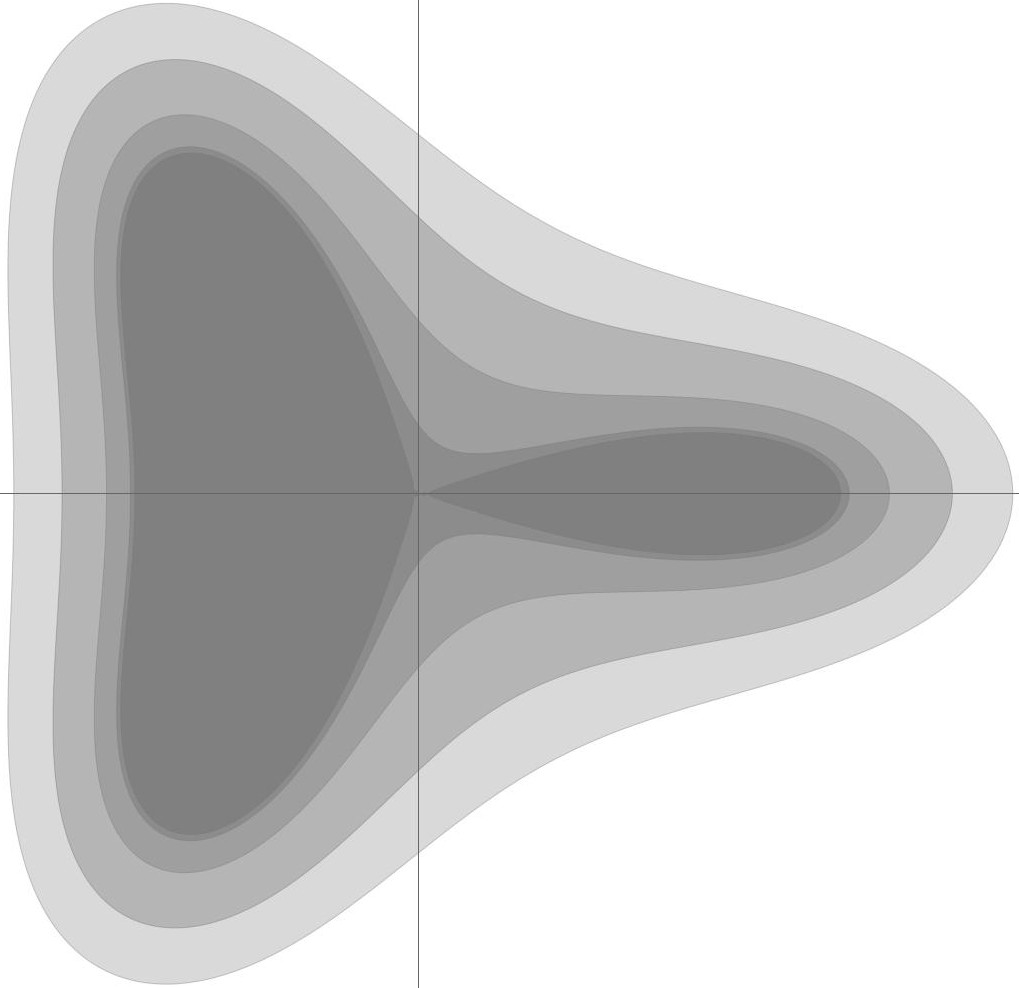}
    \caption{Families of PQDs (complements of the shaded regions) in $\QD_2(\alpha_0+\alpha_1w+w^2)$ for $\alpha_0=-30$, $\alpha_1=0$, and $c_{\min}\approx3.875$ (left); $\alpha_0=-4$, $\alpha_1=-3.04$, and $c_{\min}\approx2.887$ (center); $\alpha_0=-3.52$, $\alpha_1=1.43$, and $c_{\min}\approx2.741$ (right).}\label{fig:QuadraticPQDs12}\vspace{1.5em}
\end{figure}

\begin{figure}[ht]
  \centering
    \includegraphics[width=.31\linewidth,valign=c]{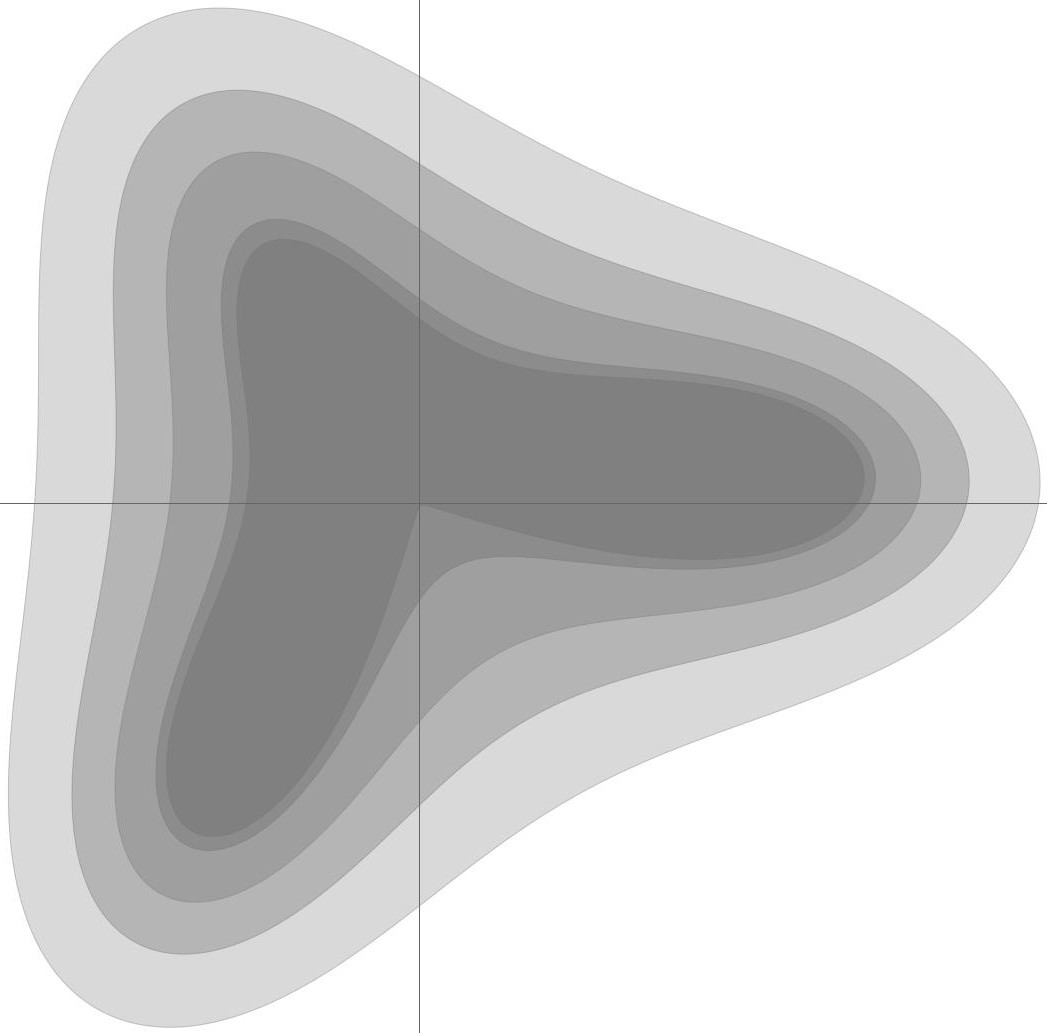}\;\;
    \includegraphics[width=.31\linewidth,valign=c]{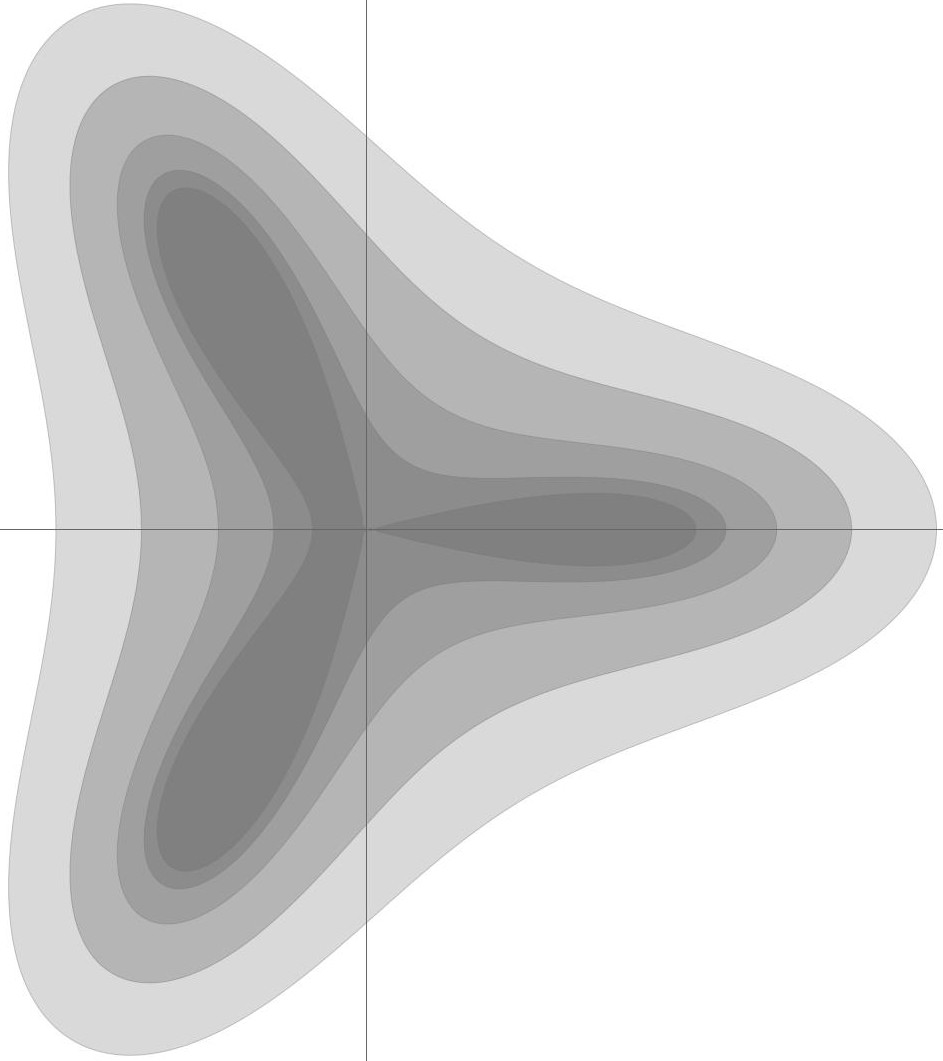}\;\;
    \includegraphics[width=.31\linewidth,valign=c]{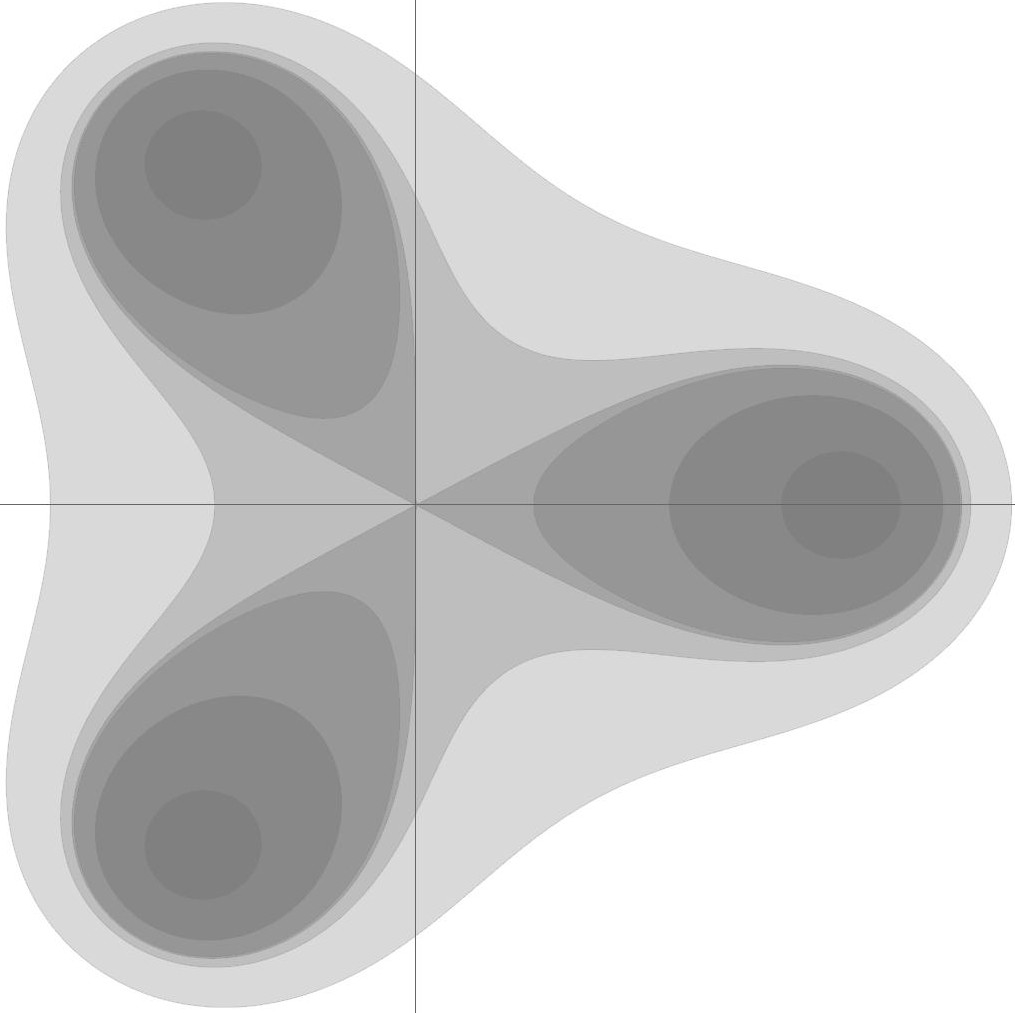}
    \caption{Families of PQDs (complements of the shaded regions) in $\QD_2(\alpha_0+\alpha_1w+w^2)$ for $\alpha_0=.86$, $\alpha_1=.96-i$, and $c_{\min}\approx2.6649$ (left); $\alpha_0=-.5$, $\alpha_1=-.17$, and $c_{\min}\approx2.026$ (center); $\alpha_0=0$, $\alpha_1=0$, and $c_{\min}=2$ (right).}\label{fig:QuadraticPQDs3}\vspace{1.5em}
\end{figure}

\section{Log-Weighted Quadrature Domains}\label{sec:LQDs}
We now shift our discussion to {\it log-weighted quadrature domains} (LQDs). These are WQDs with respect to the ``logarithmic'' weight $\rho_0(w):=|w|^{-2}=\frac{1}{4}\Delta\frac{\ln^2|w|^2}{2}$, and as we will see in Lemma \ref{lemma:LQDasPQDLimit}, they can be understood in certain cases as a limit of PQDs ($\rho_a(w)=|w|^{2(a-1)}$) as $a\to 0^{+}$.

Unlike $\rho_{a}$ for $a>0$, $\rho_0$ has a non-integrable singularity at the origin, so $\A(\Omega)$ and $\A_0(\Omega)$ aren't $\rho_0-$integrable function spaces when $0\in\Omega$. For the purposes of this article, we restrict our consideration to the non-singular ($0\notin\text{Cl}(\Omega)$) case in which $\A(\Omega)$ and $\A_0(\Omega)$ are integrable. {\it We shall implicitly assume that $0\notin\text{Cl}(\Omega)$ for the remainder of this article.} A manuscript covering the singular case is under preparation. 

The central result of this section is Theorem \ref{theorem:logweightedAQDRiemannMaps}, which generalizes the Faber transform formulae of Theorems \ref{thm:ClassicBQDRationalRiemannIffQD} and \ref{thm:ClassicUQDRationalRiemannIffQD} to LQDs. We also obtain classification results for null LQDs (Theorem \ref{thm:NullLQDClass}), one point LQDs (Theorem \ref{thm:LOGWQDBoundedOnePt}), and monomial LQDs (\S \ref{subsec:LOGWQDBasicMonomial}).

\begin{definition}[Log-weighted quadrature domain]
Let $\Omega\subset\Ch$ be a bounded (resp. unbounded) domain, equal to the interior of its closure, for which $0\notin\Cl(\Omega)$. We say that $\Omega$ is a log-weighted quadrature domain iff there exists an $h\in\Rat_0(\Omega)$ (resp. $\Rat(\Omega)$) such that
\begin{equation}
\int_{\Omega}\dfrac{f(w)}{|w|^{2}}dA(w)=\oint_{\partial\Omega}f(w)h(w)dw,
\end{equation}
for each $f\in \A(\Omega)$ (resp. $\A_0(\Omega)$). This is denoted by $\Omega\in\QD_0(h)$.
\end{definition}

Equivalently, we say that $\Omega\in\QD_0(h)$ iff $\Omega\in\QD_{\rho_0}(h;\A(\Omega))$ when $\Omega$ is bounded or $\Omega\in\QD_{\rho_0}(h;\A_0(\Omega))$ when $\Omega$ is unbounded. We also denote by $\Omega\in\QD_0$ that there exists an $h$ for which $\Omega\in\QD_0(h)$.

% Lemmas \ref{lemma:LOGWQDCE} and \ref{lemma:LOGWQDCEInv} -> Theorem \ref{theorem:EquivLQDChars}

\subsection{Basic Properties of LQDs}\label{subsec:LQDProps}
We begin by deriving an analog of the coincidence equation (\ref{eqn:QuadCoincidence}) for LQDs. Fix a bounded domain $\Omega\in\QD_0(h)$. Then, applying Lemma \ref{lemma:WQDGCE} with $F(w)=\frac{1}{2}\ln^2|w|^2$, we find that there exists $G\in\A(\Omega)$ for which
\begin{equation}\label{eqn:LogWeightedSFE}
        h(w)\dEquals\dfrac{\ln|w|^2}{w}+G(w).
\end{equation}
The same argument and conclusion holds for unbounded LQDs by simply exchanging $\A(\Omega)$ and $\A_0(\Omega)$.

Equation \ref{eqn:LogWeightedSFE} provides us with a straightforward method of proof that $h$ is unique. Suppose $\Omega$ is an LQD with quadrature functions $h_1$ and $h_2$ then
$$h_1(w)-h_2(w)\dEquals\left(\dfrac{\ln|w|^2}{w}+G_1(w)\right)-\left(\dfrac{\ln|w|^2}{w}+G_2(w)\right)=G_1(w)-G_2(w).$$
Therefore $h_1-h_2$ continues analytically into $\Omega$. So when $\Omega$ is bounded, $h_1-h_2\in\Rat_0(\Omega)\cap\A(\Omega)=\{0\}$, and when $\Omega$ is unbounded, $h_1-h_2\in\Rat(\Omega)\cap\A_0(\Omega)=\{0\}$, so $h_1=h_2$. Hence, under this definition, $h$ is uniquely associated to $\Omega$.

Equation \ref{eqn:LogWeightedSFE} moreover captures one of several equivalent characterizations of LQDs in the spirit of Theorem \ref{theorem:EquivPQDChars}:
\begin{theorem}\label{theorem:EquivLQDChars}
If $\Omega\subset\Ch$ is a domain then the following are equivalent
\begin{enumerate}
    \item There exists a rational $h$ for which $\Omega\in\QD_0(h)$.
    \item There exists a $G\in\A(\Omega)$ ($\A_0(\Omega)$ when $\Omega$ is unbounded) and a rational $h$ for which Equation \ref{eqn:LogWeightedSFE} holds.
    \item There exists a function $S_0\in\M(\Omega)$ for which $S_0(w)\dEquals\frac{\ln|w|^2}{w}$. This is the {\it generalized Schwarz function} associated to $\Omega$.
\end{enumerate}
Moreover in this case $h=\AnalyticIn{S_0}{\Omega\IntComp}$.
\end{theorem}
The generalized Schwarz function $S_0\in\M(\Omega)$ is unique when it exists as a consequence of its equality with $\frac{\ln|w|^2}{w}$ on $\partial\Omega$.
\begin{proof}[Proof of Theorem \ref{theorem:EquivLQDChars}]\;\\
$\mathbf{(1)\implies(2)}$: This follows from the application of Lemma \ref{lemma:WQDGCE} above.\\
$\mathbf{(2)\implies(3)}$: Rearranging, we find that $\frac{\ln|w|^2}{w}\dEquals h(w)-G(w)=:S_0(w)$, which is the generalized Schwarz function for $\Omega$ because $h-G\in\M(\Omega)$. Also, $\AnalyticInNoBracket{S_0}{\Omega\IntComp}=\AnalyticIn{h-G}{\Omega\IntComp}=h$.\\
$\mathbf{(3)\implies(1)}$: Finally, if $\Omega$ is a bounded domain for which such an $S_0$ exists, and $f\in\A(\Omega)$, then
\begin{align*}
\int_{\Omega}\dfrac{f(w)}{|w|^{2}}dA(w)&=\oint_{\partial\Omega}f(w)\frac{\log|w|^2}{w}dw=\oint_{\partial\Omega}f(w)S_0(w)dw\\
&=\oint_{\partial\Omega}f(w)\AnalyticIn{S_0(w)}{\Omega\IntComp}dw=\oint_{\partial\Omega}f(w)h(w)dw,
\end{align*}
($h:=\AnalyticIn{S_0}{\Omega\IntComp}$ is rational because $S_0\in\M(\Omega)$). So $\Omega\in\QD_0(h)$. The argument in the unbounded case is completely analogous.
\end{proof}

\subsubsection{Boundary Regularity of LQDs}
Fix $\Omega\in\QD_0(h)$ so, by Theorem \ref{theorem:EquivLQDChars}, Equation \ref{eqn:LogWeightedSFE} is satisfied on $\partial\Omega$. Rearranging, we obtain 
\begin{equation}\label{eqn:LQDTranscendentalSFcn}
    \overline{w}\dEquals g(w)e^{wh(w)}.
\end{equation}
where $g(w)=w^{-1}e^{-wG(w)}\in\A(\Omega)$ ($\A_0$ when $\Omega$ is unbounded). Hence $S(w):=g(w)e^{wh(w)}$ is (almost) a Schwarz function for $\Omega$, the only obstruction to this being the fact that $S$ has essential singularities at the poles of $h$ so it is not meromorphic in $\Omega$. However, this is more than enough to establish the boundary regularity of LQDs.

In particular, Sakai's regularity theorem \ref{theorem:SakaiRegularity} requires only the existence of a {\it local Schwarz function}. Thus, by applying the same reasoning as in the classical case (\S\ref{subsubsec:QDBoundaryRegularity}), we obtain a regularity theorem for LQDs along the same lines as that for QDs,
\begin{theorem}\label{theorem:LQDBoundaryRegularity}
    If $\Omega\in\QD_0$ then $\partial\Omega$ has finitely many singular points, each of which is either a cusp or a double point.
\end{theorem}

\subsubsection{Invariance Under Inversion}

LQDs are particularly well-behaved under inversion, which is essentially a consequence of the fact that $(w^{-1})^{\ast}\frac{dA(w)}{|w|^2}=\frac{|-w^{-2}|^2dA(w)}{|w^{-1}|^2}=\frac{dA(w)}{|w|^2}$, where $(w^{-1})^{\ast}$ is the pullback by $w\mapsto w^{-1}$. In particular,
\begin{lemma}\label{lemma:LQDInversion}
If $\Omega$ is a bounded domain, then $\Omega\in\QD_0(h)$ iff $\Omega^{-1}\in\QD_0(-h(w^{-1})w^{-2})$.
\end{lemma}
\begin{proof}[Proof of Lemma \ref{lemma:LQDInversion}]
Let $\Omega$ be a bounded domain. If $\Omega\in\QD_0(h)$ then for each $f\in \A(\Omega^{-1})$, $f(w^{-1})\in \A(\Omega)$, so
\begin{align*}
\int_{\Omega^{-1}}\dfrac{f(w)}{|w|^2}dA(w)&=\int_{\Omega}\dfrac{f(w^{-1})}{|w|^2}dA(w)=\oint_{\partial\Omega}f(w^{-1})h(w)dw=\oint_{\partial\Omega^{-1}}f(w)h(w^{-1})(-w^{-2})dw
\end{align*}
This implies the backwards direction as well because $\widetilde{h}(w)=-h(w^{-1})w^{-2}$ iff $-\widetilde{h}(w^{-1})w^{-2}=h(w)$.
\end{proof}

\subsubsection{Null LQDs}

We begin with a classification of null LQDs (domains $\Omega\in\QD_0(0)$), summarized by the following theorem.
\begin{theorem}\label{thm:NullLQDClass}
A domain $\Omega$ is a null LQD iff it is an exterior disk centered at the origin.
\end{theorem}
\begin{proof}[Proof of Theorem \ref{thm:NullLQDClass}]\label{proof:NullLQDClass}
Let $\Omega\in\QD_0(0)$. We first note that $\Omega$ must be unbounded because otherwise we may take $1$ as a test function which implies
$$0<\int_{\Omega}|w|^{-2}dA(w)=\oint_{\partial\Omega}1\cdot0dw=0,$$
a contradiction.

Then $\Omega$ must be unbounded, so Equation \ref{eqn:LogWeightedSFE} yields $\ln|w|^2\dEquals -wG(w)$ for some $G\in\A_0(\Omega)$. If $q=\Res{\infty}G$, then $\widetilde{G}(w):=-wG(w)-q\in\A_0(\Omega)$, and
$$\ln\left|w\right|^2\dEquals -wG(w)\;\;\;\;\implies\;\;\;\;\left|w\right|^2e^{-q}\dEquals e^{\widetilde{G}(w)},$$
where $e^{\widetilde{G}(w)}-1\in\A_0(\Omega)$. Thus there exists $U\in\A_0(\Omega)$ and $A\in\C$ for which $|w|^2\dEquals U(w)+A$. Thus if $f\in\A_0(\Omega)$,
$$\int_{\Omega}f(w)dA(w)=\oint_{\partial\Omega}f(w)\overline{w}dw=\oint_{\partial\Omega}f(w)\dfrac{U(w)+A}{w}dw.$$
As $0\notin\Omega$ this equals $0$. In particular, $\Omega$ is a (classical) null quadrature domain centered at $0$. Thus $\Omega$ is an exterior disk centered at the origin (see e.g. \cite{AharonovShapiro}). Conversely, it is straightforward to show that $(\D_{r})\IntComp\in\QD_0(0)$ for each $r>0$.
\end{proof}

\subsubsection{Potential Theory for LQDs}

As in the discussion of classical one point QDs (\ref{sec:OnePtQDClass}), it is instructive to consider the situation from the perspective of potential theory. In particular, we have the following lemmas relating the two.
\begin{lemma}\label{lemma:GenLogQuadComplementDroplet}
If $K\subset\C$ is a compact subset such that $K^c$ is a finite disjoint union of LQDs, $\Omega_j\in\QD_0(h_j)$ and a disk centered at the origin $\Omega_0=\D_r$ for some $r>0$, then $K$ is a local droplet. Furthermore, the external potential associated to the droplet is given by
$$Q(w)=\dfrac{1}{2}\ln^2|w|^2-2\Re(H(w)),\;H(w)=2\Re\left(\int_{w_l}^{w}h(\xi)d\xi\right)+\text{const}(w_l)\text{ in some open nbhd of }K_l$$
(where $h=\sum h_j+\frac{\ln|r|^2}{w}$ and $w_l$ is any point in the $l$th connected component of $K$, $K_l$). In particular, $Q$ is unique up to several additive constants.
\end{lemma}
\begin{proof}[Proof of Lemma \ref{lemma:GenLogQuadComplementDroplet}]\label{proof:GenLogQuadComplementDroplet}
Let $K^c=\Omega_\infty\cup\bigcup_{j=0}^{N}\Omega_j$ be a disjoint union where $\Omega_j\in\QD_0(h_j)$ for $j>1$ and $\Omega_0=\D_r$ for some $r>0$. Also let $h=\sum_{j}h_j+\frac{\ln|r|^2}{w}$. Applying Frostman's theorem (lemma \ref{lemma:FrostmanPt2}), it is sufficient to construct an open nbhd $\mathcal{O}\supset K$ and harmonic $H:\mathcal{O}\rightarrow\R$ such that $\frac{\partial H}{\partial w}=h$ and
\begin{equation}\label{eqn:GenLogQuadCompDropIdPf}
    \dfrac{1}{2}\ln^2|w|^2-H(w)+U^\mu(w)=0,\;\forall w\in K,
\end{equation}
where $d\mu=\1_K\frac{\Delta Q}{2}dA$. We shall choose $\mathcal{O}$ as an $\epsilon-$nbhd of $K$ for a to be determined value of $\epsilon$. We require that $\mathcal{O}$ satisfies each of the following
\begin{enumerate}
    \item $h$ is holomorphic in $\mathcal{O}$,
    \item each connected component of $\mathcal{O}$ contains precisely one connected component of $K$, 
    \item every loop in $\mathcal{O}$ is homotopic to a loop in $K$.
\end{enumerate}
By assumption, $\partial K$ has finitely many singular points. Thus there exists $\epsilon>0$ s.t. (2) and (3) hold.
Note that, for each $\Omega_j$ and each $w\in (\Omega_j)\IntComp$, the quadrature identity tells us that for each $j>0$
$$C_Q^{\Omega_j}(w)=\int_{\Omega_j}\dfrac{|\xi|^{-2}}{w-\xi}dA(\xi)=\oint_{\partial\Omega_j}\dfrac{h_j(\xi)}{w-\xi}d\xi=\oint_{\partial(\Omega_j)\IntComp}\dfrac{h_j(\xi)}{\xi-w}d\xi=h_j(w),$$
(see Equation \ref{sec:PotentialTheoryHeleShaw} for the initial discussion of $C_Q$.)
and this extends to $\partial\Omega_j$. Thus, for each $w\in\bigcap_{j}\Omega_j^{c}=K$, 
$$C_Q^{\Omega\setminus\D_r}(w)=\int_{\Omega}\dfrac{|\xi|^{-2}}{w-\xi}dA(\xi)=\sum_{j>1}\int_{\Omega_j}\dfrac{|\xi|^{-2}}{w-\xi}dA(\xi)=\sum_{j>1}C_Q^{\Omega_j}(w)=\sum_jh_j(w)=h(w)-\dfrac{\ln|r|^2}{w}.$$
But then note that
\begin{align*}
C_Q^{\D_r^c}(w)=\int_{\D_r^c}\dfrac{|\xi|^{-2}}{w-\xi}dA(\xi)&=\lim_{R\to\infty}\int_{\D_R\setminus\D_r}\dfrac{|\xi|^{-2}}{w-\xi}dA(\xi)=\dfrac{\ln|w|^2}{w}-\lim_{R\to\infty}\oint_{\partial(\D_R\setminus\D_r)}\dfrac{\ln|\xi|^2}{\xi(\xi-w)}d\xi\\
&=\dfrac{\ln|w|^2}{w}-\lim_{R\to\infty}\left(\ln|R|^2\oint_{\partial\D_R}\dfrac{1}{\xi(\xi-w)}d\xi-\ln|r|^2\oint_{\partial\D_r}\dfrac{1}{\xi(\xi-w)}d\xi\right)\\
&=\dfrac{\ln|w|^2}{w}-\lim_{R\to\infty}\left(0-\ln|r|^2\left(-\dfrac{\1_{\D_r^c}(w)}{w}\right)\right)\\
&=\dfrac{\ln|w|^2}{w}-\dfrac{\ln|r|^2}{w}\1_{\D_r^c}(w)=\dfrac{\ln|w|^2}{w}-\dfrac{\ln|r|^2}{w}
\end{align*}
So
\begin{align*}
    C_Q^K(w)+C_Q^{\Omega\setminus\D_r}(w)&=C_Q^{\D_r^c}(w)=\dfrac{\ln|w|^2}{w}-\dfrac{\ln|r|^2}{w}
\end{align*}
on $K$. Thus $C_{Q}^K(w)+h(w)=\frac{\ln|w|^2}{w}$ on $K$. Then note that if $\gamma$ is a loop in $K$,
\begin{align*}
&2\Re\left(\oint_{\gamma}h(w)dw\right)=\oint_{\gamma}h(w)dw+\overline{\oint_{\gamma}h(w)dw}\\
&=\oint_{\gamma}\left(\dfrac{\ln|w|^2}{w}-C_{Q}^{K}(w)\right)dw+\oint_{\gamma}\left(\overline{\dfrac{\ln|w|^2}{w}-C_{Q}^{K}(w)}\right)d\overline{w}\\
&=\oint_{\gamma}\partial\left(\dfrac{1}{2}\ln^2|w|^2+U^{K}(w)\right)+\oint_{\gamma}\overline{\partial}\left(\overline{\dfrac{1}{2}\ln^2|w|^2+U^{\mu}(w)}\right)\\
&=\oint_{\gamma}d\left(\dfrac{1}{2}\ln^2|w|^2+U^{\mu}(w)\right)=0.
\end{align*}
So, by Morera's theorem, $h$ is holomorphic in $K$ and thus in $\mathcal{O}$ because $h$ is a sum of quadrature functions on the components of $K^c$, so, by (3), this result extends to $\mathcal{O}$.

As a consequence, if we pick a point $w_l$ in each connected component $\mathcal{O}_l$ of $\mathcal{O}$ and set
$$H(w):=2\Re\left(\int_{w_l}^{w}h(\xi)d\xi\right)+\dfrac{1}{2}\ln^2|w_l|^2+U^\mu(w_l),\;\forall w\in\mathcal{O}_l,$$
is a real, well-defined harmonic function in $\mathcal{O}$ and $\frac{\partial H}{\partial w}=h$ on $K$. Finally, we need to verify Equation \ref{eqn:GenLogQuadCompDropIdPf}. By construction, the identity holds at each $w_l$. Furthermore, we have
$$\dfrac{\partial}{\partial w}\left(\dfrac{1}{2}\ln^2|w|^2-H+U^\mu\right)=\dfrac{\ln|w|^2}{w}-h-C^K=0\;\text{ on }K,$$
and
$$\dfrac{\partial}{\partial \overline{w}}\left(\dfrac{1}{2}\ln^2|w|^2-H+U^\mu\right)=0\;\text{ on }K.$$
Thus Equation \ref{eqn:GenQuadCompDropIdPf} holds.

An additional consequence of this (from Frostman's theorem \ref{lemma:FrostmanPt2} and Equation \ref{eqn:GenLogQuadCompDropIdPf}) is that $K$ is a local droplet of the potential
$$Q(w)=\dfrac{1}{2}\ln^2|w|^2-H(w).$$
\end{proof}

\noindent Conversely,

\begin{lemma}\label{lemma:GenComplementDropletLogQuad}
If $K$ is a local droplet of a potential $Q(w)=\dfrac{1}{2}\ln^2|w|^2-2\Re(H(w))$ and $H$ with rational derivative, then $K^c$ is a disjoint union of LQDs and a disk $\D_r$ for some $r>0$ for which $h(w):=H'$ is the sum of their quadrature functions and $\frac{\ln|r|^2}{w}$. 
\end{lemma}
This follows straightforwardly from the coincidence equation for local droplets (Equation \ref{eqn:ceq}).

\subsection{LQDs as Limits of PQDs}

Theorem \ref{theorem:EquivLQDChars} gives an analogue of the Schwarz function equation (\ref{eqn:QuadCoincidence}) \`a la that for PQDs (\ref{eqn:PQDCoincidence}). There are other results for PQDs which extend nicely to this ``limiting'' situation. This understanding of LQDs as the limit of a sequence of PQDs as $a\to0^{+}$ is formalized in the following lemma
\begin{lemma}\label{lemma:LQDasPQDLimit}
Let $\{\Omega_a\}_{0<a<a_\ast}$ be a family of domains and $h\in\Rat_0\left(\text{int}\left(\cap_{a}\Omega_a\right)\right)$ such that $\Omega_a\in\QD_a(h)$ and the lengths of $\{\partial\Omega_a\}_a$ are uniformly bounded. If $\Omega_a\xrightarrow{a\to0^{+}}\Omega_0$ in the Hausdorff metric and $0\notin\Cl(\Omega_0)$, then $\Omega_0\in\QD_0(h)$.
\end{lemma}

\begin{example}\label{ex:LOGWQDUnboundedNoZeroMonomialkLimit}
Recall that if $0\notin\Omega_a$ is an unbounded domain for which $\Omega_a\in\QD_{a}(\alpha kw^{k-1})$ for some $a,\alpha>0$ and $k\in\Z_{+}$, then $\varphi_a(z)=cz\left(1-\frac{\gamma_k}{z^{k}}\right)^{\frac{1}{a}}$, where $\gamma_k=-\frac{a \alpha k}{c^{2a-k}}$ is a Riemann map for $\Omega_a$ (Theorem \ref{thm:GenMonomialPQDNoZero}). Then, leaving the details to the reader, the lemma implies that
$$\varphi_0(z):=\lim_{a\to0^{+}}\varphi_a(z)=\lim_{a\to0^{+}}cz\left(1+\frac{a}{c^{2a}}\left(\alpha kc^kz^{-k}\right)\right)^{\frac{1}{a}}=cze^{\alpha kc^kz^{-k}}$$
is the Riemann map associated to $\lim_{a\to0^{+}}\Omega_a=:\Omega_0\in\QD_0(\alpha kw^{k-1})$. This result is confirmed by its agreement with the outcome from direct computation (\S\ref{subsec:LOGWQDBasicMonomial}).
\begin{figure}[ht]
  \centering
    \includegraphics[height=0.2\linewidth,width=.38\linewidth]{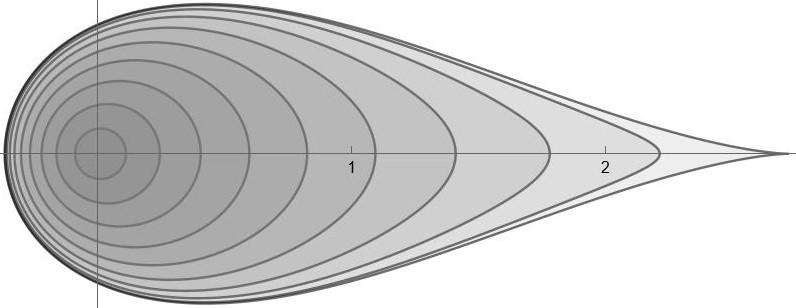}
    \includegraphics[height=0.2\linewidth,width=.38\linewidth]{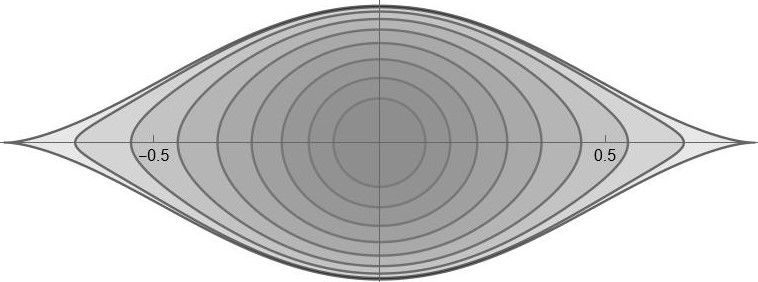}
    \includegraphics[height=0.21\linewidth]{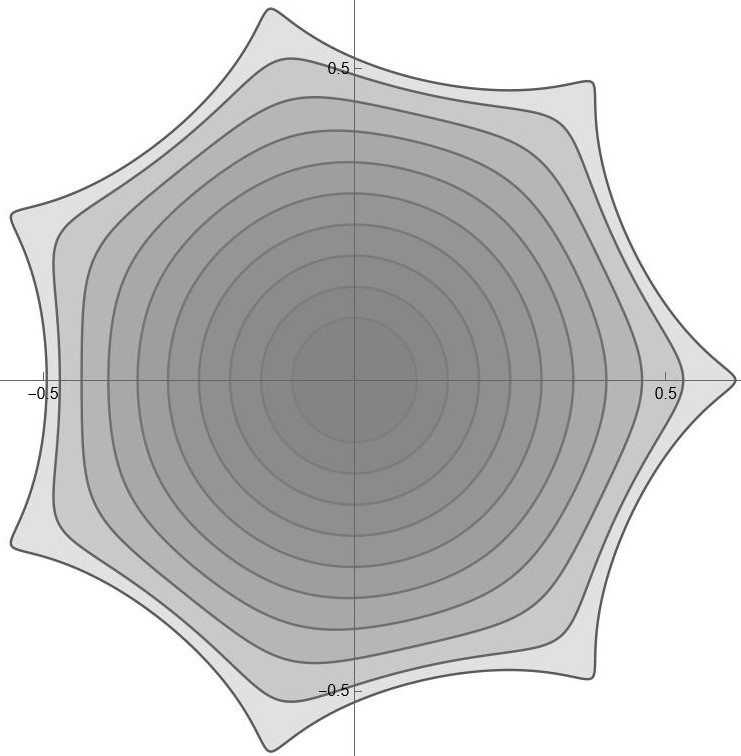}
    \caption{Family of domains associated to an unbounded simply connected $\Omega\in\QD_0\left(kw^{k-1}\right)$, $k\in\{1,2,7\}$. (Example \ref{ex:LOGWQDUnboundedNoZeroMonomialkLimit} and \S\ref{subsec:LOGWQDBasicMonomial})}\label{fig:LOGWQDUnboundedNoZeroMonomialk}\vspace{1.5em}
\end{figure}
\end{example}

\begin{proof}[Proof of Lemma \ref{lemma:LQDasPQDLimit}]
As $\Omega_a\xrightarrow{a\to0^{+}}\Omega_0$ in the Hausdorff metric, for each $\epsilon>0$ there exists $0<a_\epsilon<\min\{a_\ast,\epsilon\}$ for which $d_H(\Omega_a,\Omega_0)<\epsilon$ for all $0<a\leq a_\epsilon$. Take $\epsilon_0$ sufficiently small that $0$ is bounded away from the $\Omega_a$ by distance $\delta=\delta(\epsilon_0)>0$ for all $0<a<a_\epsilon$, $0<\epsilon\leq\epsilon_0$ then, by the uniform boundedness of the boundary lengths, we have that $\mu(\Omega_a\setminus\Omega_0),\mu(\Omega_0\setminus\Omega_a)<\frac{L\pi\epsilon^2}{\delta}$ for all $0<a<a_\epsilon$. If $U_{\epsilon_0}=\cup_{0<a\leq a_{\epsilon_0}}\Omega_a$ is unbounded and $f\in \A_0\left(U_{\epsilon_0}\right)$ is given by $f(w)=f_1w^{-1}+f_2w^{-2}+\hdots$ then
\begin{align*}
&\left|\int_{\Omega_0}\dfrac{f(w)}{|w|^2}dA(w)-\int_{\Omega_a}\dfrac{f(w)}{|w|^2}|w|^{2a}dA(w)\right|=\left|\int_{\Omega_0\cup\Omega_a}\dfrac{f(w)}{|w|^2}\left(\1_{\Omega_0}-|w|^{2a}\1_{\Omega_a}\right)dA(w)\right|\\
&=\left|\int_{\Omega_0\cup\Omega_a}\dfrac{\frac{f(w)}{w}}{\overline{w}}\left(\1_{\Omega_0\setminus\Omega_a}-\1_{\Omega_a\setminus\Omega_0}+(1-|w|^{2a})\1_{\Omega_a}\right)dA(w)\right|\\
&\leq \left|\int_{\Omega_0\setminus\Omega_a}\dfrac{f(w)}{|w|^2}dA(w)\right|+\left|\int_{\Omega_a\setminus\Omega_0}\dfrac{f(w)}{|w|^2}dA(w)\right|+\left|\int_{\Omega_a}f(w)\dfrac{|w|^{2a}-1}{|w|^2}dA(w)\right|\\
&\leq 2\dfrac{L\pi\epsilon^2}{\delta}\sup_{U_{\epsilon_0}}|f|+\left|\int_{\Omega_a}\dfrac{1}{|w|^2}\sum_{k=1}^{\infty}\dfrac{f_k}{w^k}\sum_{n=1}^{\infty}\dfrac{2^na^n}{n!}\log^n|w|dA(w)\right|\\
&\leq M(\epsilon_0)\epsilon^2+\sum_{n,k=1}^{\infty}\dfrac{2^na^n}{n!}|f_k|\int_{\Omega_a}\dfrac{|\log^n|w||}{|w|^{k+2}}dA(w)\xrightarrow{\epsilon\to0}0,
\end{align*}
where the second integral goes to $0$ with $\epsilon$ because $a_\epsilon\xrightarrow{\epsilon\to0}0$ and $0\notin\Cl(\Omega_a)$. The argument in the bounded case is similar. Thus, for each $\epsilon_0>0$ sufficiently small with $0<\epsilon\leq\epsilon_0$, $0<a<a_\epsilon<\min\{a_\ast,\epsilon\}$ and $f\in \A\left(U_{\epsilon_0}\right)$ ($\A_0$ if $U_{\epsilon_0}$ is unbounded),
\begin{align*}
\int_{\Omega_0}\dfrac{f(w)}{|w|^2}dA(w)&=\lim_{\epsilon\to0}\int_{\Omega_a}f(w)|w|^{2(a-1)}dA(w)=\lim_{\epsilon\to0}\oint_{\partial\Omega_a}f(w)h(w)dw=\oint_{\partial\Omega_0}f(w)h(w)dw.
\end{align*}
Now if $w$ is any point in $(U_{\epsilon_0})\IntComp$, $\frac{1}{\xi-w}\in \A\left(U_{\epsilon_0}\right)$ ($\A_0$ if $U_{\epsilon_0}$ is unbounded), so
\begin{align*}
0&=\int_{\Omega_0}\dfrac{dA(\xi)}{|\xi|^2(\xi-w)}-\int_{\Omega_0}\dfrac{dA(\xi)}{|\xi|^2(\xi-w)}=\oint_{\partial\Omega_0}\dfrac{\ln|\xi|^2d\xi}{\xi(\xi-w)}-\oint_{\partial\Omega_0}\dfrac{h(\xi)d\xi}{\xi-w}=-\oint_{\partial\Omega_0}\dfrac{h(\xi)-\frac{\ln|\xi|^2}{\xi}}{\xi-w}d\xi\\
&=\AnalyticIn{h(w)-\frac{\ln|w|^2}{w}}{(\Omega_0)\IntComp}
\end{align*}
But for each $w\in \Omega_0$, there exists $\epsilon_0>0$ for which $w\in (U_{\epsilon_0})\IntComp$, so the above equality holds for all $w\in(\Omega_0)\IntComp$. Moreover, $h(w)-\frac{\ln|w|^2}{w}\in C^0(\partial\Omega_0)$ because $0\notin\partial\Omega_0$, so
$$h(w)-\frac{\ln|w|^2}{w}\dEquals\AnalyticIn{h(w)-\frac{\ln|w|^2}{w}}{\Omega_0}+\AnalyticIn{h(w)-\frac{\ln|w|^2}{w}}{(\Omega_0)\IntComp}=\AnalyticIn{h(w)-\frac{\ln|w|^2}{w}}{\Omega_0}\in\A(\Omega_0).$$
That is, there exists a $\widetilde{G}\in\A(\Omega_0)$ for which $wh(w)-\ln|w|^2\dEquals \widetilde{G}(w)$ on $\partial\Omega_0$. Or, equivalently, there exists $G\in\A(\Omega_0)$ ($\A_0$ if $\Omega_0$ is unbounded) for which Equation \ref{eqn:LogWeightedSFE} is satisfied. Thus $\Omega_0\in\QD_0(h)$ by Theorem \ref{theorem:EquivLQDChars}.
\end{proof}

\subsection{Faber Transform Method for LQDs}\label{subsec:LQDFTMethod}
We would like to obtain Faber transform-based formulae for the Riemann maps associated to LQDs.
As with PQDs, a key step in obtaining such formulae is making an appropriate choice of decomposition of the Riemann map. However, unlike with PQDs, it appears that only the product (inner-outer/Nevanlinna) decomposition and not the sum (Mittag-Leffler) decomposition yields tractable formulae. The inner-outer factorization is, of course, precisely that given in \S\ref{subsec:PQDMultFTMethod}, Equation \ref{eqn:PQDInnerFactors}. As our consideration is restricted to the case in which $0\notin\Omega$, we are only concerned with $\varphi=\varphi_{\rm in}\varphi_{\rm out}$, where $\varphi_{\rm out}:=\frac{\varphi}{\varphi_{\rm in}}$ and
\begin{equation}\label{eqn:LQDInnerFactors}
\begin{alignedat}{2}
    \varphi_{\rm in}(z)&=1,\quad&&\text{when }\infty\notin\Omega,\;0\notin\Omega,\\
    \varphi_{\rm in}(z)&=z,\quad&&\text{when }\infty\in\Omega,\;0\notin\Omega.
\end{alignedat}
\end{equation}

In this case, we obtain the following simple characterization of simply connected PQDs.
\begin{theorem}\label{thm:SCLQDCharacterization}
$\Omega\in\QD_0$ iff $\ln(\varphi_{\rm out})$ extends to a rational function.
\end{theorem}

\begin{proof}[Proof of Theorem \ref{thm:SCLQDCharacterization}]
For the reverse direction, suppose that $\ln(\varphi_{\rm out})$ extends to a rational function. Then, as $|\varphi_{\rm in}|\dEquals1$, we find that $\ln\left|\varphi\right|^{2}=\ln\left|\varphi_{\rm in}\right|^{2}+\ln\left|\varphi_{\rm out}\right|^{2}\dEquals\ln(\varphi_{\rm out})+\ln(\varphi_{\rm out})^{\#}$. Hence,
\begin{align*}
\dfrac{\ln|w|^2}{w}&=\dfrac{\ln\left|\varphi\right|^{2}\circ\psi(w)}{w}\dEquals\dfrac{\left(\ln(\varphi_{\rm out})+\ln(\varphi_{\rm out})^{\#}\right)\circ\psi(w)}{w}=:S_0(w)
\end{align*}
is meromorphic in $\Omega$, and thus a generalized Schwarz function associated to $\Omega$. We conclude by Theorem \ref{theorem:EquivLQDChars} that $\Omega\in\QD_0$.

Now we only need demonstrate that $\Omega\in\QD_0$ implies $\ln(\varphi_{\rm out})$ extends to a rational function. We will begin by considering the bounded case. Note that it is necessary and sufficient to show that
$$r:=\ln(\varphi_{\rm out})^{\#}-\overline{\ln(\varphi(0))}=\ln\left(\frac{\varphi^{\#}}{\varphi_{\rm in}^{\#}}\right)-\overline{\ln(\varphi(0))}$$
extends to a rational function. By construction, $r$ is finite, non-zero, analytic (defined in $\D\IntComp$ if $\Omega$ is bounded and $\D$ if $\Omega$ is unbounded), and extends continuously to the boundary. Thus for each $w\in\partial\Omega$,
\begin{align*}
    (r+r^{\#})\circ\psi(w)&=\ln\left(\dfrac{\varphi\circ\psi(w)}{\varphi_{\rm in}\circ\psi(w)}\right)+\ln\left(\dfrac{\varphi^{\#}\circ\psi(w)}{\varphi_{\rm in}^{\#}\circ\psi(w)}\right)-\ln|\varphi(0)|^2\\
    &\dEquals\ln\left(\dfrac{w}{\varphi_{\rm in}\circ\psi(w)}\cdot\overline{w}\varphi_{\rm in}\circ\psi(w)\right)-\ln|\varphi(0)|^2\\
    &\dEquals w\dfrac{\ln|w|^2}{w}-\ln|\varphi(0)|^2.
\end{align*}
If $\Omega$ is bounded then by Equation \ref{eqn:InteriorFaberTransformProjectionExtension}, $\AnalyticIn{\left(r+r^{\#}\right)\circ\psi(w)}{\Omega\IntComp}=\Phi_{\varphi}\left(\AnalyticIn{r+r^{\#}}{\D\IntComp}\right)$. Moreover by Theorem \ref{theorem:EquivLQDChars}, there exists $G\in\A(\Omega)$ such that $\frac{\ln|w|^2}{w}\dEquals h(w)-G(w)$. Hence,
\begin{align*}
\Phi_{\varphi}\left(\AnalyticIn{r+r^{\#}}{\D\IntComp}\right)(w)&=\AnalyticIn{wh(w)-wG(w)-\ln|\varphi(0)|^2}{\Omega\IntComp}=\AnalyticIn{wh(w)}{\Omega\IntComp}
\end{align*}
Note that, as $\varphi\in\A(\D)$ ($0\notin\varphi(\D)$ by assumption) and $r=\ln(\varphi_{\rm out})^{\#}-\overline{\ln(\varphi(0))}=\left(\ln(\varphi)-\ln(\varphi(0))\right)^{\#}$, we have that $r^{\#}\in\A(\D)$ and $r\in\A_0(\D\IntComp)$. Therefore, $\AnalyticIn{r+r^{\#}}{\D\IntComp}=r$. Applying the inverse Faber transform to both sides yields
$r=\Phi_{\varphi}^{-1}\left(\AnalyticIn{wh(w)}{\Omega\IntComp}\right)$. The argument of the inverse Faber transform here is in $\M(\Omega)$, so its transform is rational, and we may conclude that $r$ is rational.

On the other hand, if $\Omega$ is unbounded and we set $r=\ln(\varphi_{\rm out})^{\#}-\ln(c)$ then by reasoning analogous to the bounded case, we find that
\begin{align*}
    \Phi_{\varphi}\left(r\right)(w)&=\AnalyticIn{wh(w)-wG(w)-\ln(c)^2}{\Omega\IntComp}=\AnalyticIn{wh(w)}{\Omega\IntComp}-\left(\ln(c^2)-\Res{\infty}G\right).
\end{align*}
Applying the inverse Faber transform and the fact that $r^{\#}\in\A_0(\D\IntComp)$, so $r(0)=0$, we obtain
\begin{align*}
r&=\Phi_{\varphi}^{-1}\left(wh(w)\right)-\Phi_{\varphi}^{-1}\left(wh(w)\right)(0)
\end{align*}
The argument of the inverse Faber transform here is in $\M(\Omega)$, so its transform is rational, and we conclude that $r$ is rational.
\end{proof}

In the course of the above proof, we obtained an implicit solution to the inverse problem for LQDs. By inverting some of the above relations, one can also obtain a solution to the direct problem. The following theorem provides a summary of these solutions to the inverse and direct problems.
\begin{theorem}\label{theorem:logweightedAQDRiemannMaps}
Let $\Omega\subset\Ch$ be a simply connected domain with $0\notin\Cl(\Omega)$. Then $\Omega\in\QD_0(h)$ for some rational $h$ if and only if either
\begin{enumerate}
    \item $\Omega$ is bounded and $\exists r\in\Rat_0(\D)$ such that $\varphi(z)=\varphi(0)e^{r^{\#}(z)}$ is a Riemann map associated to $\Omega$. In this case,
    \begin{equation}
    r=\Phi_{\varphi}^{-1}\left(\AnalyticIn{wh(w)}{\Omega\IntComp}\right).
    \end{equation}
    \item $\Omega$ is unbounded and $\exists r\in\Rat(\D\IntComp)$ and $c>0$ (the conformal radius of $\Omega$) such that $\varphi(z)=cze^{r^{\#}(z)}$ is a Riemann map associated to $\Omega$. In this case,
    \begin{equation}
    r=\Phi_{\varphi}^{-1}\left(wh(w)\right)-\Phi_{\varphi}^{-1}\left(wh(w)\right)(0).
    \end{equation}
\end{enumerate}
Furthermore in either case,
\begin{equation}
    h(w)=\dfrac{\Phi_{\varphi}\left(r\right)(w)-\Phi_{\varphi}\left(r\right)(0)}{w}
\end{equation}
\end{theorem}

We begin with a discussion of one point LQDs, along the lines of the discussion for PQDs in \S \ref{subsec:OnePtPQDs}, after which we will briefly cover the case of LQDs with a monomial quadrature function.
\subsection{One Point LQDs}
Let $\Omega$ be a simply connected LQD with quadrature function $\frac{\alpha}{w-w_0}$ for some $\alpha\in\C\setminus\{0\}$\footnote{The $\alpha=0$ (null) case is covered by Theorem \ref{thm:NullLQDClass}.} and $w_0\in\C\setminus\{0\}$; that is, $\Omega\in\QD_0\left(\frac{\alpha}{w-w_0}\right)$. Moreover by a change of variables, $w\mapsto \frac{w_0}{|w_0|}$, we may assume wlog that $w_0>0$.

It is well-known that the only one point bounded quadrature domains are disks. The solution space isn't so simple for bounded LQDs. Suppose that $\Omega\in\QD_0\left(\frac{\alpha}{w-w_0}\right)$ with $\alpha>0$ and $w_0>0$ ($w_0>0$ wlog by rotational symmetry) is bounded and simply connected with Riemann map $\varphi:\D\rightarrow\Omega$, taking $0$ to $w_0$. Then by Theorem \ref{theorem:logweightedAQDRiemannMaps}, $\varphi(z)=w_0e^{r^{\#}(z)}$, where
$$r(z)=\Phi_{\varphi}^{-1}\left(\AnalyticIn{\frac{\alpha w}{w-w_0}}{\Omega\IntComp}\right)(z)=\alpha\Phi_{\varphi}^{-1}\left(\frac{w}{w-w_0}\right)(z)=\alpha w_0\frac{\psi'(w_0)}{z-\psi(w_0)}=z^{-1}\frac{\alpha w_0}{\varphi'(0)}.$$
Applying the $\varphi'(0)$ relation tells us that $\varphi'(0)=\sqrt{\alpha}w_0$ so $\varphi(z)=w_0e^{z\sqrt{\alpha}}$. It is a standard result that the map $z\mapsto e^{\nu z}$ is univalent in $\D$ iff $0<\nu<\pi$. Thus $\varphi$ is univalent iff $0<\alpha<\pi^2$. In particular, such a simply connected domain $\Omega$ exists iff $0<\alpha<\pi^2$. (see Figure \ref{fig:LOGWQDOnePt}, left) In summary,

\begin{theorem}\label{thm:LOGWQDBoundedOnePt}
Take $\alpha,w_0>0$. There exists a bounded simply connected domain $\Omega$ for which $\Omega\in\QD_0\left(\frac{\alpha}{w-w_0}\right)$ if and only if $0<\alpha<\pi^2$, in which case $\Omega=\varphi(\D)$, where $\varphi$ is univalent and given by $\varphi(z)=w_0e^{z\sqrt{\alpha}}$.
\end{theorem}

The situation for unbounded one point LQDs isn't quite so straightforward. While we can still obtain an explicit representation of the Riemann map, we could solve for the undetermined constants only numerically.

\begin{example}\label{ex:LOGWQDUnboundedNoZeroOnePt}
Suppose that $0\notin\Omega\in\QD_0\left(\frac{\alpha}{w-w_0}\right)$ ($\alpha\in\C$, $w_0>0$) is unbounded and simply connected with Riemann map $\varphi:\D\IntComp\rightarrow\Omega$, taking $z_0\in\D\IntComp$ to $w_0$. Then by the theorem, $\varphi(z)=cze^{r^{\#}(z)}$, where
\begin{align*}
r(z)&=\alpha\Phi_{\varphi}^{-1}\left(\frac{w}{w-w_0}\right)(z)-\alpha\Phi_{\varphi}^{-1}\left(\frac{w}{w-w_0}\right)(0)=\alpha w_0\Phi_{\varphi}^{-1}\left(\frac{1}{w-w_0}\right)(z)-\alpha w_0\Phi_{\varphi}^{-1}\left(\frac{1}{w-w_0}\right)(0)\\
&=\frac{\alpha w_0\psi'(w_0)}{z-\psi(w_0)}+\frac{\alpha w_0\psi'(w_0)}{\psi(w_0)}=\dfrac{\alpha w_0}{\varphi'(z_0)}\left(\frac{1}{z-z_0}+\dfrac{1}{z_0}\right).
\end{align*}
Combining this with the relations for $\varphi(z_0)$ and $\varphi'(z_0)$ one can numerically solve for the undetermined constants $z_0$ and $\varphi'(z_0)$. (see Figure \ref{fig:LOGWQDOnePt}, right)
\end{example}

\begin{figure}[ht]
  \centering
\includegraphics[height=0.4\textwidth,width=.4\textwidth,valign=c
]{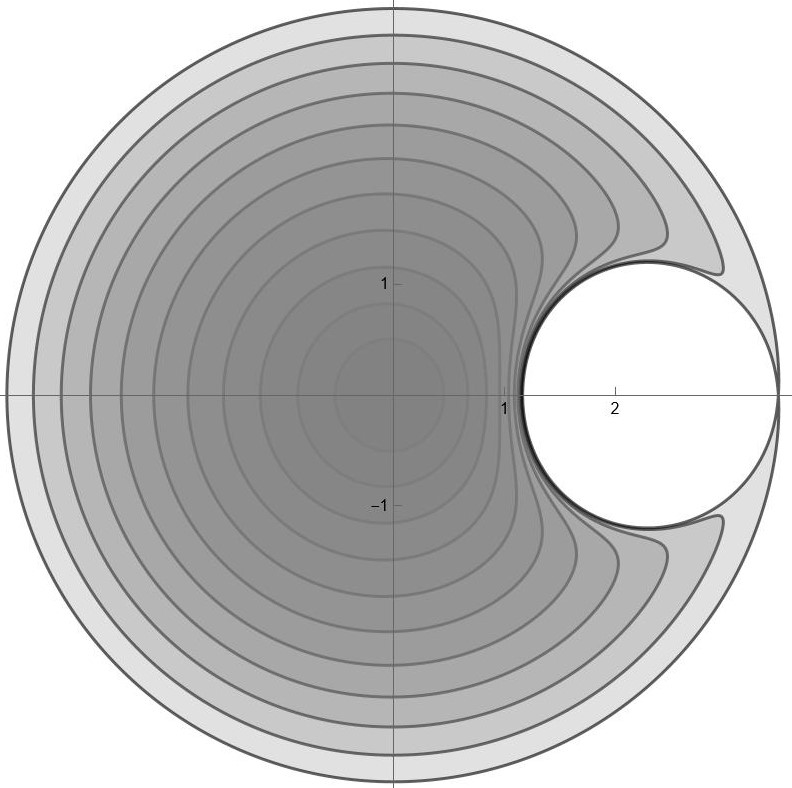}\tab\tab\includegraphics[height=0.4\textwidth,width=.45\textwidth,valign=c]{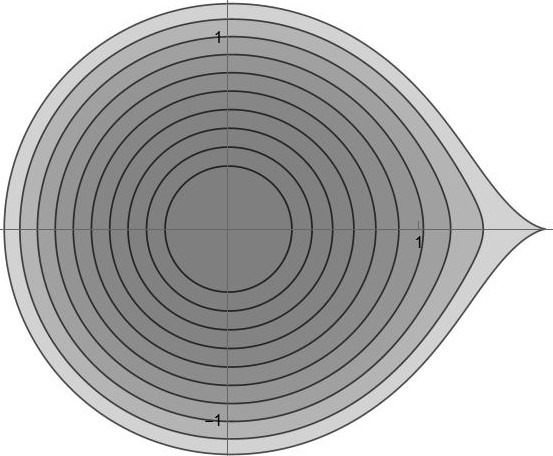}\vspace{.4em}
\caption{Family of unbounded one point LQDs with quadrature function $\frac{\alpha}{w-w_0}$ for $w_0=2$, $\alpha=.3,-.1$ (Example \ref{ex:LOGWQDUnboundedNoZeroOnePt}).}\label{fig:LOGWQDOnePt}\vspace{1.5em}
\end{figure}

\begin{figure}[ht]
  \centering
\includegraphics[height=0.4\textwidth,width=.4\textwidth,valign=c
]{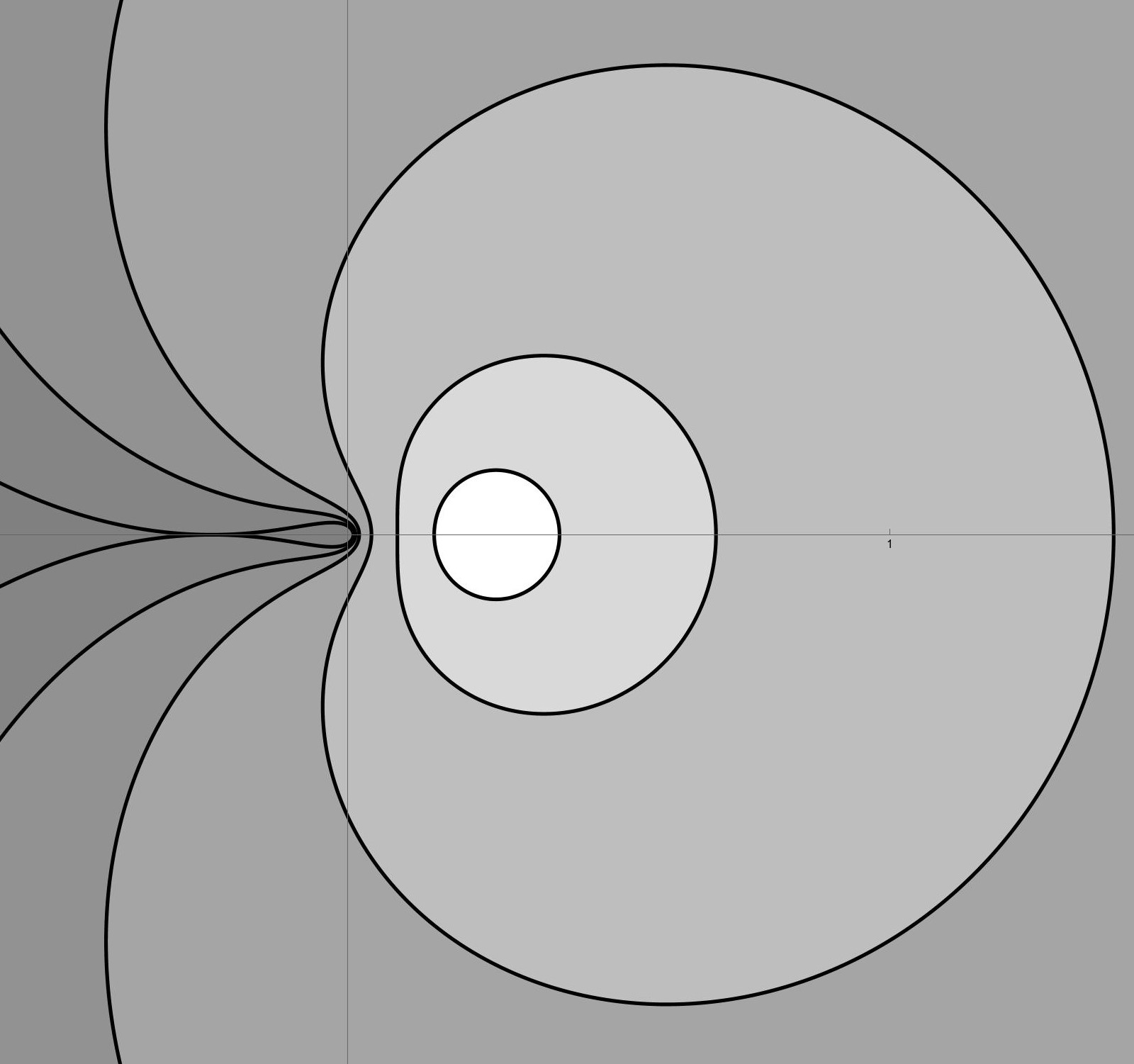}\tab\tab\includegraphics[height=0.4\linewidth,width=.4\linewidth,valign=c]{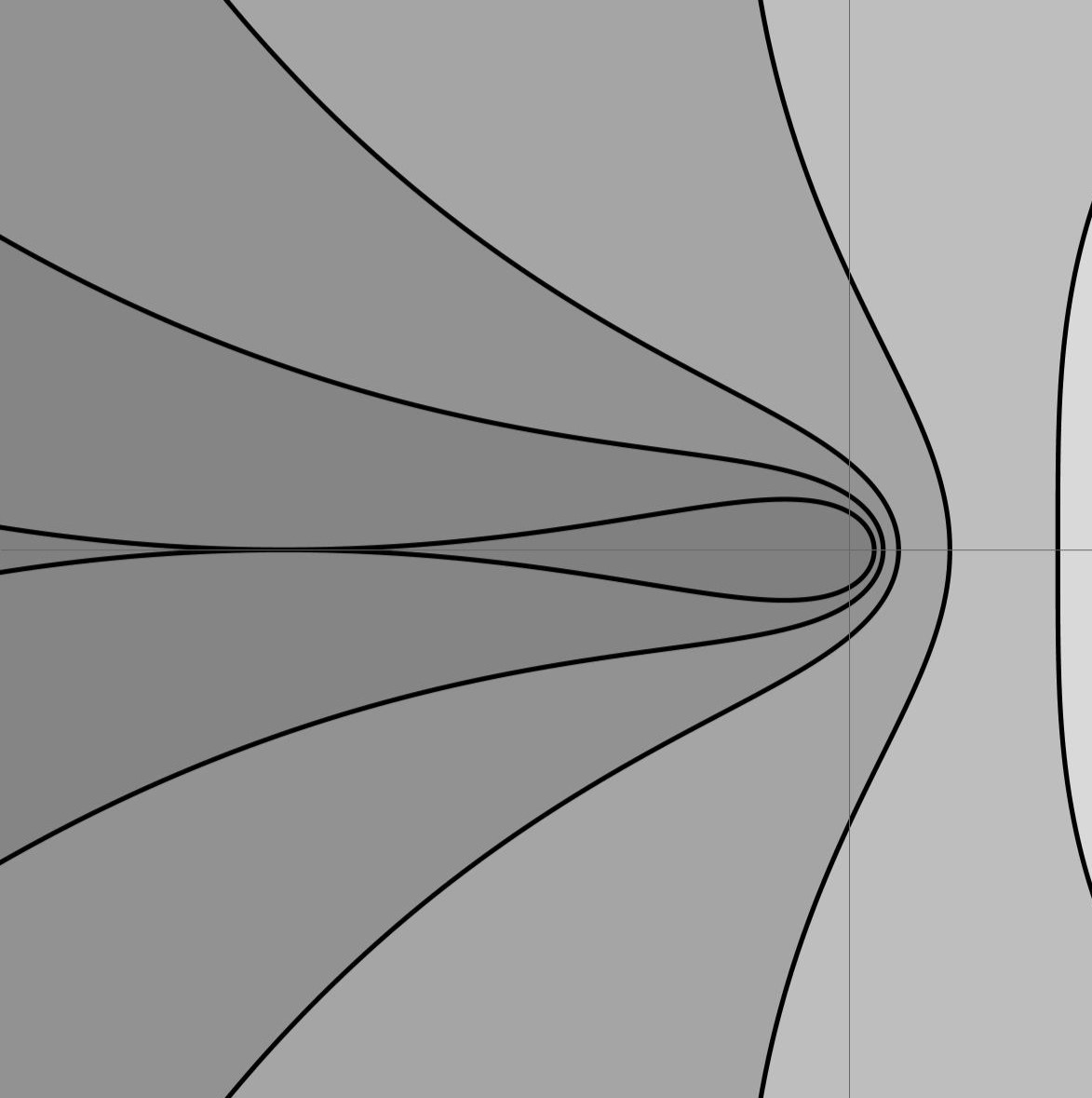}\vspace{.4em}
\caption{Family of bounded one point LQDs with quadrature function $\frac{\alpha}{w-w_0}$ for $w_0=.25$ and $2\leq \alpha\leq\pi^2$ (left) and zoom (right). (Theorem \ref{thm:LOGWQDBoundedOnePt})}\label{fig:LOGWQDBoundedNoZeroOnePtZoom}\vspace{1.5em}
\end{figure}

\subsection{Monomial LQDs}\label{subsec:LOGWQDBasicMonomial}
First we consider the case of a constant quadrature function: Suppose that $0\notin\Omega\in\QD_0\left(\alpha\right)$ ($\alpha>0$ wlog by change of variables) is unbounded and simply connected with Riemann map $\varphi:\D\IntComp\rightarrow\Omega$. Then by the theorem, $\varphi(z)=cze^{r^{\#}(z)}$, where
$$r(z)=\alpha \Phi_{\varphi}^{-1}\left(w\right)(z)-\alpha \Phi_{\varphi}^{-1}\left(w\right)(0)=c\alpha z+\alpha f_0-(c\alpha \cdot0+\alpha f_0)=c\alpha z.$$
Figure \ref{fig:LOGWQDUnboundedNoZeroMonomialk} (left) depicts one such family of these.

We now consider the more general case of $0\notin\Omega\in\QD_0\left(\alpha kw^{k-1}\right)$, $k\in\Z_{+}$ is unbounded and simply connected with Riemann map $\varphi:\D\IntComp\rightarrow\Omega$. Then by the theorem, $\varphi(z)=cze^{r^{\#}(z)}$, where
$$r(z)=\alpha k\Phi_{\varphi}^{-1}\left(w^{k}\right)(z)-\ln(c^2)=\alpha kW_{k}(z)-\ln(c^2)=\alpha kc^kz^{k}+O(z^{k-1}),$$
where $W_k$ is the $k$th inverse Faber polynomial. Then if we assume $\Z_{k}-$rotational symmetry, $\varphi(ze^{\frac{2\pi i}{k}})=e^{\frac{2\pi i}{k}}\varphi(z)$, and we find that the lower order terms in the exponent must $=0$, and
$$\varphi(z)=cze^{\overline{\alpha}kc^kz^{-k}}.$$
This function is univalent in $\D\IntComp$ whenever $|\alpha k^2c^k|<1$ (or $0<c<|\alpha k^2|^{-\frac{1}{k}}$). (Figure \ref{fig:LOGWQDUnboundedNoZeroMonomialk})

%\section{Conclusion}\label{sec:Conclusion}

\begin{comment}

Future directions
\begin{itemize}
    \item More rigorous study of ``corner phenomenon'' when $0\in\Omega$.
    \item Allow for infinitely many poles, or tempered essential singularities.
    \item Use boundary regularity of Abelian LQDs, along with relationship between Abelian LQDs and PQDs to establish the regularity of PQDs for all $a>0$ (can probably establish other properties this way as well).
    \item Can essentially do same paper, but for multiply connected domains, using Crowdy's result.
\end{itemize}

Major Theorems
\begin{itemize}
    \item ``$\Omega\in\QD_a$ iff $\varphi^a_{\rm out}$ extends to a rational function.'' (\ref{thm:SCPQDCharacterization})
    \item Weighted potentials associated to PQDs (\ref{lemma:GenQuadComplementDroplet}, \ref{lemma:GenComplementDropletQuad})
    \item What does the Blaschke factor turning up in the Riemann map mean for the corresponding algebraic equation?
\end{itemize}

\end{comment}

\section{Preview: Transcendental Dynamics of The Schwarz Reflection}
Let $\Omega$ be a quadrature domain with Schwarz function $S$ and associated droplet $K=\Omega^c$. By removing the finitely many singular points of $\partial\Omega$ from $K$, we obtain the {\it fundamental tile} $T$. We can then consider an anti-holomorphic dynamical system under the action of the {\it Schwarz reflection} $\sigma:=\overline{S}:\Cl(\Omega)\rightarrow\Ch$. This system admits a partition into a pair of sets invariant under the action of $\sigma$:
\begin{enumerate}
    \item The {\it tiling set} $T^{\infty}$, composed of the set of points which eventually escape to the fundamental tile under the action of $\sigma$;
    \item the {\it non-escaping set} $A(\infty)$, those points which do not escape to the fundamental tile.
\end{enumerate}
The tiling set has its name because it can be alternatively defined as the union of the fundamental tile $T$ and each of its preimages under the iteration of $\sigma$. In certain cases, the dynamics of $\sigma$ on $T^{\infty}$ resemble that of a reflection group and the dynamics on $A(\infty)$ resemble that of an anti-polynomial on its filled Julia set. 

Lee et al \cite{LEE_2018} demonstrate that when, $T$ is the deltoid, this phenomenon manifests with $\sigma$ as the ``mating'' of the anti-polynomial $f_0(z)=\overline{z}^2$ ($f_0:\D^c\rightarrow\D^c$) and the reflection map $\rho:\Cl(\D)\setminus\mathring{\Pi}\rightarrow\Cl(\D)$ of the ideal triangle group ($\Pi$ is the hyperbolic triangle in $\D$). This means that the dynamics of $f_0$ and $\rho$ can be glued together along $\partial\D$, resulting in a topological map $\eta$ on $\Ch$, which is conformally conjugate to $\sigma$ when $\Ch$ is equipped with the appropriate conformal structure (\cite{LEE_2018}, Theorem 1.1).\\

Recall the exterior ``teardrop'' of Figure \ref{fig:LOGWQDUnboundedNoZeroMonomialk} (left), $\Omega\in\QD_0(1)$, which is the image of $\D\IntComp$ under $\varphi(z)=ze^{z^{-1}}$. In this case, the associated transcendental Schwarz reflection $\sigma:\Omega\rightarrow\Ch$ is given by
\begin{equation}
    \sigma(w)=\overline{S(w)}=\overline{\varphi^{\#}\circ\psi(w)}=\overline{w}^{-1}e^{-W\left(-\overline{w}^{-1}\right)-W\left(-\overline{w}^{-1}\right)^{-1}},
\end{equation}
where $W$ is the principal branch of the Lambert W function. We note the striking similarity between the non-escaping set of $\sigma$, and the filled Julia set of the anti-holomorphic exponential map $w\mapsto e^{\overline{w}-1}$ (Figure \ref{fig:AntiholomorphicExpSchwarzComparison}).

\begin{figure}[ht]
  \centering
\includegraphics[height=0.47\textwidth,width=.47\textwidth,valign=c
]{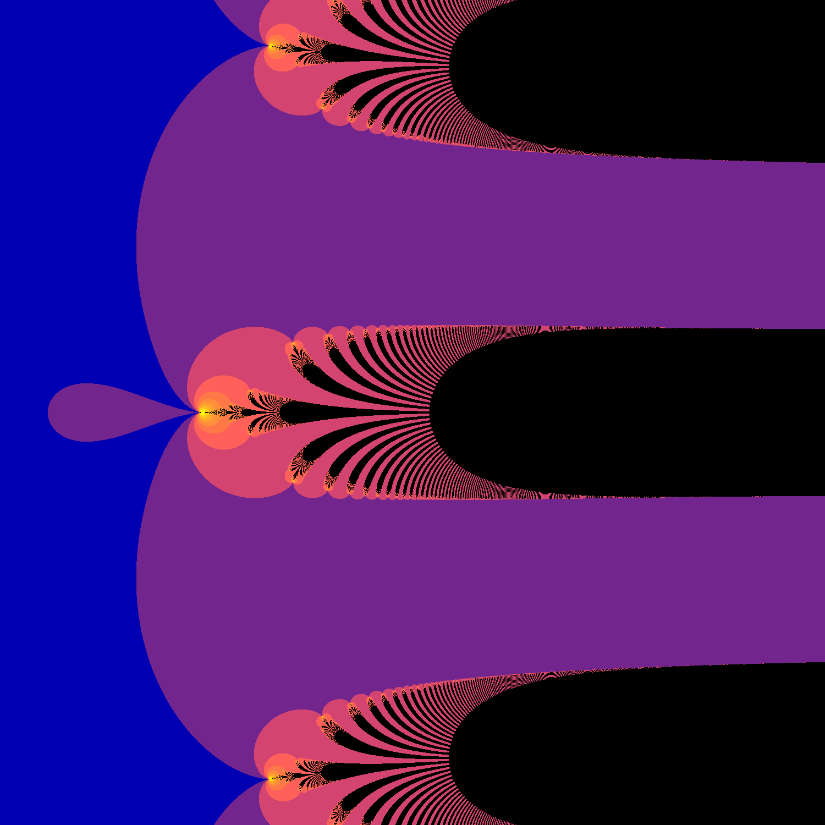}\tab\tab\includegraphics[height=0.47\linewidth,width=.47\linewidth,valign=c]{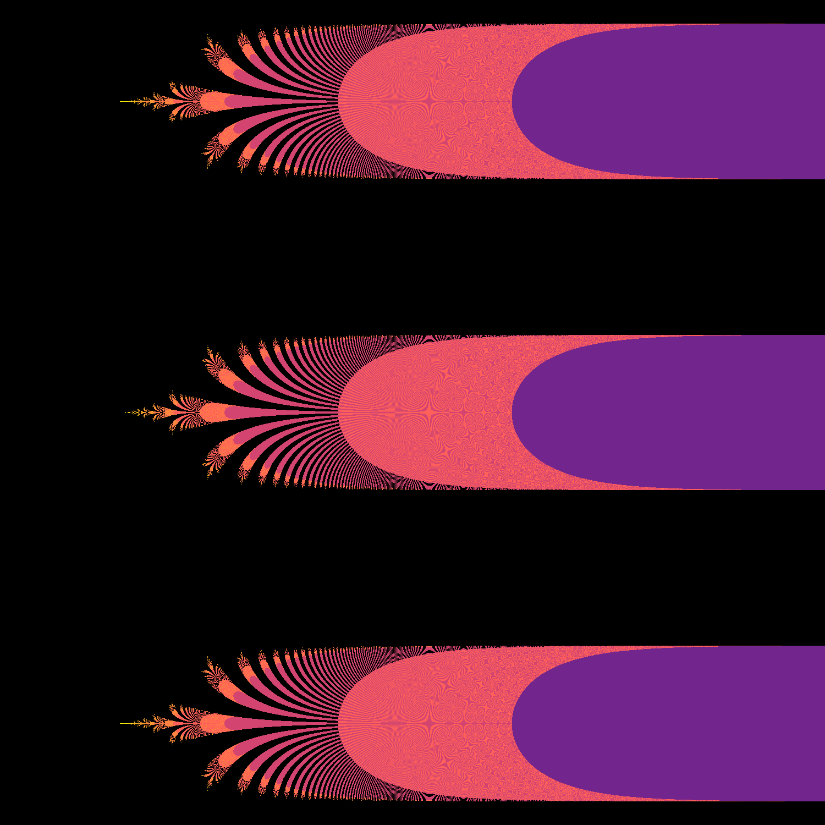}\vspace{.4em}
\caption{Tiling set of the Schwarz reflection $\sigma$ associated to $\Omega\in\QD_0(1)$ (left) and filled Julia set of the anti-holomorphic exponential map $w\mapsto e^{\overline{w}-1}$ (right).}\label{fig:AntiholomorphicExpSchwarzComparison}\vspace{1.5em}
\end{figure}
In an upcoming manuscript, the authors will conduct an analysis along the lines of Lee et al \cite{LEE_2018} in this transcendental setting.

\printbibliography

\end{document}